\documentclass{amsart}

\usepackage{amssymb}
\usepackage[all]{xy}

\usepackage{color}

\usepackage{graphicx}
\usepackage{hyperref}

\hypersetup{
    colorlinks,
    citecolor=blue,
    filecolor=blue,
    linkcolor=blue,
    urlcolor=blue
}
\hypersetup{linktocpage}

\newfont\eul{eusm10}
\newfont\euls{eusm8}

\newcommand{\bea}{\begin{eqnarray}}
\newcommand{\eea}{\end{eqnarray}}
\newcommand{\bean}{\begin{eqnarray*}}
\newcommand{\eean}{\end{eqnarray*}}

\newcommand{\mcA}{\mathcal{A}}
\newcommand{\mcB}{\mathcal{B}}

\newcommand{\mcG}{\mathcal{G}}
\newcommand{\mcI}{\mathcal{I}}
\newcommand{\mcJ}{\mathcal{J}}

\newcommand{\mcL}{\mathcal{L}}

\newcommand{\mcS}{\mathcal{S}}

\newcommand{\mcX}{\mathcal{X}}

\newcommand{\mcZ}{\mathcal{Z}}

\newcommand{\mbbP}{\mathbb{P}}
\newcommand{\mbbQ}{\mathbb{Q}}

\newcommand{\mbfL}{\mathbf{L}}

\newcommand{\mbfV}{\mathbf{V}}

\newcommand{\Kc}{\mathbf{K}^c}

\newcommand{\Le}{\mathbf{L}[E]}

\newcommand{\On}{\mathbf{On}}

\newcommand{\cff}{\mathsf{cof}}

\newcommand{\crp}{\mathsf{cr}}
\newcommand{\csp}{\mathbb{CS}}
\newcommand{\cut}{\mathop{||}}

\newcommand{\dom}{\mathrm{dom}}
\newcommand{\gch}{\mathsf{GCH}}

\newcommand{\idm}{\mathsf{id}}

\newcommand{\lht}{\mathsf{lh}}

\newcommand{\nrp}{{\mathbb{NR}}}

\newcommand{\ptm}{\mbox{\eul P}}

\newcommand{\rng}{\mathrm{rng}}
\newcommand{\rst}{\restriction}

\newcommand{\ult}{\mathsf{Ult}}

\newcommand{\zfc}{{\mathsf{ZFC}}}

\newcommand{\card}{\mathsf{card}}

\newcommand{\dfeq}{\mathop{\stackrel{\mathrm{def}}{=}}}

\newcommand{\mlex}{<_{\mathrm{mlex}}}

\newcommand{\forces}{\Vdash}

\newcommand{\symdiff}{\mathop{\triangle}}

\newtheorem{lemma}{Lemma}[section]
\newtheorem{proposition}[lemma]{Proposition}
\newtheorem{theorem}[lemma]{Theorem}
\newtheorem{corollary}[lemma]{Corollary}

\newtheorem{definition}[lemma]{Definition}
\newtheorem{remark}[lemma]{Remark}

\newtheorem{claim}[lemma]{Claim}

\def\qed{\nobreak\hfill\penalty250 \hbox{}\nobreak\hfill\qedbox\medskip}

\newtheorem*{theorem*}{Theorem}
\newtheorem*{corollary*}{Corollary}

\newtheorem{open}{Open Problem}
\newcommand{\cat}{^\frown}
\newcommand{\rest}{\ensuremath{\upharpoonright}}

\newcommand{\la}{\langle}
\newcommand{\ra}{\rangle}

\newcommand{\nn}{{\mathbb N}}

\newcommand{\poP}{\ensuremath{\mathbb P}}
\newcommand{\poQ}{\ensuremath{\mathbb Q}}

\newcommand{\mct}{{\ensuremath{\mathcal T}}}

\newcommand{\mci}{\ensuremath{\mathcal I}}

\newcommand{\mca}{\ensuremath{\mathcal A}}

\newcommand{\mcb}{\mathcal B}
\newcommand{\mcs}{{\mathcal S}}

\newcommand{\mcj}{\mathcal J}
\newcommand{\mcu}{\mathcal U}

\newcommand{\bfni}[1]{\noindent {{\bf{#1}}}}

\renewcommand{\qed}{{\nopagebreak \hfill $\dashv$ 
 \par\bigskip}}

\newtheorem*{remark*}{Remark}

\title{{Games with Filters I}}

\author{Matthew Foreman}
\address{Department of Mathematics, University of California at Irvine,
Irvine, CA 92697-3875, USA}
\email{mforeman@math.uci.edu}
\thanks{Matthew Foreman was supported in part by NSF grant DMS-2100367}

\author{Menachem Magidor}
\address{Institute of Mathematics, Hebrew university of Jerusalem, Jerusalem
91904, Israel} 
\email{mensara@savion.huji.ac.il}
\thanks{Menachem Magidor was supported by ISF grant 684-17}

\author{Martin Zeman}
\address{Department of Mathematics, University of California at Irvine, 
Irvine, CA 92697-3875, USA}
\email{mzeman@math.uci.edu}

\begin{document}

\begin{abstract}
This paper has two parts.  The first is concerned with a variant of a family
of games introduced by Holy and Schlicht, that we call \emph{Welch
games}. Player II having a winning strategy in the Welch game of length
$\omega$ on $\kappa$ is equivalent to weak compactness. Winning the game of
length $2^\kappa$ is equivalent to $\kappa$ being measurable. We show that for
games of intermediate length $\gamma$, II winning implies the existence of
precipitous ideals with $\gamma$-closed, $\gamma$-dense trees.   

The second  part shows the first is not vacuous.  For each  $\gamma$ between
$\omega$ and $\kappa^+$, it gives a model where   II wins the games of length
$\gamma$, but not $\gamma^+$.   The technique also gives models where for all
$\omega_1< \gamma\le\kappa$ there are $\kappa$-complete, normal,
$\kappa^+$-distributive ideals having dense sets that are $\gamma$-closed,
but not $\gamma^+$-closed. 
\end{abstract}

\maketitle

\begin{center}



\date={}
\end{center}

\bigskip

\section{Introduction}\label{s.intro}

Motivated by ideas of generalizing properties of the first inaccessible
cardinal~$\omega$, Tarski \cite{T62}  came up with the idea of {considering}
uncountable cardinals  $\kappa$ such that $\mcL_{\kappa\kappa}$-compactness
holds
for languages of size $\kappa$.  This became the definition of a \emph{weakly
compact} cardinal.  Hanf \cite{H63}, showed that weakly compact cardinals are
Mahlo. Work of Keisler \cite{62a} and Keisler and Tarski \cite{K-T64} showed: 

\begin{theorem*}
Let $\kappa$ be an uncountable inaccessible cardinal. Then the following are
equivalent to weak compactness: 
	\begin{enumerate}
	\item Whenever $R\subseteq V_\kappa$ there is a transitive set $X$ and
	$S\subseteq X$ such that 
	\[
	\la V_\kappa,\in, R\ra \prec \la X, \in, S\ra.
	\]
	\item \label{extension} If $\mcb\subseteq \ptm(\kappa)$ is a
	$\kappa$-complete Boolean subalgebra with $|\mcb|=\kappa$ and $F$ is
	a $\kappa$-complete filter on $\mcb$, then $F$ can be extended to a
	$\kappa$-complete ultrafilter on $\mcb$. 
	\end{enumerate}
\end{theorem*}
 
Items 1 and 2 are clearly implied by their analogues for measurable cardinals:
 
\begin{itemize}
\item[($1'$)] There is an elementary embedding of $V$ into a transitive
class $M$ that has critical point $\kappa$. 
\item[($2'$)] There is a non-atomic, $\kappa$-complete ultrafilter on
$\ptm(\kappa)$.  
\end{itemize}

\smallskip

\bfni{Holy-Schlicht Games.} This paper concerns several of a genre of
games originating in the paper \cite{HolySh} of Holy and Schlicht, which {were modified and further} explored
by Nielsen and Welch \cite{NW}. The following small variant of the
{Holy-Schlicht-Nielsen-Welch games} was suggested to us by Welch.  

{
Players I and II alternate moves:
\begin{center}
\begin{tabular}{ c||c|c|c|c|c|c } 
 I & $\mca_0$& $\mca_1$ &\dots &$\mca_\alpha$&$\mca_{\alpha+1}$&\dots \\
 \hline
II & $U_0$ & $U_1$& \dots &$U_\alpha$&$U_{\alpha+1}$& \dots \\ 
\end{tabular}
\end{center}
The game {proceeds} for some length $\ell\le \gamma$ {determined by the play}. 
The sequence  $\la \mca_\delta:0\le\delta<\ell\le\gamma\ra$ is an increasing sequence of
$\kappa$-complete subalgebras of $\ptm(\kappa)$ of cardinality $\kappa$ and
$\la U_\delta:0\le\delta<\ell\ra$ is sequence of uniform $\kappa$-complete
filters,  each $U_\alpha$ is a uniform ultrafilter on $\mca_\alpha$ and
$\alpha<\alpha'$ implies that $U_\alpha\subseteq U_{\alpha'}$. We assume
without loss of generality that $\mca_0$ contains all singletons. Player I
goes first at limit stages. The game continues until either Player II can't
play or the play has length~$\gamma$. {If Player II can't play, the game ends and $\ell$ is the length of the sequence already played.}\footnote{We could omit ``uniform" and
simply require $\mca_0$ to include the co-$<\!\!\kappa$ subsets of $\kappa$
and $U_0$ to extend the co-$<\!\!\kappa$-filter.  As noted in section \ref{weak compactness}, if $\kappa$ is inaccessible, then Player I always has a legal play in the Welch game.}
We denote this game by $\mcG^W_\gamma$. }
\medskip
 
\bfni{The winning condition.} Player II wins if the game continues through all
stages below $\gamma$.
 
\bigskip
There are two extreme cases: $\gamma\le \omega$ and
$\gamma=2^{\kappa}$.  Using item (\ref{extension}) of the characterization of
weakly compact cardinals, one sees easily that if $\kappa$ is weakly compact
then II wins the game of length $\omega$.  

The situation with the converse is slightly complicated. If $\kappa$ is inaccessible and Player II can win the Welch game of length 2, then 
$\kappa$ is weakly compact.  If $\kappa$ is not inaccessible, then either Player I does not have an opening move, or Player II loses. This {follows from work in \cite{flipping}, though stated in a different way there. For completeness it is} proved in Section \ref{weak compactness}.

At the other {extreme} if $\kappa$ is
measurable one can fix in advance a $\kappa$-complete uniform ultrafilter
$\mcu$ on $\ptm(\kappa)$ and at stage $\alpha$ play $U_\alpha=U\cap
\mca_\alpha$.  {The converse is also immediate:  if the second player has a winning strategy in the game of length 
$2^\kappa$, and the first player plays a sequence of algebras with $\bigcup_{\alpha<2^\kappa} \mca_\alpha=\ptm(\kappa)$, 
then the union of the $U_\alpha$'s in Player II's responses gives a $\kappa$-complete ultrafilter on $\kappa$.}

In \cite{NW}, Nielsen-Welch proved that Player II having a winning
strategy in  the game of length $\omega+1$ implies that there is an inner
model with a measurable cardinal. This motivated the following:
\medskip

\bfni{Welch's Question.} Welch asked whether Player~II having a winning
strategy in the game of length $\omega_1$ implies the existence of a
non-principal precipitous ideal.

\medskip

For the readers' convenience we recall the definition of \emph{precipitousness}. An ideal $\mci$ on a set $X$ is
\emph{precipitous} if for all generic $G\subseteq \ptm(X)/\mci$ the generic
ultrapower $V^X/G$ is well-founded. See \cite{JP} or \cite{Fo} for details of the definition. 

The main result of this paper is:
\bigskip

\bfni{Theorem} If  $\kappa$ is inaccessible, $2^\kappa=\kappa^+$ and Player~II
can win the game of length $\omega+1$ then there is a uniform normal
precipitous ideal on $\kappa$.  
\smallskip

\noindent In section \ref{weak compactness}, we show that even the Welch  game of length one is not meaningful if $\kappa$ is not inaccessible.

\medskip
  
We note here that for $\gamma$ a limit, there is an intermediate property
between ``Player~II wins the game of length $\gamma$" and ``Player~II wins the
game of length $\gamma+1$".  It is the  game ${\mathcal G}_\gamma^*$  of length
$\gamma$ that is played the same way as the original Welch game $\mcG^W_\gamma$,
but with a 
different winning condition: For Player~II to win, there must be an extension of
$\bigcup_{\alpha<\gamma}U_\alpha$ to a uniform $\kappa$-complete ultrafilter
on the $\kappa$-complete subalgebra of $\ptm(\kappa)$ generated by
$\bigcup_{\alpha<\gamma}\mca_\alpha$. 
\medskip

\bfni{Precipitous ideals.}
We are fortunate Welch's question leads to a number of more refined results about 
the structure of the quotients of the Boolean algebras $P(\kappa)/\mci$.  We begin by
discussing a strong hypothesis: 
\begin{quotation} \noindent A $\kappa$-complete, uniform ideal 
 $\mci$ on $\kappa$ such that the Boolean algebra $\ptm(\kappa)/\mci$   has the
 $\kappa^+$-chain condition is called a \emph{saturated ideal}. 
\end{quotation}
 
It follows from results of Solovay in  \cite{rvm} that if $\mci$ is a
saturated ideal on $\kappa$ then $\mci$ is precipitous.  Thus to show that a property $P$ implies that
there is a non-principal precipitous ideal on $\kappa$ it suffices to consider only
the case where $\kappa$ does not carry a saturated ideal. 
\bigskip

The most direct answer to Welch's question is given by the following theorem:
\begin{theorem}\label{t1}  Assume that 
$2^\kappa=\kappa^+$ and that $\kappa$ does
not carry a saturated ideal. If Player II has a winning strategy in the game
$\mathcal G^*_\omega$, then there is a uniform normal precipitous ideal on
$\kappa$.  
\end{theorem}

We recall that a normal uniform ideal on $\kappa$ is $\kappa$-complete. 
As a corollary we obtain:
 
\begin{corollary*}
Under the assumptions of Theorem \ref{t1}, if
 Player~II has a winning strategy in either $\mcG^*_\omega$ or $\mcG^W_\gamma$ for any
$\gamma\ge \omega+1$, then there is a uniform normal precipitous ideal on
$\kappa$.  
\end{corollary*}

While this is the result with the simplest statement, its proof gives a lot of structural information about the quotient algebra
$\ptm(\kappa)/\mci$.  {We prove the following theorem in section \ref{s.star}:}
 
\begin{theorem}\label{t2}  
Assume that 
$2^\kappa=\kappa^+$ and that $\kappa$
does not carry a saturated ideal.  Let $\gamma>\omega$ be a regular cardinal
less than $\kappa^+$. If Player~II has a winning strategy in the Welch game of
length $\gamma$, then there is a uniform normal ideal $\mci$ on $\kappa$
and a set $D\subseteq\mcI^+$ such that: 
\begin{enumerate}
\item $(D, \subseteq_\mci)$ is a downward growing tree of height $\gamma$,
\item $D$ is closed under $\subseteq_\mci$-decreasing sequences of
	length less than $\gamma$, 
\item $D$ is dense in $\ptm(\kappa)/\mci$.    
\end{enumerate}
In fact, it is possible to construct such a {dense} set $D$ where (1) and (2) above
hold with the almost containment $\subseteq^*$ in place of $\subseteq_{\mcI}$. 
\end{theorem}

  	\begin{definition}
	Let $\mci$ be a $\kappa$-complete ideal on $\ptm(\kappa)$ and $\gamma>\omega$ be a regular cardinal. Then $\mci$ is $\gamma$-\emph{densely treed} if there is a set $D\subseteq\mci^+$ such that
	\begin{enumerate}
	\item  $(D, \subseteq_\mci)$ is a downward growing tree,
	\item $D$ is closed under $\subseteq_\mci$-decreasing sequences of
	length less than $\gamma$, 
\item $D$ is dense in $\ptm(\kappa)/\mci$.
	\end{enumerate}
Note that this is weaker than the conclusions of Theorem \ref{t2}.	

We will abuse notation slightly and say ``$D$ is dense in $\mci^+$" to mean
	that $D$ is a dense subset of $\ptm(\kappa)/\mci$. 
	
We will say that an ideal $\mci$ is $(\kappa,\infty)$-\emph{distributive} if
$P(\kappa)/\mci$ is a $(\kappa,\infty)$-distributive Boolean Algebra.

\end{definition} 

\noindent
In this language,  Theorem \ref{t2} can be restated as saying that Player II
having a winning strategy in the Welch game implies the existence of a normal
$\gamma$-densely treed ideal and the tree has height $\gamma$. 

We have a partial converse to Theorem \ref{t2}:
\begin{theorem}\label{t4}
Let $\gamma\le\kappa$ be uncountable regular cardinals and $\mcJ$ be a
uniform $\kappa$-complete ideal over $\kappa$ 
which is $(\kappa^+,\infty)$-distributive and has a dense $\gamma$-closed
subset. Then Player~II has a winning strategy in the game $\mcG^W_\gamma$
which is constructed in a natural way from the ideal $\mcJ$, and which we
denote by $\mcS_\gamma(\mcJ)$.
\end{theorem}

\noindent
{A proof of Theorem~\ref{t4} is at the end of Section~\ref{s.star}. We note
that if $\kappa$ carries a uniform, $\kappa$-complete ideal which is
$(\kappa^+,\infty)$-distributive, then $\kappa$ must be inaccessible. 
\bigskip
   
\bfni{How does precipitousness arise?} In \cite{GJM}, Galvin, Jech and
Magidor introduced the following game of length $\omega$. Fix an ideal
$\mci$. Players I and II alternate playing  
\begin{center}
\begin{tabular}{ c||c|c|c|c|c|c } 
I & $A_0$& $A_1$ &\dots &$A_n$&$A_{n+1}$&\dots \\
 \hline
II & $B_0$ & $B_1$& \dots &$B_n$&$B_{n+1}$& \dots \\ 
\end{tabular}
\end{center}
With $A_n\supseteq B_n\supseteq A_{n+1}$ and each $A_n, B_n\in
\mci^+$. Player~II wins the game if $\bigcap_n B_n\ne \varnothing$. We will call
this game the \emph{Ideal Game} {for \mci}. They proved the following theorem. 

\begin{theorem*}\cite{GJM}
Let $\mci$ be a countably complete ideal on a set $X$. Then $\mci$ is
precipitous if and only if Player I does not have a winning strategy {in the ideal game} { for $\mci$}. 
\end{theorem*}

In the proof of Theorem~\ref{t1}, 
we construct an ideal $\mci$ and
show that Player~II has a winning strategy in the ideal game {for $\mci$}.  In
Theorem~\ref{t2}, 
the existence of a dense set $D$ closed under descending
$\omega$-sequences immediately gives that Player~II has a winning strategy in
the ideal game. (See \cite{forgames} for some information about the
relationship between games and dense closed subsets of Boolean Algebras.)
The proofs of both Theorem~\ref{t1} and Theorem~\ref{t2} are in
Section~\ref{s.star}. 
\medskip

\bfni{Is this vacuous?} So far we haven't addressed the question of the
existence of strategies in the Welch games if $\kappa$ is \emph{not}
measurable.  We answer this with the following theorem. We use the terminology
regarding closure and distributivity properties of forcing partial orderings from
\cite{ifee}. 
 
\begin{theorem}\label{t3}
Assume $\kappa$ is measurable and $\mbfV=\Le$ is a fine structural extender
model. Then there is a generic extension in which $\kappa$ is inaccessible,
carries no saturated ideals (in particular, $\kappa$ is non-measurable) and
for all regular $\gamma$ with $\omega<\gamma\le \kappa$
there is a uniform, normal $\gamma$-densely treed ideal $\mcj_\gamma$ on
$\kappa$ that is  $(\kappa^+,\infty)$-distributive.
The Boolean algebra $\ptm(\kappa)/\mcj_\gamma$ does not contain a dense
$\gamma^+$-closed subset.  
\end{theorem}

\begin{corollary}\label{c6} It follows from Theorems \ref{t4} and \ref{t3}
that in the forcing extension of Theorem~\ref{t3},  

\begin{itemize}
	\item[(a)] Player~II has a winning strategy
$\mcS_\gamma\dfeq\mcS(\mcj_\gamma)$ in $\mcG^W_\gamma$. 

	\item[(b)] There is an ideal $\mci_\gamma$ as in Theorem~\ref{t2}.  
\end{itemize}
It will follow from the proof of Theorem \ref{t3} that 
 the winning strategies $\mcS_\gamma$ in (a) are incompatible
with winning strategies $\mcS_{\gamma'}$ for Player~II in $\mcG^W_{\gamma'}$ for
$\gamma'\neq\gamma$ in the following sense: If $\gamma,\gamma'\le\kappa$ are
regular and
$\gamma\neq\gamma'$ then it is possible for Player~I to play the
first round $\mcA_0$ in such a way that the responses of $\mcS_\gamma$ and
$\mcS_{\gamma'}$ to $\langle\mcA_0\rangle$ are distinct.
\end{corollary}

We give a proof of  {Theorem \ref{t3}} in Section~\ref{s.model}.
{The
existence of winning strategies $\mcS_\gamma$ as in (a) for Player~II in
$\mcG^W_\gamma$ is a direct consequence of 
Theorem~\ref{t4}. A proof of the incompatibility of strategies $\mcS_\gamma$, 
as formulated at the end of Corollary~\ref{c6}, is at the
end of Section~\ref{s.model}.} 
\bigskip

\bfni{Strengthenings  of Theorem \ref{t3}}
We have two variants of Theorem \ref{t3} that are proved in Part II of this paper.  The first deals with a single regular uncountable  $\gamma<\kappa$, and shows that it is consistent that $\gamma$ is the only cardinal such that there is a normal $\gamma$-densely treed ideal on $\kappa$.
 The second shows that it is consistent that for all such $\gamma$ there is a normal $\gamma$-densely treed ideal $\mcj_\gamma$ on $\kappa$ but that they are all incompatible under inclusion.  

Similar statements about the relevant strategies in the Welch games are also included. Explicitly:

\begin{theorem}\label{t6} 
Assume $\kappa$ is a measurable cardinal, $\gamma<\kappa$ is regular
uncountable and 
$\mbfV=\Le$ is a fine structural extender model. Then there is a generic
extension in which $\kappa$ is inaccessible, carries no saturated ideals (in
particular, $\kappa$ is non-measurable) and there is a uniform, normal
$\gamma$-densely treed ideal  $\mcj_\gamma$ on $\kappa$ that is 
$(\kappa^+,\infty)$-distributive.
Moreover,  
in the generic extension:
\begin{itemize} 
\item[(a${}^*$)] {There does not exist a uniform 
ideal $\mcj'$
over $\kappa$ such that $\ptm(\kappa)/\mcj'$ has a dense $\gamma'$-closed
subset for any $\gamma'>\gamma$.}  
\item[(b${}^*$)] Player~II does not have any winning strategy in
$\mcG^W_{\gamma'}$ where $\gamma'>\gamma$.
\end{itemize}
In particular it is a consequence of (a*) that

\begin{itemize}
	\item[(c${}^*$)] For all regular $\gamma'>\gamma$ there is no uniform normal $\gamma'$-densely treed 
	ideal on $\kappa$. 
\end{itemize}

\end{theorem}

Another modification of the proof of Theorem~\ref{t3} which is based on
Theorem~\ref{t5} below yields the
following variant of Theorem~\ref{t3}.

	\begin{theorem}\label{t7} 
Assume $\kappa$ is a measurable cardinal, and $\mbfV=\Le$ is a fine structural extender model. 
Then there is a generic
extension in which $\kappa$ is inaccessible, carries no saturated ideals (in
particular, $\kappa$ is non-measurable) and for all regular $\gamma$ with $\omega<\gamma\le \kappa$ 
there is a uniform, normal $\gamma$-densely treed ideal $\mcJ_\gamma$ that is $(\kappa^+,\infty)$-
distributive. 
The relationship between the ideals and  strategies for different $\gamma$'s is as follows:
\begin{itemize} 
\item[(a${}^*$)] There does not exist a uniform normal ideal
$\mcj'\subseteq\mcJ_\gamma$ over $\kappa$ such that $\ptm(\kappa)/\mcj'$ has a
dense $\gamma'$-closed subset for any $\gamma'>\gamma$. 
\item[(b${}^*$)] The strategy $\mcS_\gamma\dfeq\mcS(\mcj_\gamma)$ is not
included in any winning strategy for Player~II in $\mcG^W_{\gamma'}$ where
$\gamma'>\gamma$. 
\item[(c${}^*$)] {Letting $\mci_\gamma$ be the ideal arising from the strategy $\mcs_\gamma$,} there does not exist an ideal
$\mcI\subseteq\mcI_\gamma$ which is $\gamma'$-densely treed as witnessed by a
tree $D\subseteq\mcI^+$ of height $\gamma'$, for any $\gamma'>\gamma$.  
\footnote{
There are two general techniques used in this paper for building ideals.  One is the conventional method of starting with a large cardinal embedding and extending it generically. We use the notation $\mcJ_\gamma$ for these.  The second is the new technique of \emph{hopeless ideals}, built in Theorems \ref{t1} and \ref{t2} from the strategies $\mcs_\gamma$. These will be denoted by $\mcI_\gamma$ or very similar notation.}

\end{itemize}
	\end{theorem}

In other words, the ideals $\mcI$ in (c${}^*$) in Theorem~\ref{t7} are like
ideals $\mcI$ in Theorem~\ref{t2}, with $\gamma'$ in place of $\gamma$.

The models constructed in Theorems \ref{t6} and \ref{t7} require more
sophisticated techniques than those used in the proof of Theorem~\ref{t3}.
They involve the relationship between the fine structure in the base model and
the forcing extension. 

The most substantial difference is that the model in Theorem~\ref{t3} {is
built} by iteratively shooting clubs through the complements  
of non-reflecting stationary sets which have been  added generically, 
however the proofs of Theorems \ref{t6} and \ref{t7} shoot club sets through
non-reflecting stationary sets built from canonical square 
sequences constructed in the fine structural extender model. Unlike the partial orderings
used in the construction of a model in the proof of Theorem~\ref{t3}, those
partial orderings will have low closure properties, but high degree of distributivity. It
is the proof of distributivity of the iterations of club shooting partial orderings which
uses the significant fine structural properties
of the extender model.  Here is the result allowing the desired iteration.

\begin{theorem}\label{t5}
Assume $\mbfV=\Le$ is a fine structural extender model and $\kappa$ is a
measurable cardinal as witnessed by an extender on the extender sequence
$E$. Assume further that
\begin{itemize}
\item[(i)] $(c_\xi\mid\xi<\alpha^+)$ is a canonical square sequence,\footnote{By a canonical square sequence we mean a square sequence obtained by a slight variation of Jensen's fine structural construction, generalized to extender models.  This is made precise in Part II of this paper.}
\item[(ii)] $S_\alpha\subseteq\alpha^+\cap\cff(<\alpha)$,
\item[(iii)] $S_\alpha\cap c_\xi=\emptyset$ for all $\xi$
\end{itemize}
whenever $\alpha$ is a cardinal.

Let $\mbbP^\delta$ be the Easton support iteration of length $\kappa$ of club
shooting partial orderings with initial segments where each active stage $\alpha$ is
an inaccessible $\ge\delta$ and the club subset of $\alpha^+$ generically
added at stage $\alpha$ is disjoint from $S_\alpha$. Then there is an ordinal
$\varrho<\kappa$ such that for every inaccessible $\delta$ such that
$\varrho<\delta<\kappa$ the following holds.
\begin{itemize}
\item[(a)] $\mbbP^\delta$ is $\delta^+$-distributive.
\item[(b)] If $G$ is generic for $\mbbP^\varrho$ over $\mbfV$ and
$j:\mbfV\to M$ is an elementary embedding in some generic extension $\mbfV'$
of $\mbfV$ which preserves $\kappa^+$ then $j(\mbbP^\varrho)/G$ is
$\kappa^+$-distributive in $\mbfV'$.  
\end{itemize}
\end{theorem}
Although Theorem \ref{t5} is formulated for Easton support iterations with inaccessible active stages, variations which involve iterations with supports which are not necessarily Easton, but still sufficiently large, and with active stages that are not necessarily inaccessible can also be proved. 

As the  proof of Theorem \ref{t5} is of considerable length and (we
believe) {has broader applicability and is} of interest on its own, we will
postpone  the proof to Part II of this paper. 

\bigskip
\bfni{Basic definitions and notation}
 {We now present terminology and} notation we use throughout the paper.  {We will use the phrases 
``ideal on $\kappa$" and ``ideal on $P(\kappa)$" interchangeably.  Perhaps ideals should be viewed as subsets of Boolean algebras, but the former phrase is the more common colloquialism.}

Fix a regular cardinal
$\kappa$ and $\mcI$ a $\kappa$-complete ideal on $\kappa$. We say that
$A\subseteq_\mci B$ if $A\smallsetminus B\in \mcI$, and $\supseteq_\mcI$ is the
converse relation. The notations $\subseteq^*$, $\supseteq^*$ are these
notions when $\mcI$ is the ideal of bounded subsets of $\kappa$.  The notation
$A\subsetneq_\mci B$ abbreviates the conjunction of  
$A\subseteq_\mci B$ and $A\symdiff B\notin \mci$, where $\symdiff$ means
symmetric difference. 

The ideal $\mcI$ induces an equivalence relation on $\ptm(\kappa)$ by
$[A]=[B]$ if and only if $A\symdiff B\in \mcI$. The notion $\subseteq_\mcI$
induces a partial ordering on  $\ptm(\kappa)/\mcI$, we will sometimes call
this $\le_\mci$ and refer to the set of $\mcI$ equivalence classes of  
$\ptm(\kappa)$ that don't contain the emptyset as $\mcI^+$. We will force with 
$(\ptm(\kappa)/\mcI, \subseteq^*_\mcI)$ viewed either as a Boolean algebra, or removing the equivalence class of the emptyset as a partial ordering.  These are equivalent forcing notions. Occasionally we will abuse language by saying ``forcing with $\mcI^+$" when we mean this forcing.

\begin{definition}\label{O-OK} \hypertarget{OOK}{$\kappa$-complete sub-Boolean algebras of $\ptm(\kappa)$ that have cardinality $\kappa$ are called
\emph{$\kappa$-algebras.}}
\end{definition}

If $\sigma$ and $\tau$ are sequences we will use $\sigma\cat \tau$ to mean the concatenation of $\sigma$ and $\tau$.  We will abuse this slightly when $\tau$ has length one.  For example given $\sigma=\la \alpha_i:i<\beta\ra$ and $\delta$ we will write $\la \alpha_i:i<\beta\ra\cat \delta$ for the sequence of length $\beta+1$ whose first $\beta$ elements coincide with $\sigma$ and whose last element is $\delta$.

Usually our trees grow downwards, with longer branches extending shorter branches. \hypertarget{gammaclosed}{A tree $\mct$ is $\gamma$-closed if when $b$ is a branch through $\mct$ whose length has cofinality less than $\gamma$ there is a node $\sigma\in \mct$ such that $\sigma$ is below each element of $b$.} Occasionally we will say $<\!\gamma$-closed to mean $\gamma$-closed.


\section{Weak Compactness}\label{weak compactness}
In this section we clarify  the relationship between these games and weak compactness and discuss the role of 
inaccessibility in the work of Keisler and Tarski.  It has been pointed out to us that these results appear  in {work of 
Abramson, Harrington, Kleinberg and Zwicker (\cite{flipping}) stated slightly differently and with different proofs}.  We 
include them here for completeness and because these techniques are relevant to the topics in this paper.

If $\kappa$ is inaccessible and $\mca$ is a $\kappa$-algebra and $B\subseteq [\kappa]^\kappa$ then $\mca\cup B$ 
generates a 
$\kappa$-complete subalgebra of $\ptm(\kappa)$ that has cardinality $\kappa$ (i.e. another $\kappa$-algebra).  The 
situation where $\kappa$ is not inaccessible is quite different.

\begin{proposition}\label{no A} Suppose that $\kappa$ is an infinite cardinal and either
	\begin{itemize} 
	\item a singular strong limit cardinal\footnote{We would like to thank James Cummings for giving significant help in 
	understanding  this case.} or 
	\item for some $\gamma<\kappa$, $2^\gamma>\kappa$ but for all $\gamma'<\gamma$, $2^{\gamma'}<\kappa$. 
	\end{itemize}
Then there is no Boolean subalgebra $\mca\subseteq \ptm(\kappa)$ such that $|\mca|=\kappa$, $\mca$ is $\kappa$-complete. 
\end{proposition}
	\begin{proof} In the first case, since $\kappa$ is singular, if $\mca$ is $\kappa$-complete, it is $\kappa^+$-complete.  
 For $\delta<\kappa$, let $a_\delta=\bigcap\{A\in \mca:\delta\in A\}$. Then $a_\delta$ is an atom of $\mca$ and each non-empty
 $A\in \mca$ contains some $a_\delta$. {Moreover distinct $a_\delta$'s are  disjoint.} Thus the 
 $\{a_\delta:\delta\in \kappa\}$ generate $\mca$ as a $\kappa$-algebra.
  If there are $\kappa$ many distinct $a_\delta$ then $|\mca|=2^\kappa$.  Otherwise, since $\kappa$ is a strong limit, 
  $|\mca|<\kappa$.
  
  Assume now that $\gamma<\kappa$, $2^\gamma>\kappa$ and for all $\gamma'<\gamma,\ 2^{\gamma'}<\kappa$.  {Since $|\mca|=\kappa$,}  
  $\mca$ must have fewer than $\gamma$ atoms. If $\la a_\delta:\delta<\gamma'\ra$ is the collection of these atoms, the $\kappa$-algebra $\mcb$ generated by the atoms of 
  $\mca$ has cardinality at most $2^{\gamma'}$.  Let $B=\bigcup_{\delta<\gamma'}a_\delta$ and $C=\kappa\setminus B$.
  Since $2^{\gamma'}<\kappa$ and $\mca$ has cardinality $\kappa$,   there is an $a\in \mca$ that does not belong to $\mcb$. Since $a$ is not in $\mcb$ the set $a'=a\setminus \bigcup_\delta a_\delta\ne \emptyset$.  Hence
  $C$ is non-empty.
  Replacing $\mca$ with $\{A\setminus B:A\in \mca\}$ we get an atomless, $\kappa$-complete algebra 
  on the set $C$.

 Since no element of $\mca$ is an atom we can write each $A\in \mca$ as a disjoint union of non-empty elements 
 $A_0, A_1$ of $\mca$.   Build a binary splitting tree $\mct$ of elements of $\mca$ of height $\gamma$ by induction on 
 $\delta<\gamma$ as follows:\footnote{We use the notation $(\mct)_\alpha$ for level $\alpha$ of $\mct$.}
		\begin{itemize}
		\item $(\mct)_0$ is the $\subseteq$-maximal element $C$ of $\mca$.
		\item Suppose $(\mct)_\delta$ is built.  For each $A\in \mct_{\delta}$, write $A$ as the disjoint union 
		$A_0\cup A_1$ 
		such that 
		$A_i\in \mca$,  and $A_i\ne \emptyset$ and let $(\mct)_{\delta+1}=\{A_0, A_1:A\in (\mct)_\delta\}$. 
		\item Suppose that $\delta$ is a limit ordinal. Let $(\mct)_\delta=\{\bigcap b:b$ { is a branch through }
		$(\mct)_{<\delta}$ and $\bigcap b\ne \emptyset\}$.
		
		\end{itemize}
Note that for all $\delta<\gamma$, each element $c\in C$ determines a unique path through $(\mct)_\delta$ of length $\delta$. Hence $\bigcup (\mct)_\delta=C$. 
 
 Fix a $c\in C$ and let $b=\la A_\delta:\delta<\gamma\ra$ be the branch through $\mct$ determined by $c$.  For 
 $\delta<\gamma$, the tree splits $A_\delta=A_{\delta+1}\cup A'_{\delta+1}$ with 
 $A_{\delta+1}, A'_{\delta+1}\in (\mct)_{\delta+1}$. The sets $\la A'_{\delta+1}:\delta<\gamma\ra$ each 
 belong to $\mca$ and form a collection of disjoint subsets of $C$ of size 
 $\gamma$. By taking unions of these sets we see that $|\mca|\ge 2^\gamma>\kappa$.
 \end{proof} 

In contrast to Proposition \ref{no A}, we have:

\begin{proposition}\label{yes and}
Let $\kappa$ be infinite and $2^\gamma=\kappa$. Then $\ptm(\gamma)\subseteq \ptm(\kappa)$ and $\ptm(\gamma)$ is $\kappa$-complete.
\end{proposition}
\begin{proof} Immediate.\end{proof}

The upshot of Propositions \ref{no A} and \ref{yes and} is the following theorem and corollary, which show that the Welch game is only interesting when $\kappa$ is inaccessible. 
 \begin{theorem}\label{showingprincipals} Suppose that $\kappa$ is an accessible infinite cardinal.  Then either:
	 \begin{enumerate}
 	\item There is no \hyperlink{OOK}{$\kappa$-algebra} $\mca\subseteq \ptm(\kappa)$ with $|\mca|=\kappa$ or
	\item There is a $\kappa$-algebra $\mca\subseteq \ptm(\kappa)$ with $|\mca|=\kappa$ but every 
	$\kappa$-complete ultrafilter $U$ on $\mca$ is principal
	 \end{enumerate} 
 \end{theorem}
 \begin{corollary} Consider the Welch game of length 1. 
 Suppose that there is a $\kappa$-algebra  $\mca_0$ that is a legal move for Player I and that Player II has a winning strategy in $\mcG^W_1$.  Then $\kappa$ is inaccessible.
  \end{corollary}

If $\kappa$ is inaccessible we have the following result, which can also be deduced directly from the results of \emph{Abramson et al.} (\cite{flipping}):
\begin{theorem} Suppose that $\kappa$ is inaccessible and Player II wins the Welch game of length 1. Then $\kappa$ is weakly compact.
\end{theorem}
\begin{proof} To show $\kappa$ is weakly compact, it suffices to show it has the tree property.  Let $\mct$ be a $\kappa$-tree. For $\alpha<\kappa$, let $A_\alpha$ be the the set of $\beta<\kappa$ such that $\alpha\le_\mct\beta$. Let $\mca$ be the $\kappa$-algebra generated by $\{A_\alpha:\alpha<\kappa\}$ and $\mcu$ be a uniform $\kappa$-complete ultrafilter on $\mca$. 

For each $\gamma<\kappa$, $\kappa=\bigcup_{\alpha\in (\mct)_\gamma}A_\alpha \cup R$ where $|R|<\kappa$.  It follows that for each $\gamma$ there is an $\alpha\in (\mct)_\gamma$ such that $A_\alpha\in \mcu$.  But then $\{\alpha:A_\alpha\in \mcu\}$ is a $\kappa$-branch through $\mct$.\end{proof}


\section{Hopeless Ideals}\label{s.hopeless}

In this section we define the notion of a \emph{hopeless ideal} in a general
context, and toward the end of the section we will narrow our focus to the
context of games. Fix an 
inaccessible cardinal $\kappa$. Assume $F$ is a function with domain  $R$
such that for every $r\in R$ the value $F(r)$ is a sequence of length $\xi_r$
of the form
\begin{equation}\label{hopeless.e.f}
F(r)=\langle\mcA^r_i,U^r_i\mid i<\xi_r\rangle
\end{equation}
where for every $i<\xi_r$, 
\begin{itemize}
\item[(i)] $\mcA^r_i\subseteq\ptm(\kappa)$ and 
\item[(ii)] $U^r_i$ is a $\kappa$-complete ultrafilter on the $\kappa$-algebra
of subsets of $\kappa$ generated by $\bigcup_{j\le i}\mcA^r_j$.
\item[(iii)] {For all $r\in R$, the sequence $\langle U^r_i\mid
i<\xi_r\rangle$ is monotonic with respect to the inclusion.}
\item[(iv)]\label{density} (Density) For every $r\in R$, $j<\xi_r$ and
$\mcB\in[\ptm(\kappa)]^{\le\kappa}$ there is $s\in R$ such that
$F(r)\rst j=F(s)\rst j$ and $\mcB\subseteq\mcA^s_j$. 
\end{itemize}
{We will call functions $F$ with the properties (i)--(iv) \emph{assignments}.}

{One can also formulate a variant with normal ultrafilters $U^r_i$. Denote the maximo-lexicographical ordering of
$\kappa\times\kappa$ by $\mlex$. Let $h:(\kappa\times\kappa,\mlex)\to(\kappa,\in)$ be the natural isomorphism. For a set 
$A\subseteq \kappa$, let $A_i=(h^{-1}[A])_i$ be the $i^{th}$ section of $h^{-1}(A)$. The sequence  $\la A_i\mid i<\kappa\ra$ is \emph{associated to $A$}.
We will say that a 
$\kappa$-algebra $\mcA$ of subsets of $\kappa$ is \emph{normal} if for all $A\in \mca$, each $A_i$ belongs $\mca$ and  the diagonal intersection 
$\Delta_{i<\kappa}A_i$ also belongs to $\mca$. We will say that a sequence $\la A_i\mid i<\kappa\ra$ belongs to $\mca$ if it is associated to an element of $\mca$. }
 Finally we say that an ultrafilter $U$
on a normal $\kappa$-algebra $\mcA$ is \em normal \em iff for every sequence
$\langle A_i\mid i<\kappa\rangle\in\mcA$,
\begin{equation}\label{hopeless.e.nuf}
(\forall i<\kappa)(A_i\in U)
\quad
\Longrightarrow
\quad
\Delta_{i<\kappa}A_i\in U
\end{equation}
{A variant of this definition is 
an \emph{assignment with normal ultrafilters} where we require, instead of (ii)
above, that}
\begin{itemize}
\item[(ii)${}'$] $U^r_i$ is a $\kappa$-complete normal ultrafilter on the
normal $\kappa$-algebra of subsets of $\kappa$ generated by
$\bigcup_{j\le i}\mcA^r_j$. 
\end{itemize}
If (ii)$'$ is satisfied we say that $F$ is \em normal. \em Notice that there
is no need to modify clause (iv), as normal $\kappa$-algebras are able to
decode $\kappa$-sequences of subsets of $\kappa$ {from other} subsets of
$\kappa$ via the pairing function $h$ introduced above.  However, instead of
families $\mcB\in[\ptm(\kappa)]^{\le\kappa}$,  it is convenient in (iv) to
consider sets $B\in\ptm(\kappa)$ that code~$\mcB$. 

\begin{definition}\label{hopeless.d.hi}
Given an \emph{assignment} $F$, we define the ideal $\mcI(F)$ as follows.
\begin{equation}\label{hopeless.e.hi}
\mcI(F)=\mbox{the set of all $A\subseteq\kappa$ such that $A\notin U^r_i$
for any $i<\xi_r$ and any $r\in R$.} 
\end{equation}
The ideal $\mcI(F)$ is called the \em hopeless ideal on $\ptm(\kappa)$ induced
by $F$. \em
\end{definition}
Although in the above definition we say we are defining an ideal, an argument
is needed to see that $\mcI(F)$ is indeed an ideal. 
It follows immediately that $\varnothing\in\mcI(F)$ and $\mcI(F)$ is downward
closed under inclusion. The rest is given by the following proposition.

\begin{proposition}\label{hopeless.p.completeness}
Given an assignment $F$, 
the ideal $\mcI(F)$ is $\kappa$-complete. If all ultrafilters $U^r_i$ are
uniform then $\mcI(F)$ is uniform. If additionally $F$ is normal then
$\mcI(F)$ is normal. If $F'$ is an assignment on $R'\supseteq R$ and $F'\rest
R=F$, then $\mcI(F')\subseteq \mcI(F)$. 
\end{proposition}

\begin{proof}
We first verify $\kappa$-completeness of $\mcI(F)$. 
We noted above that $\varnothing\in\mcI(F)$ and $\mcI(F)$ is downward
closed under inclusion; hence it suffices to check that $\mcI(F)$ is closed
under unions of cardinality $<\kappa$. If $\langle A_\eta\mid\eta<\xi\rangle$
is such that $\xi<\kappa$ and $A=\bigcup_{\eta<\xi}A_\eta\notin\mcI(F)$, then
there is some $r\in R$ and some $i<\xi_r$ such that $A\in U^r_i$. By the density
condition, there is some $s\in R$ such that $\mcA^s_j=\mcA^r_j$ and  
$U^s_j=U^r_j$ for all $j\le i$, and $\{A_\eta\mid \eta<\xi\}\subseteq\mcA^s_{i+1}$.
In particular, $A\in U^s_i\subseteq U^s_{i+1}$ and all sets $A_\eta$,
$\eta\le\xi$ are in the $\kappa$-algebra generated by
$\bigcup_{j\le i+1}\mcA_j$. By
$\kappa$-completeness of $U^s_{i+1}$ then $A_\eta\in U^s_{i+1}$ for some
$\eta<\xi$, hence $A_\eta\notin\mcI(F)$. 

The proof of normality of $\mcI(F)$ for normal $F$ is the same, with
$\nabla_{\eta<\kappa}A_\eta$ in place of $\bigcup_{\eta<\xi}A_\eta$.
The conclusion on uniformity of $\mcI(F)$ follows by a straightforward
argument from the definition of $\mcI(F)$.

Finally, if $R',F'$ are as in the statement of the proposition, then any
$A\in\mcI(F')$ trivially avoids all ultrafilters $U^r_i$ where $r\in R$ and
$i<\xi_r$, so $A\in\mcI(F)$. 
\end{proof}

Now assume $\mcG$ is a two player game of perfect information, and
$\mcS$ is a strategy for Player~II in $\mcG$. Denote the set of all
runs  of $\mcG$ according to $\mcS$ by
$R_\mcS$ (by a run we mean a complete play). Assume every $r\in R_\mcS$ 
is associated with a sequence of fragments $\mcA^r_i\subseteq\ptm(\kappa)$
and ultrafilters $U^r_i$, in a way that makes the function
\begin{equation}\label{hopeless.e.f-games}
F_\mcS:r\mapsto\langle \mcA^r_i,U^r_i\mid i<\lht(r)\rangle
\end{equation}
an assignment/normal assignment with domain $R_\mcS$. Here of course
$\xi_r=\lht(r)$ when compared with (\ref{hopeless.e.f}). In all concrete
situations we will consider, the rules of the game $\mcG$ will guarantee that
the function $F_\mcS$ is really an assignment. As the
strategy $\mcS$ makes it clear  which game is played, we suppress writing
$\mcG$ explicitly in our notation.

Here are some examples. If $\mcG$ is the Welch game $\mcG^W_\gamma$ then
$F_\mcS$ is the identity function. In the next section we introduce games
$\mcG_1^-,\mcG_1$ and $\mcG_2$. These games are defined relative to a sequence
of models $\langle N_\alpha\mid\alpha<\kappa^+\rangle$ increasing with respect
to the inclusion, and Player~I plays ordinals $\alpha<\kappa^+$ which refer to
these models. In the games $\mcG_1^-$ and $\mcG_1$ Player~II plays
uniform $\kappa$-complete ultrafilters on $\ptm(\kappa)\cap N_\alpha$;
in $\mcG_1$ these ultrafilters are required to be normal. Thus, if $r$ is a
run in one of these games according to $\mcS$, say
$r=\langle\alpha^r_i,U^r_i\mid i<\lht(r)\rangle$ then
\[
F_\mcS(r)=\langle \ptm(\kappa)\cap N_{\alpha^r_i},U^r_i\mid i<\lht(r)\rangle
\]
In the game $\mcG_2$ Player~II plays sets $Y\subseteq\kappa$ which determine
uniform normal $\kappa$-complete ultrafilters $U$ on $\ptm(\kappa)\cap
N_\alpha$ defined by
$U=\{X\in\ptm(\kappa)\cap N_\alpha\mid Y\subseteq^* X\}$.
Thus, if $r$ is a run in the game $\mcG_2$ according to $\mcS$, say
$r=\langle\alpha^r_i,Y^r_i\mid i<\lht(r)\rangle$ then
\[
F_\mcS(r)=
\langle
\ptm(\kappa)\cap N_{\alpha^r_i},
\{X\in\ptm(\kappa)\cap N_{\alpha^r_i}\mid Y^r_i\subseteq^* X\}
\mid i<\lht(r)
\rangle
\]

If $P$ is a position in $\mcG$ played according to $\mcS$ we let
\begin{equation}\label{hopeless.e.r-positional}
R_{\mcS,P}=\mbox{the set of all $r\in R_\mcS$ extending $P$}
\end{equation}
and
\begin{equation}\label{hopeless.e.f-positional}
F_{\mcS,P}=F_\mcS\rst R_{\mcS,P}
\end{equation}
We are now ready to define the central object of our interest.

\begin{definition}\label{hopeless.d.hig}
Assume $\mcG$ is a game of perfect information played by two players, $\mcS$
is a strategy for Player II in $\mcG$, and $F_\mcS$ is an assignment
with domain $R_\mcS$ as in (\ref{hopeless.e.f-games}).
\hypertarget{hopeless}{Consider a position $P$ in
$\mcG$ according to $\mcS$. We define 
\[
\mcI(\mcS,P)=\mcI(F_{\mcS,P})
\]
to be the \em hopeless ideal with respect to $\mcS$ conditioned on $P$. Here
we suppress mentioning the assignment $F_\mcS$ in the notation, as in all
situations we will consider it will be given by the strategy $\mcS$ in a
natural way.
\em The ideal $\mcI(\mcS,\varnothing)$ is called the \em unconditional hopeless
ideal with respect to $\mcS$.} We will write $\mcI(\mcS)$ for
$\mcI(\mcS,\varnothing)$. 
\end{definition}

When the strategy $\mcS$ is clear from the context we suppress referring to
it, and will talk briefly about the ``hopeless ideal conditioned on $P$" and
the ``unconditional hopeless ideal". By
Proposition~\ref{hopeless.p.completeness} we have the following as an
immediate consequence.

\begin{proposition}\label{hopeless.p.completeness-games}
Given a game $\mcG$ of limit length, a strategy $\mcS$ for Player~II in $\mcG$
and a position $P$ as in Definition~\ref{hopeless.d.hig}, the ideal
$\mcI(\mcS,P)$ is $\kappa$-complete. If all ultrafilters $U^r_i$ associated
with $F_{\mcS,P}$ are uniform then $\mcI(\mcS,P)$ is uniform. If
moreover $F_{\mcS,P}$ is normal, then $\mcI(\mcS,P)$ is normal as well. 
\end{proposition}


\section{Games we Play}\label{s.games}

In this section we introduce a sequence of games $\mcG_k$ closely related to
Welch's game $\mcG^W_\gamma$. The last one will be $\mcG_2$, and we will be
able to show that if $\mcS$ is a winning strategy for II in $\mcG_2$ of
sufficient length then we can construct a winning strategy $\mcS^*$ for
Player~II in $\mcG_2$ such that $\mcI(\mcS^*)$ is precipitous and more,
depending on the length of the game
and the payoff set.

To unify the notation, we let $\mcG_0$ of length $\gamma$ be the Welch game
$\mcG^W_\gamma$.
Thus, a run of the game continues until either Player II cannot play or else
until $\gamma$ rounds are played. The set of all runs of $\mcG_0$ of length
$\gamma$ is denoted by $R_\gamma$. As usual with these kinds of games, a set
$B\subseteq R_\gamma$ is called a payoff set. We say that Player~II wins {a 
run $R$ of the game $\mcG_0$ of length $\gamma$ with payoff set $B$ if  $R$
has   $\gamma$ rounds and the resulting run is an element of $B$. We call this
game $\mcG_0(B)$.} Thus, if $B=R_\gamma$ then $\mcG_0(B)$ is just the game
$\mcG_0$. With this notation, the game $\mcG^*_\gamma$ is just the game 
\hypertarget{Qgamma}{$\mcG_0(Q_\gamma)$} of length $\gamma$ where

\begin{equation}\label{intro.e.q-gamma}     
Q_\gamma=\;
\parbox[t]{3.9in}{
The set of all runs $\langle\mcA_i,U_i\mid i<\gamma\rangle\in R_\gamma$ such
that there is a $\kappa$-complete ultrafilter on the $\kappa$-algebra
generated by $\bigcup_{i<\gamma}\mcA_i$ extending all $U_i$, $i<\gamma$.
}
\end{equation}
As already discussed in the introduction, the existence of a winning strategy
for Player~II in the game $\mcG_0(Q_\gamma)$ of length 
$\gamma$ is a strengthening of the requirement that Player~II has a winning
strategy in $\mcG_0$ of length $\gamma$. This strengthening is among the
weakest ones which increase the consistency strength  
in the case $\gamma=\omega$. From the point of view of increasing the
consistency strenth, the case $\gamma=\omega$ is of primary interest, as
follows from (TO1) combined with Corollary~\ref{c6}. Here
are some trivial observations.  
\begin{itemize}
\item[(TO1)] $\mcG_0(Q_\gamma)$ is the same game as $\mcG_0$ whenever $\gamma$
is a successor ordinal, so a winning strategy for Player~II in
$\mcG_0(Q_\gamma)$ gives us something new only when $\gamma$ is a limit.
\item[(TO2)] A winning strategy for Player~II in $\mcG_0(Q_\gamma)$ is a winning
strategy for Player~II in $\mcG_0$, but the converse may not be true in
general. 
\item[(TO3)] If $\mcS$ is a winning strategy for Player~II in $\mcG_0$ of length
$>\gamma$ then the restriction of $\mcS$ to positions of length $<\gamma$ is a
winning strategy for Player~II in $\mcG_0(Q_\gamma)$ of length $\gamma$. 
\item[(TO4)] Given $\xi<\kappa$ and sequences $\langle\mcA_i\mid i<\xi\rangle$
and $\langle U_i\mid i<\xi\rangle$ where $\mcA_i\subseteq\ptm(\kappa)$ and
$U_i$ is a $\kappa$-complete 
ultrafilter on the $\kappa$-algebra  (respectively normal $\kappa$-algebra)
of subsets of $\kappa$ generated by $\bigcup_{j\le i}\mcA_j$ such that
$U_i\subseteq U_j$ whenever $i\le j$, there is at most one $\kappa$-complete
(respectively normal) ultrafilter $U$ on the (normal) $\kappa$-algebra $\mcB$ of
subsets of $\kappa$ generated by $\bigcup_{i<\xi}\mcA_i$ which extends all
$U_i$. Thus, if we changed the rules of $\mcG_0$ to require that Player~II
goes first at limit stages then Player~II has a winning strategy in this
modified $\mcG_0$ if and only if Player~II has a winning strategy in the
original game $\mcG_0$.  
\item[(TO5)]{ Let $\mcs$ be a winning strategy for Player II in $\mcG_0$ or $\mcG_0(Q_\gamma)$ and 
\begin{center}
\begin{tabular}{ c||c|c|c|c|c|c } 
 I & $\mca_0$& $\mca_1$ &\dots &$\mca_\alpha$&$\mca_{\alpha+1}$&\dots \\
 \hline
II & $U_0$ & $U_1$& \dots &$U_\alpha$&$U_{\alpha+1}$& \dots \\ 
\end{tabular}
\end{center}
be a play of the game $\mcG_0$ or $\mcG_0(Q_\gamma)$ according to $\mcS$.  Let $\mcb_i\subseteq \mca_i$ be another sequence of 
$\kappa$-complete algebras.  Then the play:
\begin{center}
\begin{tabular}{ c||c|c|c|c|c|c } 
 I & $\mcB_0$& $\mcB_1$ &\dots &$\mcB_\alpha$&$\mcB_{\alpha+1}$&\dots \\
 \hline
II & $U_0\rest \mcB_0$ & $U_1\rest \mcB_1$& \dots &$U_\alpha\rest \mcB_\alpha$&$U_{\alpha+1}\rest \mcB_{\alpha+1}$& \dots \\ 
\end{tabular}
\end{center}
is a run of the game where Player II wins.}
 
\end{itemize}

In what follows we will consider $\theta$ a regular cardinal much larger than
$\kappa$, and \hypertarget{Fixisin}{fix a well-ordering of $H_\theta$ which we denote by
$<_\theta$.} We augment our language of set theory by a binary relation symbol
denoting this well-ordering, and work in this language when taking
elementary hulls of $H_\theta$. We will thus work with the structure
$(H_\theta,\in,<_\theta)$, but will {frequently} suppress the symbols denoting $\in$ and
$<_\theta$ in our notation. 

The common  background setting for the games we are going to describe is an 
\hypertarget{IA}{internally approachable sequence $\langle N_\alpha\mid\alpha<\kappa^+\rangle$} of
elementary substructures of $H_\theta$.  That is:
a continuous sequence  such that for all $\alpha<\kappa^+$ the
following hold.
\begin{itemize}
\item[(a)] $\kappa+1\subseteq N_\alpha$ and $\card(N_\alpha)=\kappa$,
\item[(b)] ${}^{<\kappa} N_{\alpha+1}\subseteq N_{\alpha+1}$,
\item[(c)] $\langle N_\xi\mid\xi\le\alpha'\rangle\in N_\alpha$ whenever
$\alpha'<\alpha$.
\end{itemize}
\medskip

\noindent {The following are standard remarks:
	\begin{itemize}
	\item If we are playing any of the games $\mcG_0, \mcG_1^-, \mcG_1$ of $\mcG_2$ then the game has length $\gamma\le \kappa$.  Since $\kappa+1\subseteq N_0$, $\gamma\subseteq N_\alpha$ for all $\alpha$.
	\item  If $\la N_\alpha:\alpha<\kappa\ra$ is an internally approachable sequence then there is a closed unbounded set $C\subseteq \kappa^+$ such that  for $\alpha\in C$, $N_\alpha\cap \kappa=\alpha$.	%
 	\item  If $2^\kappa=\kappa^+$, then there is a well ordering of
$\ptm(\kappa)$ of order type $\kappa^+$ in $H_\theta$. Hence if $\la
N_\alpha:\alpha<\kappa^+\ra$ is an internally approachable sequence then
$\ptm(\kappa)=\bigcup_{\alpha<\kappa^+}(\ptm(\kappa)\cap N_\alpha$). Clearly
$\ptm(\kappa)\subseteq \ptm(\kappa)\cap \bigcup_{\alpha<\kappa^+}N_\alpha$ 
implies that $2^\kappa=\kappa^+$, which we stated as an assumption 
in Theorems~\ref{t1} and~\ref{t2}. 
	\end{itemize}}

\begin{definition}[The Game $\mcG_1^-$]\label{games.d.g-one-minus}
The rules of the game $\mcG^-_1$ are as follows. {Fix an ordinal $\gamma\le \kappa^+$.}
\begin{itemize}
\item Player I plays an increasing sequence of ordinals $\alpha_i<\kappa^+$.
\item Player II plays an increasing sequence of uniform 
$\kappa$-complete ultrafilters $U_i$ on $\ptm(\kappa)^M$ where
$M=N_{\alpha_i+1}$. 
\item Player I plays first at limit stages. 
\end{itemize}
A run of $\mcG^-_1$ 
continues until Player II
cannot play or  until it reaches length $\gamma$. Player II wins a run
in $\mcG^-_1$  iff the length of the run is $\gamma$.

Payoff sets $R_\gamma$ and $Q_\gamma$ for $\mcG_1^-$ are defined
{analogously to the definition for the game} $\mcG_0$. So $R_\gamma$ consists of all runs of $\mcG_1^-$
of length $\gamma$, and $Q_\gamma$ consists of all runs
$\langle\alpha_i,U_i\mid i<\gamma\rangle\in R_\gamma$ such that there is a
$\kappa$-complete ultrafilter on the $\kappa$-algebra generated by
$\ptm(\kappa)\cap N_\alpha$, where $\alpha=\sup_{i<\gamma}\alpha_i$,  
extending all $U_i$, $i<\gamma$. 
\end{definition}
{The  symbols $R_\gamma$ and $Q_\gamma$ have a double usage: They were}
also defined in connection with the game $\mcG_0$
and {were} different, but analogous to that in
Definition~\ref{games.d.g-one-minus}.  Thus, to determine the exact meaning of
$R_\gamma$ and $Q_\gamma$ one always needs to take into account which game is
being considered.
\medskip

\noindent {In the case where $\gamma=\kappa^+$,  if Player II has a winning strategy in any of the games then $\kappa$ is measurable. So for the purposes of this paper we can assume that $\gamma\le \kappa$, in particular $\gamma\in N_\alpha$ for every $\alpha$.}

{
\begin{remark*} 
Let $\langle\alpha_i,U_i\mid i<\xi\rangle$ be a full or partial play of the game $\mcG^-_1$ and 
$\alpha=\sup_{i<\xi}\alpha_i$. 
	\begin{enumerate}  
	\item If $\xi$ has cofinality $\kappa$ then $\ptm(\kappa)\cap N_\alpha$ is a $\kappa$-algebra.
	\item If $\xi=\zeta+1$, then $\alpha=\alpha_\zeta$, and again $\ptm(\kappa)\cap N_{\alpha_\zeta+1}$ is a 
	$\kappa$-algebra.
	\item if $\xi$ is a limit ordinal of cofinality less than $\kappa$, then the $\kappa$-algebra of sets generated by 
	$\bigcup_{i<\xi}N_{\alpha_i}$ is not a $\kappa$-algebra, the $\kappa$-algebra it generates is strictly larger.
	\end{enumerate}
\end{remark*}
}

 {Finally let us stress that remarks analogous to the remarks
(TO1) -- (TO5) that stated below 
formula (\ref{intro.e.q-gamma}) for games $\mcG_0$ and $\mcG_0(Q_\gamma)$  also hold for $\mcG^-_1$ and
$\mcG^-_1(Q_\gamma)$. }

\begin{proposition}\label{games.p.g-zero-g-one-minus}
Assuming $2^\kappa=\kappa^+$ and $\gamma\le\kappa^+$ is an infinite regular
cardinal, the following hold. 
\begin{itemize}
\item[(a)] Player~II has a winning strategy in $\mcG_1^-$ of length
$\gamma$ iff Player II has a winning strategy in $\mcG_0$ of length $\gamma$.
\item[(b)] Player~II has a winning strategy in $\mcG_1^-(Q_\gamma)$ of length
$\gamma$ iff Player II has a winning strategy in $\mcG_0(Q_\gamma)$ of length
$\gamma$. 
\end{itemize}
Moreover, the analogues of the above equivalences (a) and (b) also hold for
winning strategies for Player~I in the respective games. 
\end{proposition}

Although the last statement in the above proposition concerning winning
strategies for Player~I is not strictly relevant for this paper, we include it
for the sake of completeness. 

\begin{proof}
This is an easy application of auxiliary games. Regarding (a), if $\mcS_0$
is a winning strategy for Player~II in $\mcG_0$ then $\mcS_0$ induces a
winning strategy $\mcS^-_1$ for Player~II in $\mcG^-_1$ the output of which at
step $i$ is the same as the  
output of $\mcS$ at step $i$ in the auxiliary game $\mcG_0$ where Player~I
plays $\ptm(\kappa)\cap N_{\alpha_i+1}$ at step $i$ (where $\alpha_i$ is the
move of Player~I in $\mcG^-_1$ at step $i$). For the converse we proceed
similarly. This time a winning strategy $\mcS^-_1$ for Player~II in $\mcG^-_1$
induces a strategy $\mcS_0$ for Player~II in $\mcG_0$ as follows. If Player~I
plays $\mcA_i$ at step $i$ in $\mcG_0$ then Player~I plays
\[
\alpha_i=\mbox{the least $\alpha>\alpha_{i'}$ for all $i'<i$ such that
$\mcA_i\subseteq N_{\alpha+1}$}
\]
in the auxiliary game $\mcG^-_1$. Letting $U^-_i$ be the output of $\mcS^-_1$
at step $i$, we let the output of $\mcS_0$ at step $i$ to be
$U^-_i\cap\mcA_i$. That $\mcS_0$ is a winning strategy for Player~II in
$\mcG_0$ is immediate.

It is straightforward to verify that this choice of strategies also works in
the case of games with payoff sets $Q_\gamma$ in (b).

Because we will not study winning strategies for Player~I in the games we
consider, we leave the proof of the last statement in the proposition
concerning these strategies to the reader. The proof is based on the same
ideas as the proof of (a), (b) above.
\end{proof}
{We will use the following lemma: }

{\begin{lemma}\label{definable 1}
Suppose that $\mcS_0$ is the $<_\theta$ least winning strategy for Player II in $\mcG_0$ and $S_1^-$ be the strategy defined from $\mcS_0$ as in Proposition \ref{games.p.g-zero-g-one-minus}.  Suppose that {$\beta<\gamma$ and } $\la \alpha_i:i<\beta\ra$ is a sequence of ordinals such that for all {$i,\alpha_{ i+1}<\alpha$}.  Then Player II's response to 
$\la \alpha_i:i<\beta\ra$ in $\mcG_1^-$ belongs to {$N_{\alpha+1}$.}
\end{lemma}}
\begin{proof}
{Because the sequence $\la N_\alpha:\alpha<\kappa^+\ra$ is internally approachable and $\alpha_i<\alpha$,  
$\alpha_i+1<\alpha$. Since we are taking $\gamma\le \kappa$ and  $N_{\alpha+1}$ is closed under $<\kappa$-sequences,} the sequence of $\kappa$-algebras $\la P(\kappa)\cap N_{\alpha_i+1}:i<\beta\ra$ belongs to $N_\alpha$.  Since $\mcS_0$ is $<_\theta$-least, the sequence of responses by Player II  to $\la P(\kappa)\cap N_{\alpha_i+1}:i<\beta\ra$  in $\mcG_0$ belongs to $N_{\alpha+1}$, and hence the sequence of responses by Player II according to $\mcS_1^-$ as defined in Proposition \ref{games.p.g-zero-g-one-minus} belongs to $N_{\alpha+1}$.
\end{proof}

\begin{definition}[The Game $\mcG_1$]\label{games.d.g-one}
The rules of $\mcG_1$ are exactly the same as those of $\mcG_1^-$ with the
only difference that the ultrafilters $U_i$ played by Player~II are
required to be normal with respect to $N_{\alpha_i+1}$.

As before, the payoff set $R_\gamma$ is defined for $\mcG_1$ the same way as
it was for $\mcG_0$ and $\mcG_1^-$, that is, $R_\gamma$ consists of all runs
of $\mcG_1$ of length $\gamma$. For $\mcG_1$ we define a payoff set $W_\gamma$
as follows. 
\[
W_\gamma=
\parbox[t]{4in}{
the set of all $\langle \alpha_i,U_i\mid i<\gamma\rangle\in R_\gamma$ such
that if $\langle X_i\mid i<\gamma\rangle$ is a sequence satisfying
$X_i\in U_i$ for all $i<\gamma$ then
$\bigcap_{i<\gamma}X_i\neq\varnothing$.  
}
\]
\end{definition}

Notice that $W_\gamma=\varnothing$ whenever $\gamma\ge\kappa$, so the
game $\mcG_1(W_\gamma)$ is of interest only for $\gamma<\kappa$.
The existence of a winning strategy for Player~II in $\mcG_1(W_\omega)$ of
length $\omega$ seems to be exactly what is needed to run the proof of
precipitousness of the hopeless ideal $\mcI(\mcS^*)$ in Section~\ref{s.star}; see
Proposition~\ref{star.p.precipitous}. As we will see shortly, the existence of
such a winning strategy follows from the existence of a winning strategy for
Player~II in $\mcG_1^-(Q_\omega)$ of length $\omega$. 

In the case of
$\mcG_1$ we will not make use of a payoff set for $\mcG_1$ that would be an
analogue of what was $Q_\gamma$ for $\mcG_0$ and $\mcG_1^-$, so we will not
introduce it formally.  We note that $Q_\gamma$ is a subset of $W_\gamma$, so the winning condition 
for Player II is weaker using $W_\gamma$.

Let us also note that the somewhat abstract notion of normality of an
ultrafilter $U_i$ on $\mcA_i=\ptm(\kappa)\cap N_{\alpha_i+1}$ introduced in
Section~\ref{s.hopeless} is identical with the usual notion of normality with
respect to the model $N_{\alpha_i+1}$ where it is required that $U_i$ is
closed under diagonal intersections of sequences $\langle
A_\xi\mid\xi<\kappa\rangle\in N_{\alpha_i+1}$ such that $A_\xi\in U_i$ for all
$\xi<\kappa$.

{
\begin{remark}\label{allow induction} If we have a strategy $\mcS$ defined for either $\mcG^-_1$ or $\mcG_1$, then a play of the game according to this
strategy is determined by Player I's moves.  Thus, if $\mcS$  is clear from context we can save notation by referring to plays as sequences of ordinals $\la \alpha_i:i<\beta\ra$.  Similarly if $\mcS$ is a partial strategy defined on plays of length at most $\beta$ we can index these plays according to $\mcs$ by $\la \alpha_i:i<\beta^*\ra$, where $\beta^*\le \beta$.  This allows strategies to be defined by induction on the lengths of the plays.
\end{remark}}

\begin{proposition}\label{games.p.passing-to-normal}
(Passing to normal measures.) The following correspondences between the
existence of winning strategies for $\mcG_1^-$ and $\mcG_1$ hold. 
\begin{itemize}
\item[(a)] Let $\gamma\le\kappa^+$ be an infinite regular cardinal. If
Player~II has a 
winning strategy in $\mcG_1^-$ of length $\gamma$ then Player~II has a winning
strategy in $\mcG_1$ of length $\gamma$. (So in fact we have ``iff" here, as
the converse holds trivially.)
\item[(b)] {If Player II has a winning strategy in \hyperlink{Qgamma}{$\mcG_1^-(Q_\gamma)$} of length
$\gamma$ then Player~II has a winning strategy in $\mcG_1(W_\gamma)$ of length
$\gamma$. }
\end{itemize}
\end{proposition}

We do not know whether there
is an analogue of Proposition~\ref{games.p.passing-to-normal} with respect to
strategies for Player~I. 

\begin{proof}
We begin with some conventions and settings. Let $M_\alpha$ be the transitive
collapse of $N_\alpha$. We will work with models $M_\alpha$ in place of
$N_\alpha$. 

Since
$\kappa+1\subseteq N_\alpha$, we have
\[
\ptm(\kappa)^{N_\alpha}=\ptm(\kappa)\cap N_\alpha=\ptm(\kappa)\cap M_\alpha=
\ptm(\kappa)^{M_\alpha},
\]
so the games $\mcG_1^-$ and $\mcG_1$ can be equivalently defined using
structures $M_\alpha$ instead of $N_\alpha$.

If $U$ is an $M$-ultrafilter
over $\kappa$ we denote the internal ultrapower of $M$ by $U$ by $\ult(M,U)$. Then
$\ult(M,U)$ is formed using all functions $f:\kappa\to M$ which are
elements of $M$. If $U$ is $\kappa$-complete then $\ult(M,U)$ is well-founded,
and we will always consider it transitive; moreover the critical point of the
ultrapower map $\pi_U:M\to\ult(M,U)$ is precisely~$\kappa$. Recall also that
$U$ is normal if and only if $\kappa=[\idm]_U$, that is, $\kappa$ is
represented in the ultrapower by the identity map. As $M\models\zfc^-$ (by
$\zfc^-$ we mean $\zfc$ without the power set axiom), the \L o\'s Theorem
holds for all formulae, hence the ultrapower embedding $\pi_U$ is fully
elementary. Finally recall that the $M$-ultrafilter derived from $\pi_U$,
which we denote by $U^*$, is defined by
\begin{equation}\label{games.e.derived-normal}
X\in U^*
\;\Longleftrightarrow\;
\kappa\in\pi_U(X)
\end{equation}
and $U^*$ is normal with respect to $M$. 

Assume $\alpha<\alpha'$ and  $U'$ is a $\kappa$-complete
$M_{\alpha'}$-ultrafilter. Suppose 
that $U=U'\cap M_\alpha$. We have the following diagram:
\begin{equation}\label{games.diag.single-up}
\xymatrix{
M_{\alpha'} \ar[r]^{\pi_{U'}\quad\quad} & \ult(M_{\alpha'},U')\\
M_\alpha    \ar[r]_{\pi_U\quad\quad} \ar[u]^\sigma & \ult(M_\alpha,U) \ar[u]_{\sigma'}
}
\end{equation}
Here $\sigma:M_\alpha\to M_{\alpha'}$ is the natural map arising from
collapsing the inclusion map from  $N_\alpha$ to $N_{\alpha'}$, and $\sigma'$ is the natural
embedding of the ultrapowers defined by
\begin{equation}\label{games.e.up-embedding}
[f]_U\mapsto[\sigma(f)]_{U'}.
\end{equation}
Notice that $\crp(\sigma)=(\kappa^+)^{M_\alpha}$. 
Using the \L o\'s theorem, it is easy to check that the diagram is
commutative, $\sigma'$ is fully elementary, and
$\sigma'\rst\kappa=\idm\rst\kappa$. It follows that
\begin{equation}\label{games.e.sigma-prime-kappa}
\sigma'(\kappa)\ge\kappa.
\end{equation}
Given a set $X\in\ptm(\kappa)\cap M_\alpha$,
\begin{equation}\label{games.e.set-in-u-star}
X\in U^*
\;\Longleftrightarrow\;
\kappa\in\pi_U(X)
\;\Longleftrightarrow\;
\sigma'(\kappa)\in\sigma'(\pi_U(X))=\pi_{U'}(\sigma(X))=\pi_{U'}(X).
\end{equation}
Thus, using (\ref{games.e.sigma-prime-kappa}) combined with
(\ref{games.e.set-in-u-star}), 
\begin{equation}\label{games.e.distinct-normal}
U^*\not\subseteq(U')^*
\;\Longrightarrow\;
\sigma'(\kappa)>\kappa
\end{equation}

{ The property that $\sigma'(\kappa)>\kappa$ can be restated as saying that if $f$ represents $\kappa$ in $\ult(M_\alpha, U)$ and $g$ represents $\kappa$ in $\ult(M_{\alpha'},U')$, then $\{\delta: f(\delta)>g(\delta)\}\in U'$. Equation \ref{games.e.distinct-normal} can also be rephrased in this way.
} 

Notice that since $U^*,(U')^*$ are ultrafilters on the respective models,
the statement $U^*\subseteq(U')^*$ can be equivalently expressed as
$U^*=(U')^*\cap M_\alpha$.

\bigskip
 
 Before we define the winning strategies for Player II in $\mcG_1$, we prove two useful facts about the normalization process.  The first says there can't be an infinite sequence of ultrafilters that disagree on their normalizations.
 
	\begin{lemma}\label{doesn't happen} Let $\la\alpha_n:n\in\nn\ra$ be an increasing sequence of ordinals between 
 	$\kappa$ and $\kappa^+$.
	Then there is no sequence of ultrafilters $\la U_n:n\in\nn\ra$ such that
	\begin{itemize}
	\item $U_n$ is a $\kappa$-complete ultrafilter on $M_{\alpha_n+1}$
	\item $(U_n)^*\not\subseteq (U_{n+1})^*$
	\item there is a countably complete ultrafilter $V$ on $\bigcup_n(\ptm(\kappa)\cap M_{\alpha_n+1})$ with 
	$V\supseteq U_n$ for all $n$.
	\end{itemize}
	\end{lemma}
 	\begin{proof} For each $n$ let $f_n\in M_{\alpha_n+1}$ represent $\kappa$ in $\ult(M_{\alpha_n+1}, U_{n})$, and let 
	$\sigma'_n$ be the map from $\ult(M_{\alpha_n+1}, U_n)$ to $\ult(M_{\alpha_{n+1}+1}, U_{n+1})$   
	defined as in 
	equation \ref{games.e.up-embedding}.  Then there is a set $X_{n+1}\in U_{n+1}$ such that for all $\delta\in X_{n+1}, f_n(\delta)>f_{n+1}(\delta)$. The $X_n$'s all belong to $V$ and intersecting them we get a $\delta\in \kappa$ such that for all $n, f_{n}(\delta)>f_{n+1}(\delta)$, a contradiction. \end{proof}

 We note that Lemma \ref{doesn't happen} implies that in a play $\la \alpha_i, U_i:i<\gamma\ra$ there is no infinite increasing sequence $\la i_n:n\in\nn\ra$ such that $(U_{\alpha_{i_n}})^*\not\subseteq (U_{\alpha_{i_{n+1}}})^*$. 
 \medskip

{Let  $\la \alpha_i, U_i:i<\beta\ra$ be a partial or complete play of the game $\mcG_1^-$ of limit length $\beta$. Suppose that 
$N_\infty=\bigcup_{i<\beta}N_{\alpha_i+1}$. Then the transitive collapse of $N_\infty$ is the direct limit of 
$\la M_{\alpha_i+1}:i<\beta\ra$ along the canonical functions $\sigma_{j,j'}:M_{\alpha_j+1}\to M_{\alpha_{j'}+1}$ in diagram \ref{games.diag.single-up}. Denote the transitive collapse of $N_\infty$ by $M_\infty$. 
 Let $U_\infty$ be a $\kappa$-complete ultrafilter defined on the 
$\kappa$-algebra generated by $P(\kappa)\cap M_\infty$ 
that extends
$\bigcup_{i<\beta}U_i$. The next lemma implies that if for some $i<\beta, (U_i)^*\not\subseteq (U_\infty)^*$ then for some $j$ with $i < j <\beta$, $(U_i)^*\not\subseteq (U_j)^*$.}
 
  \begin{lemma}\label{not at limits}
 Let $M_\alpha\prec M_\beta \prec M_\gamma$   with  $\alpha, \beta, \gamma$ members of the 
 {$\alpha_i+1$'s.}  Let $U_\alpha\subseteq U_\beta\subseteq U_\gamma$ be $\kappa$-complete 
 ultrafilters {on the respective $\ptm(\kappa)$'s of} $M_\alpha, M_\beta, M_\gamma$.  Suppose that 
 $(U_\alpha)^*\not\subseteq (U_\gamma)^*$.  Let $X\in M_\alpha\cap \ptm(\kappa)$ be such that 
 $X\notin (U_\alpha)^*, X\in (U_\gamma)^*$.  
 
 Then we can choose $f^\alpha, g^\alpha, f^\gamma, g^\gamma$ such that 
	 \begin{eqnarray*}
	 f^\alpha, g^\alpha\in M_\alpha, &  [f^\alpha]_{U_\alpha}=\kappa, & [g^\alpha]_{U_\alpha}=X\\
	 f^\gamma, g^\gamma\in M_\gamma, & [f^\gamma]_{U_\gamma}=\kappa, & [g^\gamma]_{U_\gamma}=X
	\end{eqnarray*} 
Suppose that $f^\gamma, g^\gamma\in M_\beta$. Then $(U_\alpha)^*\not\subseteq (U_\beta)^*$.   \end{lemma}
 
 \begin{proof} {If any of $U_\alpha, U_\beta, U_\gamma$ are principal the hypothesis clearly fails.  It follows that each of the ultrafilters is uniform.}
 \smallskip
 
 The point of the proof is showing that if $f^\alpha, g^\alpha, f^\gamma, g^\gamma$ belong to $M_\beta$, then $X\in (U_\beta)^*$. Since $X\notin (U_\alpha)^*$ but does belong to $M_\alpha$, it follows that  $(\kappa\setminus X)\in U_\alpha^*$.  So $\kappa\setminus X$ witnesses the conclusion of the lemma.
 
 
Using  the notation of diagram \ref{games.diag.single-up}, since 
	\[ \sigma':\ult(M_\beta, U_\beta)\to \ult(M_\gamma, U_\gamma)\]
	is order preserving and $[f^\gamma]_{U_\gamma}=\kappa$, we must have $[f^\gamma]_{U_\beta}=\kappa$.

 	Since {$X\in (U_\gamma)^*$}, $[f^\gamma]_{U_\gamma}=\kappa$ and $[g^\gamma]_\gamma=X$, we must have 
	$\{\delta:f^\gamma(\delta)\in g^\gamma(\delta)\}\in U_\gamma$.  Since $f^\gamma$ and $g^\gamma$ belong to $M_\beta$ and $U_\beta\subseteq U_\gamma$ we have $\{\delta:f^\gamma(\delta)\in g^\gamma(\delta)\}\in U_\beta$, and hence $[g^\gamma]_{U_\beta}\in (U_\beta)^*$. 
	
	To finish it suffices to show that $[g^\gamma]_{U_\beta}=X$.  Since 
	$\{\delta:\sup(g^\gamma(\delta))=f^\gamma(\delta)\}\in U_\gamma$, we must have 
	$\{\delta:\sup(g^\gamma(\delta))=f^\gamma(\delta)\}\in U_\beta$. Thus $\sup([g]_{U_\beta})=\kappa$.
	
	For $\alpha<\kappa$, let  $c_\alpha:\kappa\to\kappa$ be the constant function $\alpha$. Then 
	$\{\delta:c_\alpha(\delta)<f^\gamma(\delta)\}\in U_\beta$, by $\kappa$-completeness. Using induction and the  
	$\kappa$-completeness of $U_\beta$, one proves that $[c_\alpha]_{U_\beta}=\alpha$.   
	But then
	\begin{eqnarray*}
	\alpha\in [g^\gamma]_{U_\beta} &\mbox{iff} &\{\delta:c_\alpha(\delta)\in g^\gamma(\delta)\}\in U_\beta\\
							& \mbox{iff}&\{\delta:c_\alpha(\delta)\in g^\gamma(\delta)\}\in U_\gamma\\
							&\mbox{iff} & \alpha\in [g^\gamma]_{U_\gamma}.
	\end{eqnarray*}
 	Since $[g^\gamma]_{U_\gamma}=X$ we have $[g^\gamma]_{U_\beta}=X$. \end{proof}

It follows from Lemmas \ref{doesn't happen} and \ref{not at limits} that if $\la \alpha_i, U_i:i<\beta+k\ra$ is a play of 
$\mcG_1^-$ 
where $\beta$ is zero or a limit ordinal and $k\in \omega$, then there is a finite set $i_0=0<i_1<i_2<\dots i_n=\beta$ such that {for all $1\le m<n$}
\begin{enumerate}
\item[A.)]{ for all  $i< j\in [i_{m-1}, i_{m})$, it holds that } $(U_i)^*\subseteq (U_j)^*$,
\item[B.)] for all $i\in [i_{m-1},i_m)$, $(U_i)^*\not\subseteq (U_{i_m})^*$.
\end{enumerate}
We will call the stages $i_1, \dots i_n$ together with $\{0\le j<k-1: (U_{\beta+j})^*\not\subseteq (U_{\beta+j+1})^*\}$ \emph{drops}.
{Note that in clause B.), $m<n$ so this does not imply that $\beta$ is a drop.}

\bigskip 
 
 A position $P$ of the game $\mcG_1^-$ has the form
\begin{equation}\label{games.e.run-g-one-minus}
P=\langle\alpha^P_i,U^P_i\mid i<\beta^P\rangle
\end{equation}
where $\alpha^P_i$ are moves of Player I and $U^P_i$ are moves of Player II,
and we will {not use} the superscripts ${}^P$ if there is no danger of confusion.  We will take 
$\beta=0$ as the length of the empty position.
Given an infinite regular cardinal $\gamma$ and a strategy $\mcS$ for
Player II in the game $\mcG_1^-$ of length $\ge\gamma$, let $\mcZ_\gamma$ be
the set of all positions in $\mcG_1^-$ of length $<\gamma$ according to
$\mcS$ that have successor length, where the last {move of Player I} is a drop. As stated in Remark \ref{allow induction} we can index plays in $\mcZ_\gamma$ by increasing sequences
of ordinals. On $\mcZ_\gamma$ we define a binary relation $\blacktriangleleft$ as
follows. Given two positions $P,Q\in\mcZ_\gamma$, we let 
\begin{equation}\label{games.e.relation-r}
P\;\blacktriangleleft\;Q
\end{equation}
if and only if $P$ properly extends $Q$.

\begin{claim}\label{games.c.well-foundedness}
Assume one of the following holds
\begin{itemize}
\item[(a)] $\gamma>\omega$ is regular and $\mcS$ is a winning strategy for Player~II in
$\mcG^-_1$ of length $\gamma$. 
\item[(b)] $\gamma=\omega$ and $\mcS$ is a winning strategy for Player~II in
$\mcG^-_1(Q_\gamma)$ of length $\gamma$.
\end{itemize}
Then $\blacktriangleleft$ is a well-founded tree. 
\end{claim}

\begin{proof} It is immediate that $\blacktriangleleft$ is a tree.  The well-foundedness follows  from the fact that 
there can be only finitely many drops along a play of the game.
\end{proof}
The proof of Claim \ref{games.c.well-foundedness}  implies that if $\mcS$ is a winning strategy in any of the variants of $\mcG^-_1$ of any length 
$\gamma$, then $\blacktriangleleft$ is well-founded. Note for well-foundedness the only relevant $\gamma$ are limit ordinals.   As stated, the Claim handles all of the cases relevant to the theorems we are proving.
\bigskip

\bigskip

 Now assume $\mcS$ is as in (a) or (b) in Claim~\ref{games.c.well-foundedness}. For $P\in \mcZ_\gamma$, 
let $i_P$ be the largest
drop in $P$ if $P$ does have a drop, and $i_P=0$ otherwise. Fix a $\blacktriangleleft$-minimal $P\in\mcZ_\gamma$.  By the minimality
of $P$, if $P'$ extends $P$ then $i_{P'}=i_P$; in other
words, $P'$ has no drops above $i_P$, hence
$(U^{P'}_i)^*\subseteq (U^{P'}_{i'})^*$
whenever $i_P\le i<i'$. 
\smallskip

\begin{quotation}
Let $\alpha^*=\alpha_{i_P}$ and $V^*=(U_{\alpha^*})^*$.
\end{quotation}

\noindent We define a winning strategy $\mcs_P$ for Player II in $\mcG_1$ of length $\gamma$. 
Viewing 
$\mcS$ as defined on sequences of ordinals $\la \alpha_i:i<\beta\ra$, we define $\mcS_P$ on such sequences $\la \alpha_i:i<\beta\ra$ by induction on their length $\beta$.

For ordinals $\alpha_i<\alpha^*$ played by the first player we assume inductively that the normal ultrafilter $V_{i}$ played by the second player is $(V)^*\cap N_{\alpha_i+1}$.

Suppose we have defined defined $\mcS_P$ on sequences of length less than $\beta$, where $\beta=0$ 
corresponds to the empty position. Formally, to $\la\alpha_i:i<\beta\ra$ we inductively associate the play 
$\la(\alpha_i, V_i):i<\beta\ra$ where $V_i$ is the response by Player II according to $\mcS_P$. We need to define 
$\mcS_P$ on 
$\la \alpha_i:i<\beta\ra \cat\alpha_\beta$. 

\begin{description}
	\item[Case 1] $\alpha_\beta\le \alpha^*$. In this case 
	\[\mcS_P(\la \alpha_i:i<\beta\ra \cat \alpha_\beta)= (V)^*\cap N_{\alpha_{\beta}+1} \]
	\item[Case 2] $\alpha_\beta> \alpha^*$.\\ Let $j$ be least such that $\alpha_j>\alpha^*$.
	Let 
	\[\mcs_P(\la \alpha_i:i<\beta\ra \cat \alpha_\beta\ra)=(\mcS(P\cat\la \alpha_i:j\le i\le \beta\ra))^*\]
\end{description}
Note that in Case 1, it is trivial that Player II's move is a legal move.  In Case 2, all of the filters played in response to ordinals less that $\alpha^*$ are sub-filters of $V^*$ and hence are legal plays and sub-filters of $\mcS(P)^*$. Going beyond $P$, the plays {of $\mcS_P$} are extensions of plays according to $\mcS$ that have initial segment $P$.  Since $P$ is $\blacktriangleleft$ minimal there are no drops for those plays--in other words, there is inclusion of the normalized responses according to $\mcS$. 
\smallskip

\noindent From this we conclude that Player II wins the game of length $\gamma$ in part (a) of Claim \ref{games.c.well-foundedness}. 

\medskip

{We only prove (b) for $\gamma=\omega$ because that is the most relevant case for this paper. 
 A straightforward
generalization of this argument gives the result for general $\gamma$. } The strategy
$\mcS_P$ is {defined using a } winning play by $\mcS$ in the game $\mcG_1^-(Q_\omega)$. Since $\mcS$ is a winning strategy in that game, if $\la (\alpha_n, U_n):n\in\nn\ra$ is that play according to $\mcS$, there is a $\kappa$-complete ultrafilter $U_\infty\supseteq \bigcup U_n$ defined on the $\kappa$-algebra generated by $\bigcup_nM_{\alpha_n}$.  By Lemmas \ref{doesn't happen} and \ref{not at limits} and the remarks preceding them, $(U_\infty)^*$ extends $V_n$ for all $n$.  Part (b) follows. 

\end{proof}

{\begin{remark}\label{definable 2} Arguing exactly as in Lemma \ref{definable 1}, if $\la \alpha_i:i<\beta\ra\in N_\alpha$ is a sequence of ordinals 
{and a $\blacktriangleleft$-minimal position
 position $P$ in the game $\mcG^-_1$} 
 belongs to $N_\alpha$ then  the sequence of responses by Player II
 to $\la \alpha_i:i<\beta\ra$ {using $\mcS_P$} belongs to $N_\alpha$. In particular if $\beta$ is a successor ordinal $j+1$ then $\mcS_P$'s responses belong to $N_{\alpha_j+2}$.
\end{remark}}


\begin{definition}[The Game $\mcG_2$]\label{games.d.g-two}
The rules of the game $\mcG_2$ are as follows.
\begin{itemize}
\item Player I plays an increasing sequence of ordinals $\alpha_i<\kappa^+$ as
before.
\item Player II plays distinct sets $Y_i\subseteq\kappa$ such that the following are
satisfied. 
   \begin{itemize}
   \item[(i)] $Y_j\subseteq^* Y_i$ whenever $i<j$, and
   \item[(ii)] Letting
     $U_i=\{X\in\ptm(\kappa)\cap N_{\alpha_i+1}\mid Y_i\subseteq^* X\}$,
     the family $U_i$ is a uniform normal ultrafilter on
     $\ptm(\kappa)\cap N_{\alpha_i+1}$.
  
   \end{itemize}
\item Player I goes first at limit stages.
\end{itemize}
A run of $\mcG_2$ of length $\gamma\le\kappa^+$ continues until Player II
cannot play or else until it reaches length $\gamma$.

Payoff sets $R_\gamma$ and $W_\gamma$ for the game $\mcG_2$ are defined
analogously to those for $\mcG_1$. So $R_\gamma$ consists of all runs 
in $\mcG_2$ of length $\gamma$ and $W_\gamma$ consists of all those runs
$\langle \alpha_i,Y_i\mid i<\gamma\rangle\in R_\gamma$ such that if
$\langle X_i\mid i<\gamma\rangle$ is a sequence satisfying
$X_i\in N_{\alpha_i+1}$ and $Y_i\subseteq^* X_i$
for all $i<\gamma$ then $\bigcap_{i<\gamma}X_i\neq\varnothing$.  
\end{definition}

Note that $Y_i\notin N_{\alpha_i+1}$ in (ii). Note also that since
the ultrafilters $U_i$ are required to be uniform, the sets $Y_i$ are
unbounded in $\kappa$. As with $\mcG_1$, we will not make any use of what
would be an analogue of payoff set $Q_\gamma$.

\begin{proposition}\label{games.p.g-one-e-g-two}
Assume $\gamma\le\kappa^+$ is an infinite regular cardinal.
\begin{itemize}
\item[(a)] Player II has a winning strategy in $\mcG_1$ of length $\gamma$ iff
Player~II has a winning strategy in $\mcG_2$ of length $\gamma$.
\item[(b)] Player~II has a winning strategy in $\mcG_1(W_\gamma)$ of length
$\gamma$ iff Player~II has a winning strategy in $\mcG_2(W_\gamma)$ of length
$\gamma$.
\end{itemize}
\end{proposition}

\begin{proof} 
For (a), it is immediate that a winning strategy for Player II in $\mcG_2$ gives a winning strategy in $\mcG_1$: if Player II plays $Y_i$ at turn $i$, then $Y_i$ generates a normal ultrafilter on $N_{\alpha_i+1}$ which is Player II's move in $\mcG_1$.

For the non-trivial direction, assume Player II has a winning
strategy $\mcS$ in $\mcG_1$ of length $\gamma$. As noted before Definition \ref{games.d.g-two}, such a strategy exists in $N_0$.   We build a winning strategy  $\mcS'$ for Player II in $\mcG_2$ of length~$\gamma$ by induction. 
\bigskip

\bfni{Induction Hypothesis}\ \ 
Suppose that Player I plays $\la \alpha_i:i<\beta\ra$ in the game $\mcG_2$, and $\la U_i:i<\beta\ra$ is the play by Player II using $\mcS$ in the game $\mcG_1$.  Then  Player II plays $\la Y_i:i<\beta\ra$ where $Y_i$ is a definable diagonal intersection of the members of  $U_i$.  
\medskip

For each $i$, let $\langle X^i_\xi\mid\xi<\kappa\rangle$ be the $<_\theta$-least  enumeration of $U_i$ of length 
$\kappa$
(recall that $<_\theta$ is the well-ordering of $H_\theta$ fixed at the
beginning of this section; see the paragraphs immediately above
Definition~\ref{games.d.g-one-minus}).  The induction hypothesis is that for all $i<\beta$, { Player II's responses according to the strategy $\mcS'$ to the sequence $\la \alpha_i:i \le \delta\ra$ are $\la Y_i:i\le \delta\ra$ where 
\[
Y_i=\Delta_{\xi<\kappa}X^i_\xi.
\]
}
This induction hypothesis is automatically preserved at limit stages.  Suppose that it holds up to $\beta$ and Player I plays $\alpha_\beta$. Then Player II plays an ultrafilter $U_\beta$ on 
$\ptm(\kappa)\cap N_{\alpha_\beta+1}$ in the game $\mcG_1$ using the strategy defined in Proposition 
\ref{games.p.passing-to-normal}.   Then, {as in Remark \ref{definable 2},} $N_{\alpha_\beta+2}$ contains the information that $U_\beta$ is 
Player II's response as well as the 
$<_\theta$-least enumeration $\la X^\beta_\xi:\xi<\kappa\ra$ of $U_\beta$. Let 
$Y_\beta=\Delta_{\xi<\kappa}X^\beta_\xi$ {and let $Y_\beta$ be} Player II's response in $\mcG_2$ using 
$\mcS'$.

Suppose now that $\la \alpha_i:i<\gamma\ra$ is a run of the game $\mcG_2$ according to $\mcS'$.  Then, since $U_{i+1}$ is normal each $Y_i$ belongs to $U_{i+1}$.  Since, $Y_j\subseteq^*X$ for all $X\in U_j$, for $i<j, Y_j\subseteq^*Y_i$. Moreover, since $Y_i$ is a diagonal intersection of the ultrafilter $U_i$, clause (ii) in  Definition \ref{games.d.g-two} is immediate. 

Since the relevant ultrafilters are the same, whether II is playing by $\mcS$ in $\mcG_1$ or $\mcS'$ in $\mcG_2$, clause (b) in  Proposition \ref{games.p.g-one-e-g-two} is immediate.
\end{proof}

{\begin{remark}\label{definable 3}
The definition of the winning strategy $\mcs'$ for Player II in the previous proof depends on the 
position $P$ in $\mcG_1^-$, beyond which there are no drops. 
Suppose that Player I plays $\la \alpha_i:i<\beta\ra$ in the game 
$\mathcal G_2$ and player II responds with $\la Y_i:i<\beta\ra$ using the winning strategy $\mcS'$. Then for all $j<\beta$ with $P\in N_j$,   
\begin{itemize}
\item $Y_j\notin N_{\alpha_j+1}$ because it induces an ultrafilter on $N_{\alpha_j+1}$,
\item $Y_j\in N_{\alpha_j+2}$ because $\la N_{\alpha_i}:i\le j\ra\in N_{\alpha_j+2}$ and Player II's response to 
$\la \alpha_i:i\le j\ra$ according to $\mcS'$ is definable from Player II's response to $\la \alpha_i:i\le j\ra$ according to the strategy
$\mcS$ for $\mcG_1$, which in turn is definable from $P$ and Player II's  response according to her strategy  in  
$\mcG_1^-$ and thus from the original strategy in $\mcG_0$. 
\end{itemize}
It follows that for all $i<j$,   $Y_{j}\subseteq^*Y_i$ and $|Y_i\setminus Y_j|=\kappa$. (Restating this $Y_j\subsetneq^*Y_i$.)
\end{remark}}

{
We complete this section with a corollary which will be used in studying
properties of the strategies constructed in Section~\ref{s.model}.
}

{
\begin{corollary}\label{games.c.ideals-g-one-g-two}
Assume $\mcS_1$ is a winning strategy for Player~II in the game $\mcG_1$ of
length $\gamma$ and $\mcS_2$ is the winning strategy for Player~II in the game
$\mcG_2$ of length $\gamma$ obtained as in
Proposition~\ref{games.p.g-one-e-g-two}. Then for every $A\in\ptm(\kappa)$,
\[
A\in\mcI(\mcS_1)
\quad\Longleftrightarrow\quad
A\in\mcI(\mcS_2).
\]
\end{corollary}
}


\section{Strategies $\mcS^*$ and $\mcS_\gamma$}\label{s.star}

Consider a winning strategy $\mcS_0$ for Player II in $\mcG_0$ of length
$\gamma$ and a position $P$ in $\mcG_0$ according to $\mcS_0$. Given a set
$X\in\mcI(\mcS_0,P)^+$, there may 
exist different runs of $\mcG_0$ extending $P$ which witness that $X$ is
$\mcI(\mcS_0,P)$-positive. This causes difficulties in proving that
$\mcI(\mcS_0,P)$ has strong properties like precipitousness or the existence
of a dense subset with a high degree of closure. To address this issue, we
construct 
a winning strategy $\mcS^*$ for Player~II in $\mcG_2$ of length $\gamma$ such
that for each position $Q$ in $\mcG_2$ according to $\mcS^*$ and each
$X\notin\mcI(\mcS^*,Q)$ there is a unique run witnessing that $X$ is
$\mcI(\mcS^*,Q)$-positive, and show that using $\mcS^*$ one can prove the
 precipitousness of {$\mcI(\mcS^*,\emptyset)$} and the existence of a
dense subset with a high degree of closure, thus proving Theorems~\ref{t1}
and~\ref{t2}.

Recall from the introduction that when we talk about saturated ideals over
$\kappa$, we always mean uniform $\kappa$-complete and $\kappa^+$-saturated
ideals over $\kappa$. The results in this section are formulated under the
assumption of the non-existence of a normal saturated ideal over $\kappa$, as
this allows to fit the results together smoothly. That the results actually
constitute a proof of Theorem~\ref{t2}, which is stated under a seemingly
stronger requirement on the non-existence of a saturated ideal over $\kappa$,
is a consequence of the following {standard} proposition.

\begin{proposition}\label{star.p.complete-vs-normal}
Given a regular cardinal $\kappa>\omega$, the following are equivalent.
\begin{itemize}
\item[(a)] $\kappa$ carries a saturated ideal.
\item[(b)] $\kappa$ carries a normal saturated ideal. 
\end{itemize}
\end{proposition}

\begin{proof}
A~standard elementary argument shows that any uniform normal ideal over
$\kappa$ is $\kappa$-complete, hence (a) follows immediately from (b).

To see that (b) follows from (a), assume $\mcI$ is a saturated ideal over
$\kappa$. Let $\mbbP_\mci$ be the partial ordering 
$(\mcI^+,\subseteq_\mci)$ and $\dot{U}$ be a $\mbbP_\mci$-term for
the normal $\mbfV$-ultrafilter over $\kappa$ derived from the generic
embedding $j_G:\mbfV\to M_G$ associated with $\ult(\mbfV,G)$ where $G$ is
$(\mbbP_\mci,\mbfV)$-generic. Let $\mcI^*\in\mbfV$ be the ideal over $\kappa$
defined by 
\[
a\in\mcI^*
\;\Longleftrightarrow\;\;\;
\forces^\mbfV_{\mbbP_\mci}\check{a}\notin\dot{U}.
\]
{Equivalently:
\[a\in \mcI^* \:\Longleftrightarrow\;\;\;
\forces^\mbfV_{\mbbP_\mci}\check{\kappa}\notin{j(\check{a})}.\]} A standard
argument shows that $\mcI^*$ is a uniform normal ideal over $\kappa$. To see
that $\mcI^*$ is saturated, we construct an incompatibility-preserving map
$e:(\mcI^*)^+\to\mcI^+$. Let $f:\kappa\to\kappa$ be a function in $\mbfV$
which represents $\kappa$ in $\ult(\mbfV,G)$ whenever $G$ is
$(\mbbP_\mci,\mbfV)$-generic. Since  $\mcI$ is saturated, such a function can
be constructed using standard techniques (see \cite{rvm}).
{Let 
\[e(a)\dfeq\{\xi<\kappa\mid f(\xi)\in a\}.\]}

Notice that for every $a\in\ptm(\kappa)^\mbfV$ and every
$(\mbbP_\mci,\mbfV)$-generic $G$, 
\[a\in\dot{U}^G
\;\Longleftrightarrow\;
\kappa\in j_G(a)
\;\Longleftrightarrow\;
[f]_G\in [c_a]_G
\;\Longleftrightarrow\;
e(a)\in G
\]
It follows from these equivalences that indeed $e(a)\in\mcI^+$ whenever
$a\in(\mcI^*)^+$. To see that $e$ is incompatibility preserving, we prove the
contrapositive. Assume $e(a),e(b)$ are compatible, so
$e(a)\cap e(b)\in\mcI^+$. Let $G$ be $(\mbbP_\mci,\mbfV)$-generic such that
$e(a)\cap e(b)\in G$. Then $e(a),e(b)\in G$, so $a,b\in\dot{U}^G$ by the above
equivalences. But then $a\cap b\in\dot{U}^G$, which tells us that
$a\cap b\in(\mcI^*)^+$. 
\end{proof}

We are now ready to formulate the main technical result of this section.

\begin{proposition}\label{star.p.tree}
Assume {$2^\kappa=\kappa^+$ and} there is no normal saturated ideal over $\kappa$. Let
$\gamma\le\kappa^+$ be an infinite regular cardinal and $\mcS$ be a winning
strategy for Player~II in $\mcG_2$ of length $\gamma$. Then there is a tree
$T(\mcS)$ which is a subtree of the partial ordering {$(\ptm(\kappa),\subseteq^*)$} such that
the following hold. 

\begin{itemize}
\item[(a)] The height of $T(\mcS)$ is $\gamma$ and $T(\mcS)$ is
\hyperlink{gammaclosed}{$\gamma$-closed.}

\item[(b)] If $Y,Y'\in T(\mcS)$ are $\subseteq^*$-incomparable then $Y,Y'$ are
almost disjoint. 
\item[(c)] There is an assignment $Y\mapsto P_Y$ assigning to each
$Y\in T(\mcS)$ a position $P_Y$ in $\mcG_2$ of successor length according to
$\mcS$ in which the
last move by Player~II is $Y$; we denote the last move of Player~I in
$P_Y$ by $\alpha(Y)$. The assignment $Y\mapsto P_Y$ has the following
property:
\[
Y'\subsetneq^* Y
\;\Longrightarrow\;
\mbox{ $\alpha(Y)<\alpha(Y')$ and $P_{Y'}$ is an extension of $P_Y$ .}
\]

\item[(d)] If $b$ is a branch of $T(\mcS)$ of length
$<\!\gamma$, let $P_b=\bigcup_{Y\in b}P_Y$. Then $P_b$ is a position in
$\mcG_2$ according to $\mcS$, and the set of all immediate successors of $b$
in $T(\mcS)$ is of cardinality $\kappa^+$.
Moreover the assignment $Y\mapsto\alpha(Y)$ is injective on this set. 
\end{itemize}

{
Finally, if $A\in\mcI(\mcS)^+$ then it is possible to construct the tree
$T(\mcS)$ in such a way that
\begin{equation}\label{star.e.a-in-t}
A\in T(\mcS)
\end{equation}
}
\end{proposition}

Clause (d) in the above definition treats both successor and limit cases for
{${\gamma}$}. The successor case in (d) simply says that if $Y\in T(\mcS)$
then the conclusions in (d) apply to the set of all immediate successors of
$Y$ in $T(\mcS)$.

\begin{proof}
The tree $T(\mcS)$ is constructed by induction on levels. Limit stages of this
construction are trivial: If $\bar{\gamma}<\gamma$ is a limit and we have 
already constructed initial segments $T_{\gamma^*}$ of $T(\mcS)$ of height
$\gamma^*$ for all $\gamma^*<\bar{\gamma}$ so that (b) -- (d) hold with
$T_{\gamma^*}$ in place of $T(\mcS)$ and $T_{\gamma'}$ end-extends
$T_{\gamma^*}$ whenever $\gamma^*<\gamma'<\bar{\gamma}$ then it is easy to see
that $T_{\bar{\gamma}}=\bigcup_{\gamma^*<\bar{\gamma}}T_{\gamma^*}$ is a tree
with tree ordering $\supseteq^*$ end-extending all $T_{\gamma^*}$,
$\gamma^*<\bar{\gamma}$, and such that (b) -- (d) hold with $T_{\bar{\gamma}}$
in place of $T(\mcS)$. We will thus focus on the successor stages of the
construction. 

Assume $\bar{\gamma}<\gamma$ and $T(\mcS)$ is constructed at all levels up to
level $\bar{\gamma}$; our task now is to construct the $\bar{\gamma}$-th level
of $T(\mcS)$. Let $b$ be a cofinal branch through this initial segment of
$T(\mcS)$, so $b$ is of length $\bar{\gamma}$. We construct the set of immediate
successors of $b$ in $T(\mcS)$, along with the assignment $Y\mapsto P_Y$ on this
set, as follows. As we are assuming there is no normal
saturated ideal over $\kappa$, we can pick an antichain $\mcA$ in
$\mcI(\mcS,P_b)^+$ of cardinality $\kappa^+$. For each $X\in\mcA$ there is a
position $Q_X$ in $\mcG_2$ of successor length $<\gamma$ according to $\mcS$
extending $P_b$ such that the last move by Player~II in $Q_X$ is almost
contained in $X$. For the sake of definability we can let this position to be
$<_\theta$-least, where recall that $<_\theta$ is the \hyperlink{Fixisin}{{fixed} well-ordering of
$H_\theta$.} 

Now construct the set $\langle Y_\xi\mid\xi<\kappa^+\rangle$ of all immediate
successors of $b$ in $T(\mcS)$ recursively as follows. Assume $\xi<\kappa^+$
and we have already constructed the set
$\langle Y_{\bar{\xi}}\mid\bar{\xi}<\xi\rangle$ along with the
assignment $Y_{\bar{\xi}}\mapsto P_{Y_{\bar{\xi}}}$ with the desired
properties. Since each model $N_\beta$ is of cardinality $\kappa$, we can pick
{the $<_\theta$-least set} $X\in\mcA$ which is not an element of any $N_{\alpha(Y_{\bar{\xi}})+1}$
where $\bar{\xi}<\xi$. Now let Player~I extend $Q_X$ by playing {the least} ordinal
$\alpha$ such that 
\begin{equation}\label{star.e.alpha-y}
\{X\}\cup\{Y_{\bar{\xi}}\mid\bar{\xi}<\xi\}\subseteq N_{\alpha+1}.
\end{equation}
This is a legal move in $\mcG_2$ following $Q_X$. Let $Y$
be the response of the strategy $\mcS$ to
${Q_X}^\smallfrown\langle\alpha\rangle$.
We let $Y_\xi$ be this $Y$ and
$P_Y={Q_X}^\smallfrown\langle\alpha,Y\rangle$. Notice that $Y_\xi\subseteq^*X$,
as $Y_\xi$, being played according to $\mcS$, is almost contained in the last
move by Player~II in $Q_X$.

 We show:
\begin{equation}\label{star.e.last-level}
\mbox{Any two sets $Y\neq Y'$ on the $\bar{\gamma}$-th level are almost
disjoint.} 
\end{equation}
If $Y,Y'$ are above two distinct cofinal branches 
then this follows immediately from the induction hypothesis: Letting $Z$,
resp.\ $Z'$ be the immediate successor of $b\cap b'$ in $b$, resp.\ $b'$, we
have $Y\subseteq^*Z$ and $Y'\subseteq^*Z'$, and the induction hypothesis tells
us that $Z,Z'$ are almost disjoint. 

Now assume $Y,Y'$ are above the same branch
$b$; without loss of generality we may assume $Y=Y_\xi$ and $Y'=Y_{\xi'}$ in
the above enumeration and $\xi'<\xi$. Then we have $X, X', P_Y, P_{Y'}$ as in
the construction, with $Y\subseteq^* X$ and $Y'\subseteq^*X'$. Also
$\alpha(Y')<\alpha(Y)$.

 If $Y\subseteq^*Y'$ then
$Y\subseteq^* X\cap X'$, thus witnessing $X\cap X'\in\mcI(\mcS,P_b)^+$. This
contradicts the fact that $\mcA$ is an antichain in $\mcI(\mcS,P_b)^+$. It
follows that $Y\not\subseteq^*Y'$. Now for every $Z\in N_{\alpha(Y)+1}$ the set
$Y$ is either almost contained in or  almost disjoint from $Z$. As
$Y'\in N_{\alpha(Y)+1}$ by our choice of $\alpha(Y)$ in (\ref{star.e.alpha-y}),
necessarily $Y$ is almost disjoint from $Y'$. This proves
(\ref{star.e.last-level}).  

To verify that (b) -- (d) hold with the tree obtained by adding the immediate
successors of a single branch $b$ as described in the previous paragraph in
place of $T(\mcS)$, notice that~(c)
and~(d) immediately follow from the construction just described, so all we
need to check is clause~(b) and the fact that $\supseteq^*$ is still a tree
ordering after adding the entire $\bar{\gamma}$-th level. But clause~(b)
follows from the combination of (\ref{star.e.last-level}) with the induction
hypothesis and the fact that every set on the $\bar{\gamma}$-th level is
almost contained in some set on an earlier level. Finally, that adding the
$\bar{\gamma}$-th level keeps $\supseteq^*$ a tree ordering follows from
clause~(b). More generally, any collection $\mcX\subseteq\ptm(\kappa)$ which
satisfies~(b) with $\mcX$ in place of $T(\mcS)$ has the property that the set
of all $Y'\in\mcX$ which are 
$\supseteq^*$-predecessors of a set $Y\in\mcX$ is linearly ordered under
$\supseteq^*$. What now remains is to see that clause~(a) holds, but this is
immediate once we have completed all $\gamma$ steps of the construction. 

{
Finally, given a set $A\in\mcI(S)^+$, to see that we can construct the tree
$T(\mcS)$ so that (\ref{star.e.a-in-t}) holds, notice that we can put $A$ into
the first level of $T(\mcS)$ at the first step in the inductive
construction. This involves a slight modification of the construction of the
first level of $T(\mcS)$, and is left to the reader.}
\end{proof}

The new strategy $\mcS^*$ for Player~II in $\mcG_2$ is now obtained by,
roughly speaking, playing down the tree $T(\mcS)$. More precisely:

\begin{definition}\label{star.d.strategy}
Assume $\gamma\le\kappa^+$ is an infinite regular cardinal, $\mcS$ is a
winning strategy for Player~II in $\mcG_2$ of length $\gamma$, and $T(\mcS)$
is a subtree of the partial ordering $(\ptm(\kappa),\supseteq^*)$ satisfying (a) -- (d) in
Proposition~\ref{star.p.tree}. We define a strategy $\mcS^*$ for
Player~II in $\mcG_2$ of length $\gamma$ associated with $T(\mcS)$ recursively
as follows.

Assume
\[
P=\{(\alpha_i,Y_i)\mid i<j\}
\]
is a position in $\mcG_2$ of length $j<\gamma$ according to $\mcS^*$. Denote
the corresponding branch in $T(\mcS)$ by $b_P$, that is,
\[
b_P=\{Y_i\mid i<j\}.
\]
If $\alpha_j$ is a legal move of Player~I in $\mcG_2$ at position $P$ then
\[
\mcS^*(P^\smallfrown\langle\alpha_j\rangle)=
\parbox[t]{3in}{
the unique immediate successor $Y$ of $b_P$ in $T(\mcS)$ with
minimal possible $\alpha(Y)\ge\alpha_j$.} 
\]
Here recall that $\alpha(Y)$ is the last move of Player~I in $P_Y$. 
\end{definition}

As an immediate consequence of the properties of $T(\mcS)$ we obtain:

\begin{proposition}\label{star.p.star-winning}
Let $\gamma\le\kappa^+$ be an infinite regular cardinal and assume $T(\mcS)$
is as in Proposition~\ref{star.p.tree}. Then 
$\mcS^*$ is a winning strategy for Player~II in $\mcG_2$ of length~$\gamma$. 

Moreover, if
\[
r^*=\langle\alpha_i,Y_i\mid i<\gamma\rangle
\]
is a run of $\mcG_2$ of length $\gamma$ according to $\mcS^*$ then
\[
r=\bigcup_{i<\gamma}P_{Y_i}
\]
is a run of $\mcG_2$ of length $\gamma$ according to $\mcS$. 
\end{proposition}

{
Before giving a proof of Theorem~\ref{t1}, we record the following obvious
fact, which will be useful in Section~\ref{s.model} in studying properties of
winning strategies for Player~II in games $\mcG_i$ of length $\gamma$, and
to which we will refer later.}

{
\begin{corollary}\label{model.c.sets-in-i-s-star}
Under the assumptions of Proposition~\ref{star.p.tree}, assume
$A\in\mcI(\mcS)^+$ and $T(\mcS)$ is constructed in such a way that
(\ref{star.e.a-in-t}) holds, that is, $A\in T(\mcS)$. Let $\mcS^*$ be the 
winning strategy for Player~II constructed as in
Definition~\ref{star.d.strategy} using this $T(\mcS)$. Then
$A\in\mcI(\mcS^*)^+$.   
\end{corollary}
}

{One of the main points of passing to $\mcS^*$ is the following remark.
\begin{remark}
For any position $P$ of a partial run according to $\mcS^*$ of successor length with $Y$ being the last move by Player II, the \hyperlink{hopeless}{conditional hopeless ideal} $\mcI(\mcS^*,P)$ is equal to the unconditional hopeless ideal restricted to $Y$:
\[\mcI(\mcS^*, P)=\mcI(S^*)\rest Y.\]
\end{remark}
}

We now turn to a proof of Theorem~\ref{t1}. If there is a 
normal saturated ideal over $\kappa$ then there is nothing to
prove. Otherwise Player~II has a winning strategy in $\mcG_2(W_\omega)$ of
length $\omega$, as follows from
Propositions~\ref{games.p.g-zero-g-one-minus}(b),
\ref{games.p.passing-to-normal}(b) and~\ref{games.p.g-one-e-g-two}(b). The
conclusion in Theorem~\ref{t1} then follows from a more specific fact we
prove, namely  Proposition~\ref{star.p.precipitous} below. In the proof of
this proposition we will make use of the criterion for
precipitousness in terms of the ideal game, see Section~\ref{s.intro}.

\begin{proposition}\label{star.p.precipitous}
Assume there is no normal saturated ideal over $\kappa$. Let 
\begin{itemize}
\item $\mcS$ be a winning strategy for Player~II in $\mcG_2(W_\omega)$ of
length $\omega$, and 
\item $\mcS^*$ be the winning strategy constructed from $\mcS$ as in
Definition~\ref{star.d.strategy}. 
\end{itemize}
Then Player~I does not have a winning strategy in the ideal game
$\mcG(\mcI(\mcS^*))$. Consequently, the ideal $\mcI(\mcS^*)$ is precipitous. 
\end{proposition}

\begin{proof}
Assume $\mcS_\mcI$ is a strategy for Player~I in the ideal game
$\mcG(\mcI(\mcS^*))$. We construct a run in $\mcG(\mcI(\mcS^*))$ according to
$\mcS_\mcI$ which is winning for Player~II.
Odd stages in this run will come from positions in $\mcG_2$ played according
to $\mcS^*$; more precisely, they will be tail-ends of sets on those
positions. So suppose 
\[
Q=\langle X_0,X_1,X_2,X_3\dots,X_{2n-1}\rangle
\]
is the finite run of $\mcG(\mcI(\mcS^*))$ constructed so far, and 
\[
\beta_0,Z_0,\beta_1,Z_1,\cdots\beta_{n-1},Z_{n-1}
\]
is the associated auxiliary run
of $\mcG_2$ according to $\mcS^*$ such that $Z_i\subseteq^*X_{2i}$ and
\[
X_{2i+1}=\mbox{the longest tail-end of $Z_i$ that is contained in $X_{2i}$}
\]
for all $i<n$. 
Let $X_{2n}$ be the response of $\mcS_\mcI$ to $Q$ in
$\mcG(\mcI(\mcS^*))$. As $X_{2n}\in\mcI(\mcS^*)^+$, there is a finite position
in $\mcG_2$ according to $\mcS^*$ where the last move of Player~II is a
set almost contained in $X_{2n}$ and, letting $Z_n$ be this set, we also have
$X_{2n}\in N_{\alpha(Z_n)+1}$. 

As the sets $Z_n$ constitute an $\subseteq^*$-decreasing chain of nodes in
$T(\mcS)$, the  positions $P_{Z_n}$ extend $P_{Z_m}$ whenever $m<n$. 
By Proposition~\ref{star.p.star-winning}

\[
r=\bigcup_{n\in\omega}P_{Z_n}.
\]
is a run in $\mcG_2$ of length $\omega$ according to $\mcS$. Let 
\[
r=\langle\alpha_i,Y_i\mid i\in\omega\rangle
\] be this run.
For each $i\in\omega$ let
\[
X'_i=X_{2n}
\mbox{ where $n$ is such that $\lht(P_{Z_n})\le i<\lht(P_{Z_{n+1}})$.}
\]
Then
\[
\bigcap_{n\in\omega}X_n=\bigcap_{n\in\omega}X_{2n}=\bigcap_{i\in\omega} X'_i
\neq\varnothing. 
\]
Here the equality on the left comes from the fact that the sets $X_n$,
$n\in\omega$ constitute an $\subseteq$-descending chain, and the inequality on
the right follows from the fact that $X'_i\in N_{\alpha_i+1}$ and
$Y_i\subseteq^* X'_i$ for all $i\in\omega$, and that $\mcS$ is a winning
strategy for Player~II in $\mcG_2(W_\omega)$ of length $\omega$; see the last
paragraph in Definition~\ref{games.d.g-two}. 
\end{proof}
\noindent We remark that the proof of Proposition \ref{star.p.precipitous} shows Player II has a winning strategy in the ideal game $\mcI(\mcS^*)$.

\bigskip
The following proposition gives a proof of Theorem~\ref{t2}. Recall that all
background we have developed so far was under the assumption that $\kappa$ is
inaccessible and $2^\kappa=\kappa^+$. Also recall that by trivial observation
(TO3) at the beginning of Section~\ref{s.games} and results in
Section~\ref{s.games}, if Player~II has a winning strategy in $\mcG_0$ of
length $\gamma>\omega$ then Player~II has a winning strategy in
$\mcG_0(Q_\omega)$ of length $\omega$ and in $\mcG_2(W_\omega)$ of length
$\omega$, as well as in $\mcG_2$ of length $\gamma$ whenever $\gamma$ is
regular. By a similar argument, if Player~II has a winning strategy in
$\mcG_2$ of length $\gamma>\omega$ then Player~II has a winning strategy in
$\mcG_2(W_\omega)$ of length $\omega$.  Thus, under the assumptions of
Theorem~\ref{t2}, the assumptions of Proposition~\ref{star.p.dense} below are
not vacuous.

\begin{proposition}\label{star.p.dense}
Assume there is no normal saturated ideal over~$\kappa$ and $2^\kappa=\kappa^+$.
Let $\gamma\le\kappa^+$ be an uncountable regular cardinal. Assume further
that $\mcS$ and $\mcS^*$ are strategies as in 
Proposition~\ref{star.p.precipitous}, with $\gamma$ in place of $\omega$.

{Then $T(\mcS)$ is a $\gamma$-closed dense subset of
$\mcI(\mcS^*)^+$.  It follows that Player~I does not have a winning strategy in the ideal game
$\mcG(\mcI(\mcS^*))$. Consequently, the ideal $\mcI(\mcS^*)$ is
precipitous.}
\end{proposition}

\begin{proof}

That $T(\mcS)$ is a $\gamma$-closed dense subset of $\mcI(\mcS^*)^+$ follows
immediately from the properties of $T(\mcS)$. If $A\in \mcI(\mcS^*)^+$, then there is a play of the game such that
$A$ is in the ultrafilter determined by some $Y_{\xi}$ played by Player II using $\mcS^*$. But then $Y_{\xi}\subseteq^*A$.  Since $Y_{\xi}$ is on $T(\mcS)$, we have shown that for every element of $\mcI(\mcS^*)^+$ there is an element of the tree below it.  Hence the tree is dense.

To see that $\mcI(\mcS^*)$ is precipitous, we use an argument originally due
to Laver. It follows the idea of Proposition \ref{star.p.precipitous} and  
shows that Player II has a winning strategy in the game $\mcG(\mcI(\mcS^*))$. At stage $n$ of the game 
suppose that Player I plays $X_{2n}$.   Player II chooses an $X'_{2n+1}\in T(\mcS)$ {(so 
$X'_{2n+1}\in \mcI(\mcS^*)^+$)} and $X'_{2n+1}\subseteq^*_{\mcI(S^*)}X_{2n}$.

{Let $A_n\in \mcI(S^*)$ be such that $X'_{2n+1}\setminus A_n\subset X_{2n}$.  Player II's response to $X_{2n}$ in $\mcG(\mcI(S^*))$ is $X_{2n+1}\dfeq X'_{2n+1}\setminus A_n$.  Let $A=\bigcup_nA_n$. Since 
$\mcI(S^*)$ is countably complete, $A\in \mcI(S^*)$. Let $X_\infty\in T(S)$, with $X_\infty\subseteq^*X'_{2n+1}$ for all $n$. Then:

{
\begin{equation*}
\bigcap_n X_n\ \ \supseteq\ \ \bigcap_n X_n\setminus A
\ \ \supseteq^*_{\mcI(\mcS^*)}\ \ X_\infty\setminus A.
\end{equation*}}

It follows that there is a set $B\in \mcI(S^*)$ such that $\bigcap X_n\supseteq X_\infty \setminus B$. Since $X_\infty\notin\mcI(S^*)$, $X_\infty \setminus B$ is not empty.  Hence
$\bigcap X_n\ne \emptyset$.}
\end{proof}

\bfni{Proof of Theorem~\ref{t4}}
\begin{proof}
Consider a uniform $\kappa$-complete ideal $\mcJ$ over $\kappa$ such that
$\ptm(\kappa)/\mcJ$ is $(\kappa^+,\infty)$-distributive and has a dense
$\gamma$-closed set. Because of notational convenience we will work with the
partial ordering $\mbbP_\mcJ=(\mcJ^+,\subseteq_\mcJ)$. (See also the partial ordering $\mbbP_\mcI$
used in the proof of Proposition~\ref{star.p.complete-vs-normal}.) Since
$a\mapsto[a]_\mcJ$ is a dense embedding of $\mbbP_\mcJ$ onto
$\ptm(\kappa)/\mcJ$, we can fix a dense $\gamma$-closed set
$D\subseteq\mbbP_\mcJ$. We work inside $H_\theta$ for a
sufficiently large $\theta$ and will use the
\hyperlink{Fixisin}{fixed well-ordering $<_\theta$ introduced in
Section~\ref{s.games}} to define a winning strategy $\mcS_\gamma$ for
Player~II in $\mcG^W_\gamma$. As usual, $\mcS_\gamma$ is defined inductively
on the length of runs.  

So assume
\[
\mcA_0,U_0,\mcA_1,U_1,\dots,\mcA_j,U_j,\dots
\]
is a run of $\mcG^W_\gamma$ according to $\mcS_\gamma$ for $j<i$. Along
the way, we define auxiliary moves $X_j$ played by Player~II; these moves are
elements of $D$, constitute a descending chain in the ordering by
$\subseteq_\mcJ$, and for each $j<i$,
\begin{equation}\label{star.e.x-j}
X_j\forces_{\mbbP_{\mcJ}}\dot{G}\cap\check{\mcA}_j=\check{U}_j.
\end{equation}
At step $i<\gamma$ Player~I plays a $\kappa$-algebra $\mcA_i$ on $\kappa$ of
cardinality $\kappa$ extending all $\mcA_j$, $j<i$. As $D$ is $\gamma$-closed
and $i<\gamma$, there is an element $X\in D$ below all $X_j$ in $\mbbP_\mcJ$,
$j<i$. If $G$ is a $(\mbbP_{\mcJ},\mbfV)$-generic filter such that $X\in G$
then by (\ref{star.e.x-j}), $U_j\subseteq G$ whenever $j<i$. Since
$\mbbP_{\mcJ}$ is $(\kappa^+,\infty)$-distributive and $\mcA_i\in\mbfV$
is of cardinality $\kappa$, the intersection $G\cap\mcA_i$ is an element of
$\mbfV$, and is a uniform $\kappa$-complete ultrafilter on $\mcA_i$ extending
all $U_j$ where $j<i$. This is then forced by some condition $Y\in G$ such
that $Y\subseteq_\mcJ X$, hence $Y\subseteq_\mcJ X_j$ for all $j<i$. As $D$ is
dense in $\mbbP_{\mcJ}$, $Y$ can be chosen to be an element of $D$. The
following is thus not vacuous. We define 
\[
X_i=\parbox[t]{3.5in}{
the $<_\theta$-least element $Y$ of $D$ such that $Y\subseteq_\mcJ X_j$ for all
$j<i$ and there is a $U\in\mbfV$ satisfying  
$Y\forces_{\mbbP_{\mcJ}}\dot{G}\cap\check{\mcA}_i=\check{U}$
}
\]
and
\[
U_i=\mbox{the unique $U\in\mbfV$ such that
$X_i\forces_{\mbbP_{\mcJ}}\dot{G}\cap\check{\mcA}_i=\check{U}$.}
\]
Letting
\[
\mcS_\gamma
(\langle\mcA_j,U_j\mid j<i\rangle^\smallfrown\langle\mcA_i\rangle)
=U_i,
\]
it is straightforward to verify that $\mcS_\gamma$ is a winning strategy
for Player~II in $\mcG^W_\gamma$. 
\end{proof}


\section{The Model}\label{s.model}

In this section we give a construction of a model where the following holds. 
\begin{equation}\label{model.e.inaccessible}
\mbox{$\kappa$ is inaccessible and carries no saturated ideals}
\end{equation}
and
\begin{equation}\label{model.e.ideal-existence}
\parbox{4in}
{For every regular uncountable $\gamma\le\kappa$ there is an ideal
$\mcJ_\gamma$ on $\ptm(\kappa)$ as in Theorem~\ref{t3}, that is, $\mcJ_\gamma$
is uniform, normal, $\gamma$-densely treed and
$(\kappa^+,\infty)$-distributive.}  
\end{equation}
The model is a forcing extension of a universe $\mbfV$ in which the following
{are} satisfied. 
\begin{itemize}
\item[(A)] $\gch$.
\item[(B)] \hypertarget{B}{$U$ is a normal measure} on $\kappa$.
\item[(C)] \hypertarget{(C)}{$\langle T_{\alpha,\xi}\mid\xi<\alpha^+\rangle$
           is a disjoint 
           sequence of stationary subsets of $\alpha^+\cap\cff(\alpha)$
	   whenever $\alpha\le\kappa$ is inaccessible. }
\item[(D)]\hypertarget{(D)}{ Assume $\mbfV[K]$ is a generic extension via a set-size
	   forcing which preserves $\kappa^+$}, and, {
	   in  $\mbfV[K]$ 
	     \begin{itemize}
	     \item there is a definable  class elementary embedding  $j':\mbfV\to M'$ where $M'$
	               is transitive, and  
	     \item Letting 
	     \begin{eqnarray*}\langle\langle
		T'_{\alpha,\xi}\mid\xi<\alpha^+\rangle\mid\alpha\le j'(\kappa)
	           \mbox{ is inaccessible in }M'\rangle\rangle =\\
                   j'(\langle\langle T_{\alpha,\xi}\mid\xi<\alpha^+\rangle\mid
                   \alpha\le\kappa\mbox{ is inaccessible}\rangle)
		\end{eqnarray*}
	      $\mbfV,M'$ agree on what $H_{\kappa^+}$ is and
	     $T'_{\kappa,\xi}=T_{\kappa,\xi}$ whenever $\xi<\kappa^+$. 
	     \end{itemize}}
\end{itemize}

We will informally explain the purpose of the sets $T_{\alpha,\xi}$ before we
begin with the construction of the model. These sets are not needed for the
construction of ideals $\mcJ_\gamma$ in Theorem~\ref{t3}, but only for the
proof that $\kappa$ does not carry a saturated ideal in our model. To
understand this proof, it suffices to accept (D) as a black box, that is, it
is not necessary to understand how the system of sets $T_{\alpha,\xi}$ is
constructed. 

Proper class models satisfying (A) -- (D) are known to exist, and can be produced via the so-called background certified constructors. The two most used background certified constructions are $K^c$-constructions and fully background certified constructions
If 
there is a proper class inner model with a measurable cardinal then any
$\Kc$-construction (see for instance \cite{oimt} for $\Kc$-constructions of
models with Mitchell-Steel indexing of extenders, and \cite{imlc} for
$\Kc$-constructions with Jensen's $\lambda$-indexing) performed
inside such a model gives rise to a fine structural proper class model
satisfying (A) -- (D). We will sketch a proof of this fact below in
Proposition~\ref{model.p.sets-t}. Similar conclusions are true of fully background certified constructions, but one needs to assume that a measurable cardinal exists in $\mathbf{V}$.

There is some similarity in the argument in Proposition~\ref{model.p.sets-t} of
the existence of a sequence of mutually disjoint stationary subsets
$T_{\kappa,\xi}$ of $\kappa^+$ which behave nicely with 
respect to the ultrapower by a normal ultrafilter on $\kappa$ to a similar {claim
in} \cite{nnuf} where it is proved that one can have such sequence of
stationary sets in $\mbfL[U]$.

A background certified construction as above gives rise to a model of the form
$\mbfL[E]$ where $E=\langle E_\alpha\mid\alpha\in\On\rangle$ is such that each
$E_\alpha$ either codes an extender in a way made precise, or 
$E_\alpha=\varnothing$. Additionally, a model of this kind admits a detailed
fine structure theory. There is an entire family of such models, so called
fine structural models; the internal first order theory of these models is
essentially the same, up to the large cardinal axioms.  There are $L[E]$ models with the properties
needed for the construction in this paper that satisfy the statement:

\begin{quotation}
\noindent There is a Woodin cardinal $\kappa$ that is a limit of Woodin cardinals,
\end{quotation}
as is shown in \cite{wlw}.

 We now list some
notation, terminology and general facts which will be used for the proof of
\hyperlink{(C)}{(C)} and \hyperlink{(D)}{(D)}. Clauses (A) and (B) follow from 
the construction of the {$L[E]$} model,
and their proofs can be found in \cite{oimt} or \cite{imlc}. In fact, each
proper initial segment of the model is acceptable in the sense of fine
structure theory. We omit the technical definition here and merely say that
acceptability is a local form of $\gch$, and is proved along the way  
the model is constructed.

From now on assume $W=\mbfL[E]$ is a fine structural extender model with indexing of extenders 
as in \cite{oimt} or in  Chapter~9 of \cite{imlc} (which covers all $\mathbf{L}[E]$ models discussed above). We often write
$E^W$ in place of $E$ to emphasize that $E$ is the extender sequence of $W$.   
\begin{itemize}
\item[FS1] $W \cut\alpha$ is the initial segment of $W$ of height $\omega\alpha$
with the top predicate, that is, $W\cut\alpha=(J^E_\alpha,E_{\omega\alpha})$.
\item[FS2] If $\alpha$ is a cardinal of $W$ then $E_\alpha=\varnothing$. Thus,
in this case $W\cut\alpha=(J^E_\alpha,\varnothing)$ and we identify this
structure with $J^E_\alpha$.
\item[FS3] If $\mu$ is a cardinal of $W$ then the structure $W\cut\mu$
calculates all cardinals and cofinalities $\le\mu$ the same way as $W$. This
is a consequence of acceptability.
\item[FS4] $\beta(\tau)$ is the unique $\beta$ such that $\tau$ is a cardinal
in $W\cut\beta$ but not in $W\cut(\beta+1)$.
\item[FS5] $\varrho^1$ stands for the first projectum; that
$\varrho^1(W\cut\beta)\le\alpha$ is equivalent to saying that there is a
surjective partial map $f:\alpha\to J^E_\beta$ which is $\Sigma_1$-definable
over $W\cut\beta$ with parameters.
\item[FS6] (Coherence.) If $i:W\to W'$ is a $\Sigma_1$-preserving map in
possibly some outer universe of $W$ such that $\kappa$ is the critical point
of $i$ and $\tau=(\kappa^+)^W$ then $E^{W'}\rst\tau=E^W\rst\tau$.
\item[FS7] (Cores.) Assume $\alpha$ is a cardinal in $W$ and $N$ is a
structure such that $\varrho^1(N)=\alpha$ and there is a $\Sigma_1$-preserving
map $\pi$ of $N$ into a level of $W$ such that $\pi\rst\alpha=\idm$. Let $p_N$
be the $<^*$-least finite set of ordinals $p$ such that there is a set of
ordinals $a$ which is $\Sigma_1(N)$-definable in the parameter $p$ and 
satisfies $a\cap\alpha\notin N$. Here
$<^*$ is the usual well-ordering of finite sets of ordinals, that is, finite
sets of ordinals are viewed as descending sequences and $<^*$ is the
lexicographical ordering of these sequences. Let $X$ be the $\Sigma_1$-hull of
$\alpha\cup\{p_N\}$ and $\sigma:\bar{N}\to N$ be the inverse of the collapsing
isomorphism. Then $\rho^1(\bar{N})=\alpha$, the models $\bar{N},N$ agree on
what $\ptm(\alpha)$ is, and $\pi$ is $\Sigma_1$-preserving and maps $\bar{N}$
cofinally into $N$. In this situation, $\bar{N}$ is called the core of $N$ and
$\sigma$ is called the core map.
\item[FS8] (Condensation lemma.) Assume $\alpha$ is a cardinal in $W$ and
$N,\bar{N},\pi$ and $\sigma$ are as in FS7. Then $\bar{N}$ is a level of $W$,
that is, $\bar{N}=W\cut\beta$ for some $\beta$.  
\end{itemize}

\begin{proposition}\label{model.p.sets-t}
\hypertarget{Talphas}{There is a formula $\varphi(u,v,w)$ in the language of
extender models such that the following holds. If $W=\mbfL[E]$ is a fine
structural extender model, $\alpha$ is an inaccessible cardinal of $W$ and
$\xi<\alpha^+$, letting 
\[
T_{\alpha,\xi}=
\{\tau\in\alpha^+\cap\cff(\alpha)\mid
W\cut(\alpha^+)^W\models\varphi(\tau,\alpha,\xi)\}, 
\]
each $T_{\alpha,\xi}$ is a stationary subset of $\alpha^+\cap\cff(\alpha)$ in
$W$, and $T_{\alpha,\xi}\cap T_{\alpha,\xi'}=\varnothing$ whenever
$\xi\neq\xi'$. Moreover, the sequence
$(\langle T_{\alpha,\xi}\mid\xi<\alpha^+\rangle\mid\alpha\le\kappa
 \mbox{ is inaccessible in $W$})$
satisfies clause \hyperlink{(D)}{(D)} above with $W$ in place of $\mbfV$.}
\end{proposition}

\begin{proof}
Since the definition of $\langle T_{\alpha,\xi}\mid\xi<\alpha^+\rangle$ takes
place inside $W\cut(\alpha^+)^W$, any two extender models $W,W'$ such that
$(\alpha^+)^W=(\alpha^+)^{W'}$ and $E^W\rst\alpha^+=E^{W'}\rst\alpha^+$
calculate this sequence the same way (here $\alpha^+$ stands for the common
value of the cardinal successor of $\alpha$ in both models). Now if $\mbfV=W$
and $j$ is as in (D) above then
\[
T'_{\alpha,\xi}=
\{\tau\in(\alpha^+)^{M'}\cap\cff(\alpha)\mid
M'\cut(\alpha^+)^{M'}\models\varphi(\tau,\alpha,\xi)\}, 
\]
whenever $\alpha\le j'(\kappa)$ is inaccessible in $M'$, 
so to see that $T'_{\kappa,\xi}=T_{\kappa,\xi}$ for all $\xi<\kappa^+$ it
suffices to prove that $(\kappa^+)^{M'}=(\kappa^+)^{\mbfV}$ and
$E^\mbfV\rst\kappa^+=E^{M'}\rst\kappa^+$ (where again $\kappa^+$ stands for the
common value of the cardinal successor of $\kappa$ in $\mbfV$ and
$M'$). Regarding the former, the inequality
$(\kappa^+)^{\mbfV}\le(\kappa^+)^{M'}$ is entirely general and follows from
the fact that 
$\ptm(\kappa^\mbfV)\subseteq\ptm(\kappa)^{M'}$. The reverse inequality follows
from the assumption that the generic extension preserves $\kappa^+$, so
$(\kappa^+)^{\mbfV}$ remains a cardinal in $M'$. The latter is then a
consequence of the coherence property FS6.

It remains to come up with a formula $\varphi$ such that the sets
$T_{\alpha,\xi}$ are stationary  in $W$ for all $\alpha,\xi$ of interest, and
pairwise disjoint. Here we make a more substantial use of the fine structure
theory of $W$. Given an inaccessible $\alpha$ and a $\xi<\alpha^+$, letting 
\begin{equation}\label{model.e.collapsing-level}
T_{\alpha,\xi}\dfeq\;
\parbox[t]{3.5in}{
the set of all $\tau\in\alpha^+\cap\cff(\alpha)$ such that
$\varrho^1(W\cut\beta(\tau))=\alpha$ and $W\cut\beta(\tau)$ has $\xi+1$
cardinals above $\alpha$,} 
\end{equation}
it is clear that $T_{\alpha,\xi}\cap T_{\alpha,\xi'}=\varnothing$ whenever
$\xi\neq\xi'$. Then it suffices to show that 
\begin{equation}\label{model.e.t-stationarity}
\mbox{$T_{\alpha,\xi}$ is stationary in $W$,}
\end{equation}
as we can then take $\varphi$ be the defining formula for the system
$(T_{\alpha,\xi})_{\alpha,\xi}$.

The first step toward the proof of (\ref{model.e.t-stationarity}) is the
following observation.
\begin{equation}\label{model.e.collapsing-level-cofinality}
\parbox{4in}{
Assume $\nu>\alpha$ is regular in $W$, $p\in W\cut\nu$ and $X$ is the
$\Sigma_1$-hull of $\alpha\cup\{p\}$ in $W\cut\nu$. Let
$\nu^X=\sup(X\cap\nu)$. Then $\cff^W(\nu^X)=\alpha$.
} 
\end{equation}

\begin{proof}
Obviously, $\gamma=\cff^W(\nu^X)\le\alpha$. Assume for a contradiction that
$\gamma<\alpha$. Let 
$\langle\nu_i\mid i<\gamma\rangle$ be an increasing sequence converging to
$\nu^X$ such that $\nu_i\in X$ for every $i<\gamma$. For each such $i$ pick a
$j_i\in\omega$ and an ordinal $\eta_i<\alpha$ such that
$\nu_i=h_{W\cut\nu}(j_i,\langle\eta_i,p\rangle)$ where $h_{W\cut\nu}$ is the
standard $\Sigma_1$-Skolem function for $W\cut\nu$. Here $W\cut\nu$ is of the
form $\langle J^E_\nu,\varnothing\rangle$ (see FS2), and we identify it with
the structure $J^E_\nu$. The Skolem function $h_{W\cut\nu}$ has a
$\Sigma_1$-definition of the form $(\exists w)\psi(w,u_0,u_1,v)$ where
$\psi$ is a $\Delta_0$-formula in the language of extender models. (The
standard $\Sigma_1$-Skolem function has a uniform $\Sigma_1$-definition, which
means that there is a $\Sigma_1$-formula which defines a
$\Sigma_1$-Skolem function $h_N$ over every acceptable structure $N$. However,
the argument below does not make use of uniformity of the definition.) Since
$\nu>\alpha$ is regular, 
\begin{equation}\label{model.e.skolem}
(\exists\bar{\nu})
\Big( J^E_{\bar{\nu}}\models(\forall i<\gamma)(\exists w)(\exists v)
\psi(w,j_i,\langle\eta_i,p\rangle,v) \Big)
\end{equation}
Since the statement in (\ref{model.e.skolem}) is $\Sigma_1$, there is some such
$\bar{\nu}$ with $J^E_{\bar{\nu}}\in X$. To justify this note that the sequences
$\langle\eta_i\mid i<\gamma\rangle$ and $\langle j_i\mid j<\gamma\rangle$ are
elements of $X$ as $J^E_\alpha\subseteq X$, and we can view these sequences as
parameters in the formula in (\ref{model.e.skolem}). Fix such an
ordinal~$\bar{\nu}$. Now consider $i<\gamma$ such that
$\nu_i>\omega\bar{\nu}$. Using 
(\ref{model.e.skolem}) pick $z$ and $\nu^*$ in $J^E_{\bar{\nu}}$ such that
$J^E_{\bar{\nu}}\models\psi(z,j_i,\langle\eta_i,p\rangle,\nu^*)$. Since $\psi$
is $\Delta_0$, we actually have
$J^E_\nu\models\psi(z,j_i,\langle\eta_i,p\rangle,\nu^*)$, which tells us that
$\nu^*=h_{W\cut\nu}(j_i,\langle\eta_i,p\rangle)=\nu_i$. As
$\nu_i>\omega\bar{\nu}$, this is a contradiction. This completes the proof of 
(\ref{model.e.collapsing-level-cofinality}). 
\end{proof}

Now let $C$ be a club subset of $\alpha^+$, $X$ be the $\Sigma_1$-hull of
$\alpha\cup\{C,\xi,\alpha^{+\xi+1}\}$ in $W\cut\alpha^{+\xi+2}$, $N$ be the
transitive collapse of $X$, and $\pi:N\to W\cut\alpha^{+\xi+2}$ be the inverse
of the collapsing isomorphism. Let further
$\tau=X\cap\alpha^+=\crp(\pi)$. Then $\tau>\xi$ as
$\alpha\cup\{\xi\}\subseteq X$. It is a standard fact that
$\cff^W(\tau)=\cff^W(\sup(X\cap\On))$ (and can be proved similarly as
(\ref{model.e.collapsing-level-cofinality}) above). Now
$\cff^W(\sup(X\cap\On))=\alpha$ by (\ref{model.e.collapsing-level-cofinality}),
hence $\cff^W(\tau)=\alpha$. Moreover $\tau\in C$ as $C$ is closed and $\tau$
is a limit point of $C$. Thus, the proof of
(\ref{model.e.collapsing-level}) will be complete once we show that
$\varrho^1(W\cut\beta(\tau))=\alpha$ and $W\cut\beta(\tau))$ has $\xi+1$
cardinals above $\kappa$. We first look at the set of cardinals in $N$.

By acceptability, the
structures $W\cut\alpha^{+\xi+1}$ and $W\cut\alpha^{+\xi+2}$ agree on what is a
cardinal below $\alpha^{+\xi+1}$. It follows that in $W\cut\alpha^{+\xi+2}$, the
statement
\[
\mbox{``The order type of the set of cardinals in the interval
$(\alpha,\alpha^{+\xi+1})$ is $\xi$"}
\]
can be expressed in a $\Sigma_1$-way as
\begin{equation}\label{model.e.cardinals-sigma-1}
\mbox{``The order type of the set of cardinals above $\alpha$ in the structure
$W\cut\alpha^{+\xi+1}$ is $\xi$."}
\end{equation}
Since $\pi$ is $\Sigma_1$-preserving and $\crp(\pi)=\tau$, this
$\Sigma_1$-statement can be pulled back to $N$ via $\pi$. Also by the
$\Sigma_1$-elementarity of $\pi$ we have $\pi^{-1}(\alpha^{+\xi+1})$ is the
largest cardinal in $N$. Then, using acceptability in $N$, we conclude:
\begin{equation}\label{model.e.cardinals-n}
\mbox{The order type of the set of cardinals above $\alpha$ in $N$ is $\xi+1$.}
\end{equation}
By construction, the $\Sigma_1$-Skolem function of $N$ induces a partial
surjection of $\alpha$ onto $N$. Then $\varrho^1(N)\le\alpha$ by FS5. Since
$\alpha$ is a cardinal in $W$, we conclude $\varrho^1(N)=\alpha$. Let
$\bar{N}$ be the core of $N$ and $\sigma:\bar{N}\to N$ be the core map. By
FS7, $\varrho^1(\bar{N})=\alpha$ and $\ptm(\alpha)^{\bar{N}}=\ptm(\alpha)^N$,
so in particular $\tau=(\alpha^+)^N=(\alpha^+)^{\bar{N}}$. By FS8,
$\bar{N}=W\cut\beta$ for some $\beta$. Since $\varrho^1(\bar{N})=\alpha$, FS5
implies $\beta=\beta(\tau)$. To see that $\bar{N}=W\cut\beta(\tau)$ has
$\xi+1$ cardinals above $\alpha$, first notice that, since by FS7 the map
$\sigma$ is cofinal, the largest cardinal in $N$ must be in the range of
$\sigma$. This along with (\ref{model.e.cardinals-n}) provides a
$\Sigma_1$-definition of $\xi$ in $N$ from parameters in $\rng(\sigma)$. The
point here is that we can reformulate the notion of cardinal in $N$ below
$\alpha^{+\xi+1}$ as the cardinal in the sense of the structure
$N\cut\alpha^{+\xi+1}$, similarly as in (\ref{model.e.cardinals-sigma-1}). It
follows that $\xi\in\rng(\sigma)$, and since $\xi<(\alpha^+)^N$ we have
$\xi<\crp(\sigma)$. Then, using the $\Sigma_1$-reformulation of
(\ref{model.e.cardinals-n}) one more 
time, we conclude that $\alpha^{+\eta}\in\rng(\sigma)$ for every $\eta\le\xi$,
which means that $W\cut\beta(\tau)=\bar{N}$ has $\xi+1$ cardinals above
$\alpha$. This completes the proof of (\ref{model.e.collapsing-level}) and
thereby the proof of Proposition~\ref{model.p.sets-t}.
\end{proof}

\subsection{The tools} \label{forcing tools}
Two main tools we will use to construct the forcing used to build our model
are club shooting with initial segments, and adding non-reflecting stationary
sets with initial segments. We then use variations of standard techniques for
building ideals using elementary embeddings. The background information on the
first two can be found in \cite{ifee}, \cite{CFM},  \cite{cfm1} and on ideal
constructions in \cite{quotalg}, but we review the relevant facts for the
reader's convenience. When discussing the successor of a regular cardinal
$\lambda$ we will often assume $\gch$ {even} when it is known that
$\lambda^{<\lambda}$ suffices. Since the models we work in satisfy the $\gch$
this  is not important for our results. 

Recall that if $S\subseteq\lambda^+$ is a stationary set (where
 $\lambda$ is a cardinal) then the club shooting partial ordering $\csp(S)$
consists of closed bounded subsets of $\lambda^+$ which are contained in 
$S$, and is ordered by end-extension. In general, this partial ordering may not have good
preservation properties, but if {$S$} is sufficiently large then it is known to
be highly distributive. {The following is standard.}

\begin{proposition}[See \cite{cummingssings}, \cite{CFM}, \cite{cfm1}]\label{model.p.distributivity}
{Assume  $\lambda$ is regular, $\lambda^{<\lambda}=\lambda$ and $T$ is a subset of
$\lambda^+$ such that $T\cap\alpha$ is non-stationary in $\alpha$ whenever
$\alpha<\lambda^+$, and $(\lambda^+\cap\cff(\lambda))\smallsetminus T$ is
stationary.} Then the following hold. 
\begin{itemize}
\item[(a)] $\csp(\lambda^+\smallsetminus T)$ is
$(\lambda^+,\infty)$-distributive, that is, it does not add any new function
$f:\lambda\to\mbfV$. In particular, generic extensions of $\mbfV$ via
$\csp(\lambda^+\smallsetminus T)$ agree with $\mbfV$ on all cardinals and
cofinalities $\le\lambda^+$, and on what $H_{\lambda^+}$ is.
\item[(b)] If $\gamma\le\lambda$ is regular and
$T\subseteq\lambda^+\cap\cff(\gamma)$ then $\csp(\lambda^+\smallsetminus T)$
has a dense set which is $\gamma$-closed but if $T$ is stationary then it does
not have a dense set which is $\gamma^+$-closed.
\item[(c)] If $G$ is $(\csp(\lambda^+\smallsetminus T),\mbfV)$-generic then
$C_G=\bigcup G$ is a closed unbounded subset of $\lambda^+$ such that
$C_G\subseteq\lambda^+\smallsetminus T$.
\end{itemize}
\end{proposition}

{To show that there is no saturated ideal in the model of Theorem \ref{t3} and Corollary \ref{c6} we will need to see that the forcing for shooting a closed unbounded set through the complement of a non-reflecting stationary set $A$ preserves stationary sets disjoint from $A$. This is the content of the next proposition that appears in \cite{CFM}, \cite{cfm1} and \cite{cummingssings}. We give the proof here for the reader's convenience.}

\begin{lemma}\label{what you need} Assume $\lambda$ is an uncountable cardinal with 
$\lambda^{<\lambda}=\lambda$, $A_1, A_2$ are disjoint stationary subsets of $\lambda^+$ 
and  that for all $\delta<\lambda^+, A_2\cap \delta$ is non-stationary. If $G\subseteq \csp(\lambda^+\smallsetminus A_2)$ is generic, then $A_1$ remains stationary in $V[G]$. 
\end{lemma}
\begin{proof} Let $p\in \csp(\lambda^+\smallsetminus A_2)$ force that $\dot{D}$ is a closed unbounded subset of 
$\lambda^+$ with 
$\dot{D}\cap A_1=\emptyset$. Let 
$\theta> (2^{2^\lambda})$ be a regular cardinal and  let $\la N_\alpha:\alpha<\lambda^+\ra$ be an \hyperlink{IA}{internally approachable sequence of elementary substructures} of $\la H_\theta, \epsilon, <_\theta,\{A_1, A_2, p, \dot{D}\}\ra$.
Then $\la N_\alpha\cap\lambda^+:\alpha$ is a limit$\ra$ is a closed unbounded subset of $\lambda^+$ and for each such $\alpha, N_\alpha^{<\cff(\alpha)}\subseteq N_\alpha$.

{Choose a limit $\delta$  such that $N_\delta\cap \lambda^+\in A_1$. Let 
$\gamma=\cff(\delta)$ and $C_\delta\subseteq(\delta\smallsetminus A_2)$ be a closed unbounded set of order type $\gamma$. {Build a decreasing sequence of conditions $\la p_\alpha:\alpha<\gamma\ra$ such that
\begin{itemize}
\item $p_0=p$
\item for each $\beta<\gamma$, $\la p_\alpha:\alpha<\beta\ra\in N_\delta$
\item if $i$ is the $\alpha^{th}$ member of $C_\delta$, then for some ordinal
$\xi<\delta$, with $i<\xi$
	\[p_{\alpha+1}\Vdash \xi\in \dot{D}.\]
\item $\sup(p_{\alpha+1})>i$.
\end{itemize}
Such a sequence is possible to build, because
$N_\delta^{<\gamma}\subseteq N_\delta$. }} 

{{But then $\sup(\bigcup_{\alpha<\gamma}p_\alpha)=\delta$ and $\delta\notin
A_2$, hence
\[
q=\bigcup_{\alpha<\gamma}p_\alpha\cup\{\delta\}\in \csp(\lambda^+\smallsetminus A_2).
\]
Moreover $q\Vdash \dot{D}\cap A_1\ne \emptyset$. This contradiction
establishes Lemma \ref{what you need}.}} \end{proof}

\bigskip

{Our application of the next definition and the following lemmas will be with $\mu=\lambda^+$ for a regular $\lambda$.
\begin{definition}\label{nonreflecting}
Let  $\mu$ be a regular cardinal. 
\begin{itemize}
\item[(a)] The  partial ordering  $\nrp(\mu)$ for adding a
non-reflecting stationary subset of $\mu$ consists of functions
$p:\alpha\to\{0,1\}$ for some $\alpha<\mu$ and letting
           \[
           S_p=\{\xi<\alpha\mid p(\xi)=1\},
           \]
           for every limit $\bar{\alpha}\le\alpha$ there is
	   a closed unbounded set $C\subseteq \bar{\alpha}$ such that
	   $S_p\cap C=\varnothing$.
\item[(b)]  Let $\gamma<\mu$ be regular. The partial ordering $\nrp(\mu,\gamma)$ for adding a
non-reflecting stationary subset of $\mu\cap\cff(\gamma)$ consists of those conditions $p\in\nrp(\mu)$ which concentrate on
$\mu\cap\cff(\gamma)$:
\[ p(\xi)=0 \mbox{ whenever }\cff(\xi)\neq\gamma.\]
\end{itemize}
\end{definition}

Let $\gamma<\mu$ be uncountable regular cardinals and define the map 
\[\pi_\gamma:\nrp(\mu)\to \nrp(\mu, \gamma)\]
by setting 
\[\pi_\gamma(p)(\xi)=
	\begin{cases}
	p(\xi) & \mbox{ if }\cff(\xi)=\gamma\\
	0 & \mbox{ otherwise.}
	\end{cases}
\]

We will use the following lemma which relates $\nrp(\mu)$ with 
$\nrp(\mu, \gamma)$.
\begin{lemma}\label{projecting}
Let $\gamma<\mu$ be uncountable regular cardinals. 
	\begin{itemize}
	\item[(a)] If $G$ is generic for $\nrp(\mu)$, then 
		\[S_G\dfeq\{\xi<\mu:\mbox{ for some }p\in G, p(\xi)=1\}\]
	 is a non-reflecting stationary subset of $\mu$.
	\item[(b)] If $H$ is generic for $\nrp(\mu,\gamma)$  then 
	\[S_H=_{def}\{\xi<\mu:\mbox{ for some }p\in H, p(\xi)=1\}\] is a non-reflecting stationary subset of $\mu\cap \cff(\gamma)$.
	\item[(c)] If $G\subseteq \nrp(\mu)$ is generic over $V$, and $H=\pi_\gamma``G$, then $H$ is generic over $V$ for
	$\nrp(\mu,\gamma)$. (In other words the map $\pi_\gamma$ is a \emph{projection}.)
	\end{itemize}
\end{lemma}
\begin{proof}
The first two items are immediate.  For the third note that for all $p\in \nrp(\mu)$ and all $q\in \nrp(\mu, \gamma)$ with $q\le_{\nrp(\mu, \gamma)}\pi_\gamma(p)$, there is a $p'\le_{\nrp(\mu)}p$ with $\pi_\gamma(p')\le_{\nrp(\mu,\gamma)}q$.  This is the standard criterion for being a projection.
 \end{proof}
 {It is an easy remark that in (a) $V[G]=V[S_G]$ and in (b) $V[H]=V[S_H]$. For this reason we will frequently write $V[S]$ for the extension, when it is clear from context whether we are in case (a) or (b).}

We will make use of these partial orderings in the special case where $\mu$ is of the form
$\lambda^+$. For this reason we formulate the next proposition for cardinals
of the form $\lambda^+$, although it is true for any regular $\mu>\omega$. 

\begin{proposition}[See \cite{ifee}, \cite{CFM},  \cite{cfm1}]\label{model.p.partial ordering-for-nr}
Assume $\gamma\le\lambda$ where $\gamma$ is regular and {$\lambda^{<\lambda}=\lambda$.}  Then the following hold. 
\begin{itemize}
\item[(a)] Both $\nrp(\lambda^+)$ and $\nrp(\lambda^+,\gamma)$ are
strategically $\lambda^+$-closed.  In particular, both $\nrp(\lambda^+)$ and
$\nrp(\lambda^+,\gamma)$ preserve stationarity of stationary 
subsets of $\lambda^+$, are $(\lambda^+,\infty)$-distributive, so they 
do not add any new functions $f:\lambda\to\mbfV$, and generic extensions of
$\mbfV$ via these partial orderings agree with $\mbfV$ on all cardinals and
cofinalities $\le\lambda^+$ and on what $H_{\lambda^+}$ is.
\item[(b)] $\nrp(\lambda^+,\gamma)$ is $\gamma$-closed but not
$\gamma^+$-closed.

\item[(c)] If $G$ is $(\nrp(\lambda^+,\gamma),\mbfV)$-generic then
$S_G=\bigcup\{S_p\mid p\in G\}$ is a non-reflecting stationary subset of
$\lambda^+\cap\cff(\gamma)$.
\item[(d)] If $G$ is $(\nrp(\lambda^+),\mbfV)$-generic then
$S_G=\bigcup\{S_p\mid p\in G\}$ is a non-reflecting stationary subset of
$\lambda^+$ such that {$S_G$ has stationary intersection with each stationary subset of $\lambda^+$ that lies in $V$.  In particular,} $S_G\cap\lambda^+\cap\cff(\gamma)$ is stationary for all
regular $\gamma<\lambda^+$.
\end{itemize}
\end{proposition}

\begin{proof} {Only (d) is not explicitly proved in the earlier literature (though it was known). 
Let $T$ be a stationary subset of $\lambda^+$ in $V$.
Let $G\subseteq \nrp(\lambda^+)$ be generic and  $S\subseteq \lambda^+$ be the generic stationary set added by $G$.  We claim that $S$ has stationary intersection   with
$T$. We assume without loss of generality that every ordinal in $T$ has the same cofinality $\gamma\le \lambda$.}

{If the claim fails let $p\in \nrp({\lambda^+})$ force over $V$ that $\dot{S}\cap T\cap\dot{D}=\emptyset$ where 
$\dot{D}$ is a term for a closed unbounded subset of $\lambda^+$ in $V[G]$.}

 { Let 
$\theta> (2^{2^\lambda})$ be a regular cardinal and  let $\la N_\alpha:\alpha<\lambda^+\ra$ be an 
\hyperlink{IA}{internally approachable sequence of elementary substructures} of 
$\la H_\theta, \epsilon, <_\theta,\{\dot{S}, T, \dot{D}\}\ra$.
Then $\la N_\alpha\cap\lambda^+:\alpha$ is a limit$\ra$ is a closed unbounded subset of $\lambda^+$ and for each such 
$\alpha, N_\alpha^{<\cff(\alpha)}\subseteq N_\alpha$. Choose a limit ordinal $\delta$  such that $N_\delta\cap \lambda^+\in T$ and 
$N_\delta\cap \lambda^+=\delta$. Then $\delta$ has cofinality $\gamma$.  Let $C_\delta\subseteq \delta$  be closed and 
unbounded in $\delta$ with order type $\gamma$ such that every initial segment of $C_\delta$ belongs to $N_\delta$.  }

{By recursion on $\beta$ build a decreasing sequence of conditions $\la p_\alpha:\alpha<\delta\ra$ in $\nrp(\lambda^+)$ such that
\begin{itemize}
\item $p_0=p$
\item for each $\beta<\gamma$,  $\la p_\alpha:\alpha<\beta\ra\in N_\delta$
\item if $i$ is the $\alpha^{th}$ member of $C_\delta$, then for some ordinal
$\zeta<\delta$, with $i<\zeta$
	\[p_\alpha\Vdash \zeta\in \dot{D}.\]
\item $\sup(dom(p_\alpha))>i$.
\item If $\beta$ is a limit ordinal, then $p_\beta=\bigcup_{\bar{\beta}<\beta}p_{\bar{\beta}}$ and if
$\delta_\beta=\sup(\bigcup_{\bar{\beta}<\beta} dom(p_{\bar{\beta}}))$, then $p_{\beta+1}$ forces $\delta_\beta\notin S$. 
\end{itemize}
Let $p^*=\bigcup_{\beta<\delta}p_\beta$.  Then $dom(p^*)$ has supremum $\delta$ and forces that 
\begin{itemize}
\item $\dot{S}\cap \delta$ is non-stationary
\item $\dot{D}\cap \delta$ is cofinal in $\delta$
\end{itemize}
  Extending  $p^*$ by one point to get a condition $q$ that  forces $\delta\in \dot{S}$ gives a condition $q\in \nrp(\lambda^+)$ that forces 
$\delta\in \dot{S}\cap T\cap \dot{D}$. This contradiction shows that in the extension by $\nrp(\lambda^+)$, $S$ intersects every stationary $T$. }
\end{proof}

Although both partial orderings $\nrp(\lambda^+,\gamma)$ and
$\csp(S)$ have a low degree of closure in general, {the
iteration  $\nrp(\lambda^+,\gamma)*\csp(\lambda^+\smallsetminus \dot{S})$ that
generically adds  a   non-reflecting stationary set $S$ followed by adding a
closed unbounded subset of the complement of ${S}$
does have a high degree of closure. }

\begin{proposition}\label{model.p.nrp-star-csp}
Assume $\lambda$ is a cardinal, $\gamma\le\lambda$ is regular, and $\dot{S}$
is the canonical $\nrp(\lambda^+,\gamma)$-term for the generic non-reflecting
stationary subset of $\lambda^+\cap\cff(\gamma)$. Then the composition
\[
\nrp(\lambda^+,\gamma)*\csp(\lambda^+\smallsetminus\dot{S})
\]
has a dense $\lambda^+$-closed subset $D\subseteq H_{\lambda^+}$. {In
particular, this two step iteration preserves stationarity of stationary 
subsets of $\lambda^+$.}
\end{proposition}
{\begin{proof} 
Let $D$ be the collection of all 
$(p,\dot{c})\in\nrp(\lambda^+,\gamma)*\csp(\lambda^+\smallsetminus\dot{S})\cap
H_{\lambda^+}$
such that 
\begin{itemize} 
\item $\{\xi:p(\xi)=0\}$ is closed (so has successor order type), and 
\item $p\forces\dot{c}=\check{c}$ for some closed unbounded set 
$c\subseteq\dom(p)$ with $p(\xi)=0$ for all $\xi\in c$.
\end{itemize}
Then $D$ is dense in $\nrp(\lambda^+,\gamma)*\csp(\lambda^+\smallsetminus\dot{S})$.  For details see \cite{ifee}, \cite{CFM} or \cite{cfm1}. 
\end{proof}

Fix a regular cardinal $\lambda$. At successor steps in the iteration used to prove Theorem \ref{t3}, we will use an iteration of the form
\begin{equation}\label{def of Palpha} \nrp(\lambda^+,\gamma)*\csp(\lambda^+\smallsetminus \dot{T})*\csp(\lambda^+\smallsetminus \dot{S}),\end{equation}
where $\dot{S}$ is {a term for} the generic non-reflecting stationary subset
of $\lambda^+\cap \cff(\gamma)$ given by $\nrp(\lambda^+,\gamma)$ and
$\dot{T}$ will be {a term for}  a certain subset of $\lambda^+\cap \cff(\lambda)$. 
{We note in passing that the realization of $\dot{T}$ is a non-reflecting stationary set.}
 Since both $\dot{S}$ and $\dot{T}$ lie in $V^{\nrp(\lambda^+,\gamma)}$, the following three forcing notions are equivalent:
\begin{description}
\item[Version 1]$\nrp(\lambda^+,\gamma)*\csp(\lambda^+\smallsetminus \dot{T})*\csp(\lambda^+\smallsetminus \dot{S})$
\item[Version 2] $\nrp(\lambda^+,\gamma)*(\csp(\lambda^+\smallsetminus \dot{T})\times\csp(\lambda^+\smallsetminus \dot{S}))$
\item[Version 3]$\nrp(\lambda^+,\gamma)*\csp(\lambda^+\smallsetminus \dot{S})*\csp(\lambda^+\smallsetminus \dot{T}).$
\end{description}

\begin{lemma}\label{isolation}
Let $\poP=\nrp(\lambda^+,\gamma)*\csp(\lambda^+\smallsetminus \dot{T})*\csp(\lambda^+\smallsetminus \dot{S})$.  Then $\poP$ has a dense set $D$ such that 
	\begin{itemize}
	\item[(i)] $D$ has cardinality $\lambda^+$, 
	\item[(ii)] $D\subseteq H_{\lambda^+}$,
	\item[(iii)] $D$ is $\lambda$-closed, and
	\item[(iv)] $D$ is $(\lambda^+,\infty)$-distributive.
        \end{itemize}
\end{lemma}
\begin{proof}
Proposition \ref{model.p.nrp-star-csp} shows that $\nrp(\lambda^+,\gamma)*\csp(\lambda^+\smallsetminus \dot{S})$ has a dense $\lambda^+$-closed subset.  Since $\dot{T}$ consists of ordinals of cofinality $\lambda$,
$\csp(\lambda^+\smallsetminus \dot{T})$ is $\lambda$-closed and 
$(\lambda^+,\infty)$-distributive.  Since $\poP$ is isomorphic to
$\nrp(\lambda^+,\gamma)*\csp(\lambda^+\smallsetminus
\dot{S})*\csp(\lambda^+\smallsetminus \dot{T})$, items (iii) and (iv) 
follow. Now (i) is immediate, since $\csp(\lambda^+\smallsetminus \dot{T})$ has a
dense set of size $\lambda^+$ after forcing with the first two partial
orderings.  

To see (ii), use Version~3
of the partial ordering  $\poP$.  The first step is clearly a subset of $H_{\lambda^+}$.  By Proposition \ref{model.p.nrp-star-csp} there is a dense subset of the first two steps that lies in $H_{\lambda^+}$ and is $\lambda^+$-closed.
After forcing with $\nrp(\lambda^+,\gamma)*\csp(\lambda^+\smallsetminus \dot{S})$ the conditions in 
$\csp(\lambda^+\smallsetminus \dot{T})$ belong to $H_{\lambda^+}$ and can be realized by elements of $V$ using the closure of $\nrp(\lambda^+,\gamma)*\csp(\lambda^+\smallsetminus \dot{S})$.  Hence there is a dense subset of Version~3 consisting of triples $(p, c, d)$ where each coordinate belongs to $H_{\lambda^+}$. Rearranging, we get (ii).
\end{proof}

In the iteration, we will construct $\dot{T}$ as a coding tool.   Let 
$\{T_\xi:\xi<\lambda^+\}$ be a sequence of disjoint stationary subsets of $\lambda^+\cap \cff(\lambda)$.
Let $S\subset \lambda^+$ and  define
\begin{equation}T(S)=\bigcup_{\xi\in S}T_{2\xi}\cup \bigcup_{\xi\notin S}T_{2\xi+1}.\label{T(S)}
\end{equation}
We will use $T({S})$ for a set $S$ that is $V$-generic for $\nrp(\lambda^+)$.  When forcing with $\nrp(\lambda^+,\gamma)$ we will use the following variant:

\begin{equation}
T_\gamma(S)=\bigcup_{\xi\in S\cap \cff(\gamma)}T_{2\xi}\ \cup \bigcup_{\xi\in \cff(\gamma)\cap 
(\lambda^+\smallsetminus S)}T_{2\xi+1}. \label{Tgamma(S)} 
\end{equation}
\bigskip

Given an $\nrp(\lambda^+)$-generic $S\subseteq \lambda$, and a sequence of sets $S_\gamma$ for each regular uncountable $\gamma\le \lambda$ with $S_\gamma=S\cap \cff(\gamma)$, the following holds: 
\begin{equation}\label{unioning}
T(S)=\bigcup_\gamma T_\gamma(S_\gamma)
\end{equation}
\hypertarget{out is out}{In particular if $\delta\notin T(S)$ then $\delta\notin T_\gamma(S)$.}

\begin{proposition}\label{new-coding}
Suppose $\lambda$ is regular and the $\gch$ holds.
Let $\poP$ be the partial ordering 
\[
\nrp(\lambda^+)*\csp(\lambda^+\smallsetminus{T(\dot{S})})*
\csp(\lambda^+\smallsetminus\dot{S})
\]
where $\dot{S}$ is the canonical $\nrp(\lambda^+)$-term for the generically
added non-reflecting stationary set $S$ and $T(\dot{S})$  is the canonical
$\nrp(\lambda^+)$-term for the set $T(S)$. 
If $G\subseteq\poP$ is generic then in $V[G]$:
\begin{itemize}
	\item[(a)] If $\xi\in\dot{S}^G$, then  $T_{2\xi}$ is non-stationary and $T_{2\xi+1}$ is stationary.
	\item[(b)] If $\xi\notin\dot{S}^G$, then $T_{2\xi+1}$ is non-stationary, and $T_{2\xi}$ is stationary.
\end{itemize}
\end{proposition}

\begin{proof} {Force with $\nrp(\lambda^+)$ to get a generic stationary set $S$ and} 
let $\dot{C}$ be a term for the closed unbounded set added by 
$\csp(\lambda^+\smallsetminus{T(S)})$. If
$H\subseteq\csp(\lambda^+\smallsetminus{T(S)})$ is $V[S]$-generic then
$\dot{C}^H\cap T(S)$ is empty which shows the non-stationarity claims in both
(a) and (b).  

What is left is to show that the appropriate $T_{\eta}$'s stationarity is
preserved. The argument in each case is the same, so assume we argue for case
(a). Since the partial ordering $\nrp(\lambda^+)*\csp(\lambda^+\smallsetminus\dot{S})$ has a
dense $<\!\lambda^+$-closed subset, it preserves the stationarity of each 
$T_\xi$.   

Suppose that $\xi\in\dot{S}^G$. Applying  Lemma \ref{what you need} in $V[G]$
with $A_1=T_{2\xi+1}$ and $A_2=T(S)$ shows that $T_{2\xi+1}$ is stationary in
$V[G][H]$. \end{proof}

Essentially the same proof shows:
\begin{proposition} \label{doggammait}
Suppose $\lambda$ is regular, $\gamma\le \lambda$ is regular and uncountable,
and {that the $\gch$ holds}.  
Let $\poP$ be the partial ordering 
\[
\nrp(\lambda^+,\gamma)*
\csp(\lambda^+\smallsetminus{T_\gamma(\dot{S})})*
\csp(\lambda^+\smallsetminus \dot{S})
\]
where $\dot{S}$ and $T_\gamma(\dot{S})$ are {defined as  in Proposition~\ref{new-coding}}. 
If $G\subseteq\poP$ is generic then in $V[G]$:
\begin{itemize}
	\item[(a)] If $\xi\in\dot{S}^G\cap \cff(\gamma)$, then  $T_{2\xi}$ is
	non-stationary and $T_{2\xi+1}$ is stationary. 
	\item[(b)] If $\xi\in\cff(\gamma)\cap(\lambda^+\smallsetminus\dot{S}^G)$,
	then $T_{2\xi+1}$ is non-stationary, and $T_{2\xi}$ is stationary. 
	\item[(c)] If $\xi\notin \cff(\gamma)$, then $T_\xi$ is stationary.
\end{itemize}
\end{proposition}
\hypertarget{codingcoding}{The point of this coding is that using  the forcing in either Proposition \ref{new-coding} 
or \ref{doggammait}, for $\xi$ of the appropriate cofinality we have:
		\[\xi\in S \mbox{ if and only if }T_{2\xi} \mbox{ is non-stationary and } T_{2\xi+1} \mbox{ is 
		stationary}.\]
	}

\begin{proposition}\label{Spreserved} 
Under the hypotheses of Proposition \ref{new-coding} (or Proposition \ref{doggammait}), the set $S$ added by $\nrp(\lambda^+)$ (respectively $\nrp(\lambda^+,\gamma)$)  remains stationary after forcing with
$
\csp(\lambda^+\smallsetminus{T(\dot{S})})$ (respectively 
$\csp(\lambda^+\smallsetminus{T_\gamma(\dot{S})})$).
\end{proposition}

\begin{proof}
\noindent We prove it with the hypotheses of Proposition \ref{new-coding}, the proof using the hypotheses of Proposition \ref{doggammait} is essentially the same. 

Let $G\subseteq \nrp(\lambda^+)$ be generic and  $S\subseteq \lambda^+$ be the generic stationary set constructed by $G$. 
By Proposition \ref{model.p.partial ordering-for-nr} item (d), in $V[G]$,  $S$ has stationary intersection with each $T_\xi$. Choose a $\xi_0$ such that $T_{\xi_0}\cap T(S)=\emptyset$. Let $A_1=S\cap T_{\xi_0}$ and $A_2=T(S)$.  The $A_1$ and $A_2$ satisfy the hypotheses of   Lemma \ref{what you need} for the forcing $\csp(\lambda^+\smallsetminus T(S))$.  
Hence $S\cap T_{\xi_0}$ is stationary in the generic extension of $V[S]$ by $\csp(\lambda^+\smallsetminus T(S))$, and so $S$ is stationary after the forcing $\nrp(\lambda^+)*\csp(\lambda^+\smallsetminus T(\dot{S})$.
\end{proof}

\subsection{The construction}\label{the construction of the interation}
  \medskip
 
 Let $U$ be the normal measure  as in
\hyperlink{B}{(B)} above and 
	\begin{equation}\label{model.e.map-j}
	j:\mbfV\to M
	\end{equation}
be the ultrapower embedding by $U$ where $M$ is transitive.  Let $\kappa$ be the critical point of $j$.
 
 {The forcing will be an iteration of length $\kappa+2$
  with Easton supports. If $\alpha<\kappa$ is inaccessible 
we will choose a regular uncountable $\gamma\le \alpha$ and do a three step forcing.  First we add a non-reflecting stationary set $S$. We then force to code the non-reflecting stationary set using the stationary sets $T_{\alpha,\xi}$. The last step is to shoot a club through the complement of the stationary set $S$ created in the first step.}

{At stage $\kappa$ we do the analogous forcing except that we only use the first two steps.}
\bigskip

\bfni{Description of the Forcing.}
We now {formally define the  partial orderings used in the construction.} 
For an inaccessible cardinal $\alpha$ fix the stationary sets \hyperlink{Talphas}{$\la T_{\alpha, \xi}:\xi<\alpha^+\ra$} from Proposition \ref{model.p.sets-t}. Fix a regular uncountable $\gamma\le \alpha$.   For this $\gamma$, let 
$\poQ_\alpha^\gamma$ be the partial ordering

\begin{equation}
\label{now-with-alpha}
\nrp(\alpha^+,\gamma)*
\csp(\alpha^+\smallsetminus T_{\alpha,\gamma}(\dot{S}_{\alpha,\gamma}))*
\csp(\alpha^+\smallsetminus \dot{S}_{\alpha,\gamma}),
\end{equation}
defined as in Proposition~\ref{doggammait}, with $\alpha$ in place of
$\lambda$, $T_{\alpha,\gamma}$ in place of $T_\gamma$, and
$\dot{S}_{\alpha,\gamma}$ in place of $\dot{S}$. (We will often
suppress $\gamma$ in the notation if $\gamma$ is clear from the context, and
write simply $\dot{S}_{\alpha}$.) 
\bigskip

The final partial ordering $\poP^*$ will be an iteration with Easton supports of length 
{$\kappa+2$.} We define the initial segment of length $\kappa$,  $\poP_\kappa$, as follows. 
$\poP_\kappa$ will be the direct limit of the forcing iteration
\[
(\mbbP_\alpha\mid\alpha\le\kappa)  \]
satisfying the following.
\begin{itemize}
\item[FI-1] \hypertarget{F1}{For inaccessible $\alpha$, conditions in each $\mbbP_\alpha$ are partial functions $p$ with
$\dom(p)$ contained in inaccessibles below $\alpha$ such that
$\dom(p)\cap\beta$ is bounded in $\beta$ whenever $\beta\le\alpha$ is
inaccessible.}
\item[FI-2] \hypertarget{F2}{If $p\in\mbbP_\alpha$ and $\bar{\alpha}\in\dom(p)$ then
\[
p(\bar{\alpha})=
(\gamma^p(\bar{\alpha}),w^p(\bar{\alpha}))\]
is an ordered pair such that 
\[
\gamma^p(\bar{\alpha})\in R_{\bar{\alpha}}=
\{\gamma\le\bar{\alpha}\mid\gamma\mbox{ is regular uncountable}\},
\]}
and $w^p(\bar{\alpha})\in H_{\alpha^+}$ is a $\mbbP_{\bar{\alpha}}$-term  for a condition in the three step forcing 
$\mbbQ^{\gamma^p(\bar{\alpha})}_{\bar{\alpha}}$ defined in equation \ref{now-with-alpha}.\footnote{We can view $w^p(\bar{\alpha})$ as a triple $(w^p(\bar{\alpha}), w^p(\bar{\alpha}+1), w^p(\bar{\alpha}+2))$ but the notation $w^p(\bar{\alpha})$ is frequently  more convenient.}
\end{itemize}
The ordering on $\mbbP_\alpha$ is defined in the standard way, that is,
\begin{itemize}
\item[FI-3] \hypertarget{F3}{$p\le q$ iff the following hold:
  \begin{itemize}
  \item[(1)] $\dom(p)\supseteq\dom(q)$ and
  \item[(2)] for every $\bar{\alpha}\in\dom(q)$: 
      \begin{itemize}
      \item[(a)] $\gamma^p(\bar{\alpha})=\gamma^q(\bar{\alpha})$ and 
      \item[(b)] $p\rst\bar{\alpha}\forces_{\mbbP_{\bar{\alpha}}}
            \mbox{``$w^p(\bar{\alpha})$ extends $w^q(\bar{\alpha})$
	    in $\dot{\mbbQ}^{\gamma^p(\bar{\alpha})}_{\bar{\alpha}}$"}$
      \end{itemize}
   (where, by $p\rst\bar{\alpha}$ we  mean
   $p\rst(\dom(p)\cap\bar{\alpha})$.   )    }
\end{itemize}
\end{itemize}

From lemmas \ref{model.p.partial ordering-for-nr} to \ref{isolation}, we conclude that:
	\begin{itemize}
	\item[(i)] \hypertarget{theobv}{For all inaccessible $\alpha, \poP_\alpha\subseteq V_\alpha$}
	\item[(ii)] For $\alpha$ Mahlo, $\poP_\alpha$ is $\alpha$-c.c. 
	\item[(iii)] If $G$ is $(\poP_\alpha,\mbfV)$-generic then in
	$\mbfV[G]$ the partial ordering $(\dot{\poQ}^{\gamma^p(\bar{\alpha})}_\alpha)^G$
	contains a dense $\alpha$-closed set and is
	$(\alpha^+,\infty)$-distributive. 
	\item[(iv)] \hypertarget{top closed}{For $\alpha<\kappa$, if
	$\poP_\kappa=\poP_\alpha*\dot{\poP}^\alpha_\kappa$ is the canonical
	factorization, and $G$ is $(\poP_\alpha,\mbfV)$-generic, then 
	\[
	\mbfV[G]\models \mbox{``$(\dot{\poP}^\alpha_\kappa)^G$ has an
	$\alpha$-closed dense subset"}
	\]}
	\item[(v)] For each inaccessible $\alpha<\kappa$, if $p\in
	\poP_\kappa$ then $(p(\bar{\alpha}),p(\bar{\alpha}+1), p(\bar{\alpha}+3))\in H_{\bar{\alpha}^+}$.
	\item[(vi)] For all cardinals $\alpha$, $\poP_{\alpha+3}$ preserves both $\alpha$ and $\alpha^+$.
	\item[(vii)] $\poP_\kappa$ preserves all cardinals.
	\end{itemize}

\bigskip
Now define a \hypertarget{Pstar}{partial ordering $\mbbP^*$} as {the $\kappa+2$ length iteration:}
\begin{equation}\label{model.e.p-star}
\mbbP^*=\mbbP_\kappa*\nrp(\kappa^+)*\csp(\kappa^+\smallsetminus{T(\dot{S}}))
\end{equation}
{where $\dot{S}$ and $T(\dot{S})$ are as in Proposition~\ref{new-coding}, with
$\kappa$ in place of $\lambda$.
\medskip

We claim that any generic extension via $\mbbP^*$ produces a model as in
Theorem~\ref{t3}. We will first focus on the proof of the following
proposition.

\begin{proposition}\label{model.p.extensions-via-p-star}
In any generic extension via $\mbbP^*$ all cardinals and cofinalities are
preserved, $\kappa$ remains inaccessible, and for  each regular {uncountable} 
$\gamma\le\kappa$
there is a uniform normal $(\kappa^+,\infty)$-distributive ideal $\mcJ_\gamma$
such that $\ptm(\kappa)/\mcJ_\gamma$ has a dense $\gamma$-closed set, but no
dense $\gamma^+$-closed set. 
\end{proposition}

\begin{proof} Fix a regular  {uncountable} cardinal $\gamma\le \kappa$.
\medskip

By $\gch$ in $\mbfV$, any generic extension via $\mbbP_\kappa$ satisfies
$2^\kappa=\kappa^+$, so in any such generic extension the partial ordering
$\nrp(\kappa^+)$ has cardinality $\kappa^+$. Using the strategic closure of
$\nrp(\kappa^+)$ we conclude that $2^\kappa=\kappa^+$ in the generic extension
via $\mbbP_\kappa*\nrp(\kappa^+)$. Let $S$ be the non-reflecting stationary set added by 
$\nrp(\kappa^+)$.  Then $\csp(\kappa^+\smallsetminus T(S))$ has
cardinality $\kappa^+$ in any such generic extension. All of this combined
with the distributivity properties of $\nrp(\kappa^+)$ and
$\csp(\kappa^+\smallsetminus T(S))$, {shows that $2^\kappa=\kappa^+$.  Similar arguments show that} 
\begin{equation}\label{model.e.p-star-preservation}
\mbox{$\mbbP^*$ preserves all cardinals and cofinalities and also the $\gch$.}
\end{equation}

Now return to the map $j$ from (\ref{model.e.map-j}). Let $G$ be
$(\mbbP_\kappa,\mbfV)$-generic.
Because \hyperlink{theobv}{$\card(\poP_\kappa)=\kappa$ and $\poP_\kappa$ is $\kappa$-c.c.}, $M[G]$ is closed under
$\kappa$-sequences in $\mbfV[G]$ and the models $M[G],\mbfV[G]$ agree on what
$H_{\kappa^+}$ is.  It follows that the models $M[G],\mbfV[G]$ agree on what {$\nrp(\kappa^+)$} and $\nrp(\kappa^+,\gamma)$ are.

Let  $G'=G'_0*G'_1$ be
$(\nrp(\kappa^+)*\csp(\kappa^+\smallsetminus T(\dot{S})),\mbfV[G])$-generic
where $\dot{S}$ is as in equation \ref{model.e.p-star}. 
It follows that $S=\dot{S}^{G'_0}=\bigcup G'_0$.

Let $G_{\kappa,0}=\pi_\gamma(G_0')$ where $\pi_\gamma$ is as in Lemma \ref{projecting}. Then $G_{\kappa,0}$ is generic for $\nrp(\kappa,\gamma)$ {over both $M[G]$ and $V[G]$. Let
$\dot{S}$ be the term for the non-reflecting stationary set coming from $G'_0$. Then 
$\dot{S}_{\kappa,\gamma}=\dot{S}\cap \cff(\gamma)$.  Denote $ S_{\kappa,\gamma}$ by $S_\kappa$.}

{Since $\nrp(\kappa^+)$ and $\nrp(\kappa^+,\gamma)$ are  
$(\kappa^+,\infty)$-distributive in the models where they live,
\begin{equation}\label{model.e.m-ko-closure}
\mbox{$M[G,G_{\kappa,0}]$ is closed under $\kappa$-sequences lying in  $\mbfV[G,G'_0]$.}
\end{equation} 
In particular,
$M[G,G_{\kappa,0}]$ and $\mbfV[G,G'_0]$ agree on what $H_{\kappa^+}$ and
$\csp(\kappa^+\smallsetminus T_{\kappa,\gamma}(S_\kappa))$ are.}

Let $C\in\mbfV[G,G']$ be the closed unbounded subset of
$\kappa^+\smallsetminus T(S)$ associated with the generic ultrafilter $G'_1$ for
$\csp(\kappa^+\smallsetminus T(S))$ over $\mbfV[G,G'_0]$.

Notice that
$T_{\kappa,\gamma}(S_\kappa)\in M[G,G_{\kappa,0}]$ and 
$C\cap T_{\kappa,\gamma}(S_\kappa)=\varnothing$ because
$T_{\kappa,\gamma}(S_\kappa)\subseteq T(S)$. From the point of view of
$\mbfV[G,G'_0]$ there are only $\kappa^+$ many dense subsets of
$\csp(\kappa^+\smallsetminus T_{\kappa,\gamma}(S_\kappa))$ which are in
$M[G,G_{\kappa,0}]$.

 We can construct a 
$(\csp(\kappa^+\smallsetminus T_{\kappa,\gamma}(S_\kappa)),M[G,G_{\kappa,0}])$-generic 
filter $G_{\kappa,1}\in\mbfV[G,G']$ as follows.
 In $\mbfV[G,G'_0]$ fix an
enumeration $\langle D_\beta\mid\beta<\kappa^+\rangle$ of dense subsets of
$\csp(\kappa^+\smallsetminus T_{\kappa,\gamma}(S_\kappa))$ which {belong to}
$M[G,G_{\kappa,0}]$. {Using recursion} on $\beta<\kappa^+$ construct a descending
chain $\langle c_\beta,c'_\beta\mid\beta<\kappa^+\rangle$ in
$\csp(\kappa^+\smallsetminus T_{\kappa,\gamma}(S_\kappa))$ as follows.
\begin{itemize}
\item Let $c'_0=\varnothing$. 
\item Given $c'_\beta$, pick $c_\beta\in D_\beta$ such that
$c_\beta\le c'_\beta$ in
$\csp(\kappa^+\smallsetminus T_{\kappa,\gamma}(S_\kappa))$. 
\item Given $c_\beta$, let $c'_{\beta+1}=c_\beta\cup\{\delta_{\beta+1}\}$ where
$\delta_{\beta+1}$ is the least element of $C$ larger than $\max(c_\beta)$.
\item If $\beta$ is a limit let
$c'_\beta=\left(\bigcup_{\bar{\beta}<\beta}c'_{\bar{\beta}}\right)
\cup\{\delta_\beta\}$ 
where $\delta_\beta=\sup\{\max(c'_{\bar{\beta}})\mid\bar{\beta}<\beta\}$.  
\end{itemize}
To see that this works, notice that for every $\beta<\kappa^+$ both $c_\beta$
and $c'_\beta$ are elements of $M[G,G_{\kappa,0}]$, which is verified
inductively on $\beta$. The only non-trivial step in the induction is to see
that $c'_\beta\in M[G,G_{\kappa,0}]$ for $\beta$ limit.  That the sequence $\la c'_{\bar{\beta}}:{\bar{\beta}}<\beta\ra$ 
belongs to $M[G,G_{\kappa,0}]$ follows from the $\kappa$-closure property of $M[G,G_{\kappa,0}]$.
For the union to be a condition 
 requires that the supremum $\delta$ of $c'_\beta$ does not belong to $T_{\kappa,\gamma}(S_\kappa)$.  
 However by equation (\ref{unioning}), since $\delta\notin T(S)$ we know that 
 $\delta\notin T_{\kappa,\gamma}(S_\kappa)$.

Now let $G_{\kappa,1}$ be the filter on
$\csp(\kappa^+\smallsetminus T_{\kappa,\gamma}(S_\kappa))$ generated by the
sequence $\langle c_\beta\mid\beta<\kappa^+\rangle$; it is clear that
$G_{\kappa,1}\in\mbfV[G,G']$ and is
$(\csp(\kappa^+\smallsetminus T_{\kappa,\gamma}(S_\kappa)),M[G,G_{\kappa,0}])$-generic. Finally set $G_\kappa=G_{\kappa,0}*G_{\kappa,1}$.

{We note here that by Proposition \ref{Spreserved}, $S_\kappa$ is stationary in $\mbfV[G,G']$. Thus in $\mbfV[G,G']$, $S_\kappa$ is a non-reflecting stationary set.}

Consider a 
$(\csp(\kappa^+\smallsetminus S_\kappa),\mbfV[G,G'])$-generic filter
$H$. Then the filter 
$G_\kappa*H$ is
$(\nrp(\kappa^+,\gamma)*
  \csp(\kappa^+\smallsetminus T_{\kappa,\gamma}(\dot{S}_\kappa))*
  \csp(\kappa^+\smallsetminus\dot{S}_\kappa),M[G])$-generic.
It follows that $G*G_\kappa*H$ is
$(j(\mbbP_\kappa)\rst(\kappa+3),M)$-generic. By 
the Factor Lemma applied inside $M[G,G_\kappa,H]$, the quotient
$j(\mbbP_\kappa)/G*G_\kappa*H$ is isomorphic to the iteration
$\mbbP^{\kappa+3}_{j(\kappa)}$ as calculated in $M[G,G_\kappa,H]$.
 Let $\mu$ 
be the least inaccessible of $M$ above $\kappa$. Using \hyperlink{top
closed}{(iv)} in the list of the properties of the iteration stated
below~FI-3, we conclude that $M[G,G_\kappa,H]$ satisfies the following:
\begin{equation}\label{model.e.closure-quotient}
\mbox{$j(\mbbP_\kappa)/G*G_\kappa*H$ has a dense $\mu$-closed subset.}
\end{equation}
Since
$\nrp(\kappa^+)*\csp(\kappa^+\smallsetminus T(\dot{S}_\kappa))*
 \csp(\kappa^+\smallsetminus\dot{S}_\kappa)$
is $(\kappa^+,\infty)$-distributive in $\mbfV[G]$,
\begin{equation}\label{model.e.m-final-closure}
\mbox{$M[G,G_\kappa,H]$ is closed under $\kappa$-sequences in $\mbfV[G,G',H]$.}
\end{equation}
Working in $\mbfV[G,G',H]$: since the cardinality of {$\mbbP^{\kappa+3}_{j(\kappa)}$} 
 is $\kappa^+$, we have an enumeration
$\langle D_\beta\mid\beta<\kappa^+\rangle$ of all dense subsets of
$j(\mbbP_\kappa)/G*G_\kappa*H$ which are in $M[G,G_\kappa,H]$. Using
\hyperlink{top closed}{(iv)} in the list of the properties of the iteration
stated below~FI-3,
the sentences labelled (\ref{model.e.closure-quotient}), (\ref{model.e.m-final-closure}) above and 
the fact that $\mu>\kappa^+$, we can construct a descending sequence
$\langle p_\beta\mid\beta<\kappa^+\rangle$ with each proper initial segment 
being an element of $M[G,G_\kappa,H]$ and such that $p_\beta\in D_\beta$
for all $\beta<\kappa^+$. {Let {$K_1$} be the filter on
$j(\mbbP_\kappa)/G*G_\kappa*H$ generated by this sequence. Then $K_1$ is
$(j(\mbbP_\kappa)/G*G_\kappa*H,M[G,G_\kappa,H])$-generic and 
$K_1\in\mbfV[G,G',H]$. Let $K= G*G_\kappa*H*K_1$.  Then $K$ can be
viewed as a $(j(\mbbP_\kappa),M)$-generic filter, so} we can extend $j$ to an
elementary embedding
{\[j_{H,K}:\mbfV[G]\to M[K]\]}
\noindent defined {by setting}
$j_{H,K}(\dot{x}^G)=j(\dot{x})^K$ whenever
$\dot{x}\in\mbfV$ is a $\mbbP_\kappa$-term.
Since
{$K_1$ can  be 
constructed inside $\mbfV[G,G',H]$, there is a
$\csp(\kappa^+\smallsetminus S_\kappa)$-term $\dot{K_1}\in\mbfV[G,G']$ such that
$\dot{K}_1^H$} is 
$(j(\mbbP_\kappa)/G*G_\kappa*H,M[G,G_\kappa,H])$-generic whenever $H$ is
$(\csp(\kappa^+\smallsetminus S_\kappa),\mbfV[G,G'])$-generic. {In particular there is a $M$-generic $K^H\subseteq j(\mbbP)$ determined by {forcing over  $V[G,G_\kappa]$ to get a generic $H\subseteq \csp(\kappa^+\smallsetminus S_\kappa)$.}}

{Changing notation slightly to emphasize the dependence on $H$, define}  
$j_H$ be as follows.  
\begin{equation}\label{model.e.j-k}
j_H=j_{H,\dot{K}^H}:\mbfV[G]\to M[\dot{K}^H].
\end{equation}

{We also }have a
$\csp(\kappa^+\smallsetminus S_\kappa)$-term $\dot{U}\in\mbfV[G,G']$ such that
$\dot{U}^H$ is the normal $\mbfV[G]$-measure over $\kappa$ derived from $j_H$.
That is, 
\begin{equation}\label{model.e.dot-u}
\dot{U}^H=\{x\in\ptm(\kappa)^{\mbfV[G]}\mid \kappa\in j_H(x)\}
\end{equation}
whenever $H$ is a $(\csp(\kappa^+\smallsetminus S_\kappa),\mbfV[G,G'])$-generic
filter. It is a standard fact that
\begin{equation}\label{model.e.ultrapower}
\parbox{3.2in}{$M[\dot{K}^H]=\ult(\mbfV[G],\dot{U}^H)$ and
$j_H:\mbfV[G]\to M[\dot{K}^H]$ is the associated ultrapower map.}
\end{equation}
Since the composition 
$\nrp(\kappa^+)*\csp(\kappa^+\smallsetminus T(\dot{S}_\kappa))*
 \csp(\kappa^+\smallsetminus\dot{S_\kappa})$ is
$(\kappa^+,\infty)$-distributive in $\mbfV[G]$, the models $\mbfV[G]$ and
$\mbfV[G,G']$ agree on what $\ptm(\kappa)$ is, so $\dot{U}^H$ is also a normal 
$\mbfV[G,G',H]$-measure over $\kappa$. Since $\dot{U}^H\in\mbfV[G,G',H]$ we
record that
\begin{equation}\label{model.e.measurability}
\mbox{$\kappa$ is measurable in $\mbfV[G,G',H]$.}
\end{equation}

We now define the ideal $\mcJ_\gamma$ on $\ptm(\kappa)$ in $\mbfV[G,G']$. For every $x\in\ptm(\kappa)^{\mbfV[G,G']}$,  
\begin{equation}\label{model.e.j-gamma}
x\in\mcJ_\gamma
\;\Longleftrightarrow\;\;
\forces^{\mbfV[G,G']}_{\csp(\kappa^+\smallsetminus S_\kappa)}
\check{x}\notin\dot{U}, 
\end{equation}
{Note  that this definition takes place in $\mbfV[G,G']$ so $\mcJ_\gamma\in\mbfV[G,G']$ and standard}
arguments show
that $\mcJ_\gamma$ is a uniform normal ideal on $\ptm(\kappa)$ in
$\mbfV[G,G']$. 

{Recall
that $S_\kappa\subseteq\kappa^+\cap\cff(\gamma)$ where $\gamma$ was fixed at
the in $V[G]$. This is crucial for determining the closure properties of
$\ptm(\kappa)/\mcJ_\gamma$.} 
The main tool for analyzing properties of
$\mcJ_\gamma$ is the duality theory developed in \cite{quotalg}. Rather than simply cite theorems there, we show the following proposition. 

\begin{proposition}\label{model.p.dense-embedding}
In $\mbfV[G,G']$ there is a dense embedding
\[
e:\csp(\kappa^+\smallsetminus S_\kappa)\to\ptm(\kappa)/\mcJ_\gamma.
\]
\end{proposition}

\begin{proof}
In $\mbfV$, fix an assignment $x\mapsto f_x$ where $x\in M$ and
$f_x:\kappa\to\mbfV$ is such that  
\begin{equation}\label{model.e.f-x}
x=[f_x]_U=j(f_x)(\kappa).
\end{equation}
The partial ordering $\csp(\kappa^+\smallsetminus S_\kappa)$ in the generic extension
$M[G,G_\kappa]$ can be viewed as the quotient
$\left(j(\mbbP_\kappa)\rst(\kappa+1)\right)/G*G_\kappa$, so we can consider
conditions in $\csp(\kappa^+\smallsetminus S_\kappa)$ as elements of $M$ {that are}
ordered the same way as conditions in $j(\mbbP_\kappa)$. Hence
each such condition $p$ is represented in the ultrapower by $U$ by the
function $f_p$. 

Next, recall that at each inaccessible $\alpha<\kappa$, {stages $\alpha, \alpha+1$ and $\alpha+2$ of 
$\mbbP_\kappa$ are a composition of three
partial orderings where the last one is $\csp(\alpha^+\smallsetminus S_\alpha)$.}
The $\alpha+1, \alpha+2, \alpha+3$ components of the generic filter $G$ are then of the
form $G_{\alpha,0}*G_{\alpha,1}*h(\alpha)$ where $h(\alpha)$ is
$(\csp(\alpha^+\smallsetminus S_\alpha),
  \mbfV[G\rst\alpha*G_{\alpha,0}*G_{\alpha,1}])$-generic.
The function $h$ is thus an element of $\mbfV[G]$ and represents the filter
$H$ in the ultrapower by $\dot{U}^H$, that is, $H=j_H(h)(\kappa)$; see
(\ref{model.e.ultrapower}).  

Then for any $p\in\csp(\kappa^+\smallsetminus S_\kappa)$ we have the following:
\begin{equation}\label{model.e.calculation}
p\in H
\;\Longleftrightarrow\;
j_H(f_p)(\kappa)\in j_H(h)(\kappa)
\;\Longleftrightarrow\;
a_p\dfeq\{\alpha<\kappa\mid f_p(\alpha)\in h(\alpha)\}\in \dot{U}^H.
\end{equation} 
We show that in $\mbfV[G,G']$, the map
$e:\csp(\kappa^+\smallsetminus S_\kappa)\to\ptm(\kappa)/\mcJ_\gamma$
defined by 
\begin{equation}\label{model.e.map-e}
e(p)=[a_p]_{\mcJ_\gamma}
\end{equation}
is a dense embedding. The proof is a standard variant of the duality
argument, which we include for the reader's convenience. We write
briefly $[a]$ for $[a]_{\mcJ_\gamma}$.

To see that $e$ is order-preserving, consider $p\le q$ in
$\csp(\kappa^+\smallsetminus S_\kappa)$. By the above remarks on the ordering
of the quotient, we 
have $p\le q$ in $j(\mbbP_\kappa)$, hence $j(f_p)(\kappa)\le j(f_q)(\kappa)$
in $j(\mbbP_\kappa)$. It follows that
\[
b_{p,q}\dfeq\{\xi<\kappa\mid f_p(\xi)\le f_q(\xi)\}\in U,
\]
and so $b_{p,q}\in\dot{U}^H$ whenever $H$ is a 
$(\csp(\kappa^+\smallsetminus S_\kappa),\mbfV[G,G'])$-generic filter. It
follows that  
$\kappa\smallsetminus b_{p,q}\in\mcJ_\gamma$. Since
$a_p\smallsetminus a_q\subseteq\kappa\smallsetminus b_{p,q}$, we have
$[a_p]\le_{\mcJ_\gamma}[a_q]$.

To see that the map $e$ is incompatibility preserving, we prove the
contrapositive. Assume $p,q\in\csp(\kappa^+\smallsetminus S_\kappa)$ are such
that $a_p\cap a_q\in\mcJ^+_\gamma$. It follows that there is some
$(\csp(\kappa^+\smallsetminus S_\kappa),\mbfV[G,G'])$-generic filter $H$ such that
$a_p\cap a_q\in\dot{U}^H$. Then $a_p\in\dot{U}^H$ and $a_q\in\dot{U}^H$. Using
(\ref{model.e.calculation}) we conclude that $p,q\in H$. Hence $p,q$ are
compatible.

To see that {the range of } $e$ is dense, assume that $a\in\mcJ^+_\gamma$. It follows that
there is some $(\csp(\kappa^+\smallsetminus S_\kappa),\mbfV[G,G'])$-generic
filter $H$ such that $a\in\dot{U}^H$. So there is some $p\in H$ such that
\begin{equation}\label{model.e.a-forced}
p\forces^{\mbfV[G,G']}_{\csp(\kappa^+\smallsetminus S_\kappa)}
\check{a}\in\dot{U}. 
\end{equation}
Now for every $(\csp(\kappa^+\smallsetminus S_\kappa),\mbfV[G,G'])$-generic
filter $H$ we have 
\[
a_p\in\dot{U}^H
\;\Longrightarrow\;
p\in H
\;\Longrightarrow\;
a\in\dot{U}^H.
\]
Here the first implication follows from (\ref{model.e.calculation}) and the
second implication from (\ref{model.e.a-forced}). 
We thus conclude that $a_p\smallsetminus a\notin\dot{U}^H$ whenever $H$ is a
$(\csp(\kappa^+\smallsetminus S_\kappa),\mbfV[G,G'])$-generic filter, which
means that $a_p\smallsetminus a\in\mcJ_\gamma$, or equivalently,
$[a_p]\le_{\mcJ_\gamma}[a]$.
\end{proof}
We can now complete the proof of
Proposition~\ref{model.p.extensions-via-p-star}
{by {looking}} at the
properties of the partial ordering $\csp(\kappa^+\smallsetminus S_\kappa)$ in
$\mbfV[G,G']$. {By Proposition \ref{Spreserved}, $S_\kappa$ is stationary in $\mbfV[G,G']$, 
so $\csp(\kappa^+\smallsetminus S_\kappa)$ is a standard forcing for killing a non-reflecting stationary subset of $\kappa^+$.
The $(\kappa^+,\infty)$-distributivity follows from
Proposition~\ref{model.p.distributivity}(a). The existence of a dense
$\gamma$-closed set as well as the non-existence of a dense $\gamma^+$-closed
set follows from Proposition~\ref{model.p.distributivity}(b) and the fact that
$S_\kappa\subseteq\kappa^+\cap\cff(\gamma)$. }           \qed

\bigskip

The last major step toward the proof of Theorem~\ref{t3} is the following
proposition.

\begin{proposition}\label{model.p.no-sat-ideals}
$\kappa$ does not carry a saturated ideal in a generic extension
\hyperlink{Pstar}{via~$\mbbP^*$}.  
\end{proposition}

\begin{proof}
Assume for a contradiction that $\kappa$ does carry a saturated ideal in
$\mbfV[G,G']$ where $G,G'$ are as above. Denote this ideal by $\mcI$, and let
$L$ be a $(\mbbP_\mcI,\mbfV[G,G'])$-generic filter where $\mbbP_\mcI$ is the
partial ordering $(\mcI^+,\subseteq)$ {and} 
\[j':\mbfV[G,G']\to N\]
     be the
generic embedding associated with the ultrapower
$\ult(\mbfV[G,G'],L)$. Letting $M'=j'(\mbfV)$ and $(K,K')=j'(G,G')$, we have
$N=M'[K,K']$. The partial ordering $\mbbP^**\mbbP_\mcI$ preserves $\kappa^+$, which
allows us to refer to \hyperlink{(D)}{(D)} at the beginning of this section. It follows that
the  models $\mbfV,M'$ and all transitive extensions of these models which are
contained in $\mbfV[G,G',L]$ have a common cardinal successor of $\kappa$,
which we denote by $\kappa^+$.    

Now look at the $\kappa$-th step of the iteration
$j'(\mbbP_\kappa)$. Obviously $j'(\mbbP_\kappa)\rst\kappa=\mbbP_\kappa$ and
$K\cap\mbbP_\kappa=G$. Let $\gamma\in R^{M'}_\kappa=R_\kappa$ be the ordinal 
chosen by the generic filter $K$ at step $\kappa$ of the iteration
$j'(\mbbP_\kappa)$ (see
\hyperlink{F2}{FI-2}).  Then steps $\kappa, \kappa+1$ and $\kappa+2$  are thus forcing with 
\[
\nrp(\kappa^+,\gamma)*\csp(\kappa^+\smallsetminus T(\dot{S}_\kappa))*
\csp(\kappa^+\smallsetminus\dot{S}_\kappa)
\]
over $M'[G]$. 
This composition of partial orderings is computed the
same way in $M'[G]$ and $\mbfV[G]$, as by \hyperlink{(D)}{(D)} at the beginning of this
section, the models $\mbfV$ and $M'$ agree on what $H_{\kappa^+}$ is,  but we don't use this
directly. What is relevant is
the agreement of the models on what $\kappa^+$ is, along with the fact that
$T'_{\kappa,\xi}=T_{\kappa,\xi}$ for all $\xi<\kappa^+$ where the sets
$T_{\kappa,\xi}$ and $T'_{\kappa,\xi}$ are as in \hyperlink{(D)}{(D)} quoted above.

The $\kappa$-th component $K_\kappa$ of $K$ has the form 
$K_{\kappa,0}*K_{\kappa,1}*K_{\kappa,2}$. Let $S_\kappa$ be the generic
non-reflecting stationary subset of $\kappa^+\cap\cff(\gamma)$ added by
$K_{\kappa,0}$ over $M'[G]$. Since
$\bigcup K_{\kappa,2}\in M'[K]\subseteq\mbfV[G,G',L]$ is a closed unbounded
subset of $\kappa^+$ disjoint from $S_\kappa$, the set $S_\kappa$ is
non-stationary in $\mbfV[G,G',L]$. 

By elementarity, the generic filter $K_{\kappa,1}$ codes the set $S_\kappa$
inside $M'[K]$ as follows. Given an ordinal $\xi\in\kappa^+\cap\cff(\gamma)$, 
\[
\xi\in S_\kappa
\;\Longleftrightarrow\;
\mbox{$T'_{\kappa,2\xi+1}$ is stationary and $T'_{\kappa,2\xi}$ is
non-stationary.} 
\]
By the agreement $T'_{\kappa,\xi}=T_{\kappa,\xi}$ coming from \hyperlink{(D)}{(D)} and
mentioned above,
\[
\xi\in S_\kappa
\;\Longleftrightarrow\;
\mbox{$T_{\kappa,2\xi+1}$ is stationary and $T_{\kappa,2\xi}$ is
non-stationary} 
\]
for all such $\xi$. Recall that $S$ is the subset of $\kappa^+$ with
characteristic function $\bigcup G'_0$, and the generic filter $G'_1$ codes $S$
in $\mbfV[G,G']$ the same way as the generic filter $K_{\kappa,1}$ codes the set
$S_\kappa$ inside $M'[K]$, that is,
\[
\xi\in S
\;\Longleftrightarrow\;
\mbox{$T_{\kappa,2\xi+1}$ is stationary and $T_{\kappa,2\xi}$ is
non-stationary.} 
\]
whenever $\xi<\kappa^+$. 
It follows that for every $\xi\in\kappa^+\cap\cff(\gamma)$,
\begin{eqnarray*}
\xi\in S_\kappa
& \Longrightarrow &
\mbox{$T_{\kappa,2\xi}$ is non-stationary in $M'[K]$}\\
& \Longrightarrow &
\mbox{$T_{\kappa,2\xi}$ is non-stationary in $\mbfV[G,G',L]$}\\
& \Longrightarrow &
\mbox{$T_{\kappa,2\xi}$ is non-stationary and $T_{\kappa,2\xi+1}$ is
stationary in $\mbfV[G,G',L]$}\\
& \Longrightarrow &
\xi\in S
\end{eqnarray*}
Here the third implication follows from the fact that in $\mbfV[G,G']$, if
$\xi<\kappa^+$ then exactly one of $T_{\kappa,2\xi}$, $T_{\kappa,2\xi+1}$ is
stationary. As $\mbbP_\mcI$ is $\kappa^+$-c.c., for each $\xi<\kappa^+$
exactly one of $T_{\kappa,2\xi}$, $T_{\kappa,2\xi+1}$ is stationary in
$\mbfV[G,G',L]$, namely the one which is stationary in
$\mbfV[G,G']$. Similarly we verify the implication $\xi\notin
S_\kappa\;\Longrightarrow\;\xi\notin S$ whenever
$\xi\in\kappa^+\cap\cff(\gamma)$. Altogether we then  conclude that
$S_\kappa=S\cap\cff(\gamma)$. But then, by 
{Proposition~\ref{Spreserved},} $S_\kappa$ is stationary in
$\mbfV[G,G']$. Then, again by the $\kappa^+$-c.c. of $\mbbP_\mcI$, $S_\kappa$
remains stationary in $\mbfV[G,G',L]$, a contradiction. 
\end{proof}

{Finally we give a proof of incompatibility of strategies $\mcS_\gamma$
from Corollary~\ref{c6}(a), as formulated at the end of Corollary~\ref{c6}.} 
\medskip

The point here is that in the construction of
$\mcJ_\gamma$, the ordinal $\gamma$ at the $\kappa$-th stage in
$j(\mbbP_\kappa)$ is chosen before the generic filter $H$ comes into
play. Therefore the set $x_\gamma$ defined by  
\[
x_\gamma=\{\alpha<\kappa\mid\gamma^p(\alpha)=\gamma
\mbox{ for some/all $p\in G$ with $\alpha\in\dom(p)$}\}
\quad \mbox{ if } \gamma<\kappa 
\]
and 
\[
x_\gamma=\{\alpha<\kappa\mid\gamma^p(\alpha)=\alpha
\mbox{ for some/all $p\in G$ with $\alpha\in\dom(p)$}\}
\quad \mbox{ if } \gamma=\kappa 
\]
is an element of $\dot{U}^H$ for all
$(\csp(\kappa^+\smallsetminus S_\kappa),\mbfV[G,G'])$-generic filters $H$,
hence $x_\gamma$ is in the filter dual to $\mcJ_\gamma$. Now if Player~I plays
$\mcA_0$ such 
that $x_\gamma,x_{\gamma'}\in\mcA_0$ and Player~II responds with $U_0$
according to $\mcS_\gamma$ then $x_\gamma\in U_0$, as $U_0=W\cap\mcA_0$ for
some $(\mbbP_{\mcJ_\gamma},\mbfV[G,G'])$-generic filter $W$. Similarly as
above, $x_{\gamma'}\in U'_0$ for the response $U'_0$ of $\mcS_{\gamma'}$ to
$\langle\mcA_0\rangle$. Since $x_\gamma\cap x_{\gamma'}=\varnothing$, we have
$U_0\neq U'_0$. 
\end{proof}

\begin{remark} We could do the construction without the  ``lottery" aspect, aiming at a single $\gamma$.  Indeed that works for that $\gamma$, but leaves open the problem of whether ideals exist with dense trees of height $\gamma'$ for $\gamma'\ne \gamma$ and for which $\gamma'$ strategies exist in the Welch game.  These questions are thorny and are left to the second part of this paper.  The solutions there use extensive fine structural arguments. 
\end{remark}


\section{Open Problems}\label{s.open}
\noindent 
In this section we raise questions we don't know the answer to. We do not guarantee any of these 
questions are deep, difficult or even make sense.

\medskip
\begin{open}\bfni{Removing Hypotheses} Theorem \ref{t2} requires the GCH and the non-existence of saturated ideals on $\kappa$.  Are either of these hypotheses necessary?  Can some variant of the proof work without those hypotheses?
\medskip
\end{open}

\begin{open}
\bfni{What can be said about correspondence between ideals and strategies? }
\noindent Theorem \ref{t4} says that starting with a nice ideal $\mcJ_\gamma$ one can build a winning strategy 
$\mcS^*_\gamma$ for Player II in $\mcG_\gamma$. 
In turn, $\mcS_\gamma^*$ can used to build the ideal 
$\mci_\gamma$ with the methods in Theorems \ref{t1} and \ref{t2}:

\[ \mcJ_\gamma\Longrightarrow \mcS^*_\gamma\Longrightarrow \mci_\gamma\]
Inspection of the proof shows that  $\mcJ_\gamma\subseteq\mci_\gamma$. Is there anything else one can say?  For example, are the two ideals equal?
\end{open}

 \medskip

\noindent{\bf An Ulam Game}
Consider the following variant of the cut-and-choose game of length $\omega$
derived from games introduced by Ulam in \cite{Ulam64} (see
\cite{kanomag}).\footnote{Velickovic \cite{Velo} calls these Mycielski games}  

\begin{center}
\begin{tabular}{ c||c|c|c|c|c|c } 
 I & $(A^0_0,  A^0_1)$& ($A^1_0,  A^1_1)$ &\dots &$(A^n_0,
 A^n_1)$&$(A^{n+1}_0,  A^{n+1}_1)$&\dots \\ 
 \hline
II & $B_0$ & $B_1$& \dots &$B_n$&$B_{n+1}$& \dots \\ 
\end{tabular}
\end{center}
At stage $0$, Player~I plays a partition $(A^0_0, A^0_1)$ of $\kappa$. At stage
$n\ge 0$ Player~II lets $B_n$ be either $A^n_0$ or $A^n_1$, and plays $B_n$.
At stage $n\ge 1$ Player~I plays a partition $(A^{n+1}_0, A^{n+1}_1)$ of
$B_n$.  The winning condition for Player~II is that $|\bigcap_{n\in \omega}
B_n|\ge 2$. 

These games generalize to lengths $\gamma>\omega$ as follows:
\begin{enumerate}
	\item At successor stages $\alpha+1$, Player $I$ partitions $B_\alpha$ into two pieces and Player $II$ chooses 
	one of the pieces. 
	\item At limit stages $\alpha$, let $B_\alpha=\bigcap_{\beta<\alpha}B_\beta$ and then Player $I$ partitions $B_\alpha$ into two pieces, and Player $II$ chooses one of the pieces.
	\item The winning condition is the same: the intersection of the pieces that player $II$ chooses has to have at least two elements.
\end{enumerate}

\begin{description}
\item[Observation] If Player~II has a winning strategy in the game 
$\mathcal
G^*_\omega$, then Player~II has a winning strategy in the Ulam game. 
\end{description}
This is immediate: Player~II follows her strategy in an auxiliary play of the game 
$\mathcal G^*_\omega$
against the Boolean  Algebras $\mca_n$ generated by $\{A^i_0, A^i_1:i\le n\}$.
In the game $\mathcal G^*_\omega$ she then plays as $B_n$ whichever of $A^n_0$
or $A^n_1$ belongs to $U_n$.  By the winning condition on $G^*_\omega$,
$\bigcap_{n\in\omega} B_n$ belongs to a $\kappa$-complete, uniform
filter. Hence $|\bigcap_nB_n|=\kappa>1$. 
 
Silver and Solovay (see \cite{kanomag}, page 249) showed that if Player~II
wins the Ulam game, then there is an inner model with a measurable
cardinal. This provides an alternate proof that the  consistency strength of
the statement ``Player~II has a winning strategy in $\mcG^*_\omega$" is
that of a measurable cardinal. 
 
What is unclear is the exact relationship between the Ulam Game and the Welch Game. Laver showed that if a measurable 
cardinal is collapsed to  $\omega_2$ by the L\'evy collapse  and $\mathcal I$ is the ideal generated by the original 
normal measure on $\kappa$, then in the extension $\ptm(\omega_2)/\mci$  has a dense countably closed subset 
(\cite{quotalg}). He showed that it follows from this that Player~II has a winning strategy in the Ulam game.  

In Section \ref{weak compactness}, it is shown that the Welch games only make sense at regular 
cardinals $\kappa$ such that for all $\gamma<\kappa$, $2^\gamma\le \kappa$. At successor 
cardinals $\kappa$ there is a single play by Player I (the algebra in part (2) of 
Theorem ~\ref{showingprincipals}) that defeats Player II in the game of length 1.  Moreover at non-weakly compact inaccessible cardinals $\kappa$, the  Keisler-Tarski Theorem shows player I has a winning strategy in the game of length 1.  But if $\kappa$ is weakly compact, Player II has a winning strategy in the game of length $\omega$.

The upshot of this discussion is that a comparison between the Ulam games and the Welch games should occur at weakly compact cardinals.
 
\begin{open}\label{Ulam}
\noindent {Suppose that $\kappa$ is weakly compact and}
that Player~II has a wining strategy in the Ulam game of length
$\gamma$ (for $\gamma\ge \omega$), does Player~II have a winning strategy in
$\mathcal G^*_\gamma$?  
\end{open}

\bfni{Determinacy of the Welch Games}  The discussion in the paragraphs before Problem \ref{Ulam} (based on Section \ref{weak compactness} of this paper) shows that questions about the determinacy of Welch Games really only make sense at inaccessible cardinals.  Moreover at non-weakly-compact inaccessible cardinals Player I wins the game of length 1 and at weakly compact cardinals Player II wins the game of length $\omega$. By work of Nielsen and Welch if  II has a winning strategy in the game of length $\omega_1$, then there is an inner model with a measurable cardinal--so Player II can't have such a winning strategy in $L$. (Theorem \ref{t1} in this paper also gives this result.)
Welch showed that for all regular $\gamma$, $\mcG^W_\gamma$ is determined in $L$ (this also follows immediately from Theorem 5.6 in \cite{HolySh}).


%
 However the following seems to be an open problem: 

\begin{open}
Is there a model of $\zfc + \gch$ with a measurable cardinal where the Welch games are determined? With a supercompact cardinal?  
\end{open}

\medskip
\bfni{Welch Games on Larger cardinals} In this paper the Welch games are shown to provide intermediary properties between weakly compact cardinals and measurable cardinals. What is the analogue for cardinals that are at least measurable?  Perhaps the most interesting question is the following:

\begin{open}{Are  there $\ptm_\kappa(\lambda)$ versions of the game?}
\end{open}

It is not trivial to even formulate a reasonable analogue of Welch games on supercompact cardinals.  The classical ultrafilter extension properties  on $\ptm_\kappa(\lambda)$ that follow from  large cardinals  suggest one, but it is not clear how to proceed.

Another technical obstacle that would have to be overcome is the following:  in the proofs in this paper one 
passes from a $\kappa$-filter $U$ on an $N_\alpha$ to its normal derivative $U^*$.  Normality presents an obstacle for   $\ptm_\kappa(\lambda)$  because this is the crucial difference between supercompact and strongly compact cardinals.

In \cite{BD} Buhagiar and Dzamonja found analogies of strongly compact cardinals that  Dzamonja suggested might be candidates for this game. 

\medskip
\bfni{Extender Algebras}  Large cardinals whose embeddings are determined by Extender Algebras also form  
candidates for places games like this can be played.  If $E$ is an extender with generators $\lambda^{<\omega}$ one might consider games 
where Player I plays elements of $\lambda^{<\omega}$ and sequences of $\kappa$-algebras in a coherent way, and 
player II plays ultrafilters on the associated algebras. 

In this manner one might hope to extend these results to $\ptm^2(\kappa)$ or further.

\medskip
\bfni{Games on accessible cardinals}
\begin{open}
Are there small cardinal versions of these games?
\end{open} 
The results in Section \ref{weak compactness} limit the Welch games to inaccessible cardinals.  However one might hope that there is some version of these games that end up creating ideals on cardinals that are not weakly compact.  A random suggestion is to require Player II to play ideals with some combinatorial property at each stage (rather than ultrafilters).
One target would be to define a game  similar to the Welch games that  gives
 $\omega$-closed densely treed ideals on $\omega_2$ (the original Laver ideals).

\bibliography{games}

\bibliographystyle{plain}

\end{document}